\documentclass[11pt]{amsart}
\usepackage{amsmath}
\usepackage{amssymb}
\usepackage{amsthm}
\usepackage{latexsym}
\usepackage{hyperref}
\usepackage{enumerate}
\usepackage{graphicx}
\usepackage[all]{xy}

\newtheorem{thm}{Theorem}[section]

\newtheorem{lem}[thm]{Lemma}
\newtheorem{prop}[thm]{Proposition}
\newtheorem{conj}[thm]{Conjecture}
\newtheorem{thmintro}{Theorem}
\newtheorem{conjintro}[thmintro]{Conjecture}

\theoremstyle{definition}
\newtheorem{defn}[thm]{Definition}
\newtheorem{rem}[thm]{Remark}

\newtheorem{ex}[thm]{Example}

\providecommand{\norm}[1]{\left\| #1 \right\|}
\newcommand{\enuma}[1]{\begin{enumerate}[\textup{(}a\textup{)}] {#1} \end{enumerate}}

\newcommand{\mb}{\mathbf}
\newcommand{\bW}{{\mathbf W}}
\newcommand{\bI}{{\mathbf I}}
\newcommand{\mh}{\mathbb}
\newcommand{\mr}{\mathrm}
\newcommand{\mc}{\mathcal}
\newcommand{\mf}{\mathfrak}

\newcommand{\isom}{\xrightarrow{\;\sim\;}}

\newcommand{\N}{\mathbb N}
\newcommand{\Z}{\mathbb Z}
\newcommand{\Q}{\mathbb Q}
\newcommand{\R}{\mathbb R}
\newcommand{\C}{\mathbb C}
\newcommand{\ep}{\epsilon}

\newcommand{\matje}[4]{\left(\begin{smallmatrix} #1 & #2 \\ 
#3 & #4 \end{smallmatrix}\right)}

\newcommand{\q}{/\!/}

\def\Hom{{\rm Hom}}
\def\End{{\rm End}}

\def\Irr{{\rm Irr}}
\def\Gal{{\rm Gal}}
\def\Jord{{\rm Jord}}
\def\SO{{\rm SO}}
\def\O{{\rm O}}
\def\rS{{\rm S}}
\def\Sp{{\rm Sp}}
\def\GL{{\rm GL}}

\def\PGL{{\rm PGL}}
\def\SL{{\rm SL}}

\def\St{{\rm St}}
\def\cC{{\mathcal C}}

\def\cS{{\mathcal S}}

\def\cN{{\mathcal N}}

\def\cL{{\mathcal L}}
\def\cH{{\mathcal H}}

\def\cO{{\mathcal O}}
\def\cE{{\mathcal E}}
\def\cF{{\mathcal F}}
\def\cR{{\mathfrak R}}
\def\fP{{\mathfrak P}}
\def\cZ{{\mathcal Z}}
\def\Fr{{\rm Frob}}

\def\reg{{\rm reg}}

\def\ind{{\rm ind}}

\def\nr{{\rm nr}}

\def\rU{{\rm U}}

\def\fs{{\mathfrak s}}
\def\ft{{\mathfrak t}}

\def\Rep{{\rm Rep}}

\def\Res{{\rm Res}}
\def\Stab{{\rm Stab}}

\def\bdd{{\rm bdd}}
\def\der{{\rm der}}
\def\ad{{\rm ad}}
\def\sc{{\rm sc}}
\def\Aut{{\rm Aut}}
\def\LLC{{\rm LLC}}

\def\Cent{{\rm Z}}
\def\cusp{{\rm cusp}}
\def\uni{{\rm un}}

\def\bdd{{\rm bdd}}
\def\fB{{\mathfrak B}}
\def\Omega{{\fB}}
\def\IC{{\rm{IC}}}
\def\restriction#1#2{\mathchoice
              {\setbox1\hbox{${\displaystyle #1}_{\scriptstyle #2}$}
              \restrictionaux{#1}{#2}}
              {\setbox1\hbox{${\textstyle #1}_{\scriptstyle #2}$}
              \restrictionaux{#1}{#2}}
              {\setbox1\hbox{${\scriptstyle #1}_{\scriptscriptstyle #2}$}
              \restrictionaux{#1}{#2}}
              {\setbox1\hbox{${\scriptscriptstyle #1}_{\scriptscriptstyle #2}$}
              \restrictionaux{#1}{#2}}}
\def\restrictionaux#1#2{{#1\,\smash{\vrule height .8\ht1 depth .85\dp1}}_{\,#2}}

\begin{document}

\title{Generalizations of the Springer correspondence and cuspidal Langlands parameters}

\author[A.-M. Aubert]{Anne-Marie Aubert}
\address{Institut de Math\'ematiques de Jussieu -- Paris Rive Gauche, 
U.M.R. 7586 du C.N.R.S., U.P.M.C., 4 place Jussieu 75005 Paris, France}
\email{anne-marie.aubert@imj-prg.fr}
\author[A. Moussaoui]{Ahmed Moussaoui}
\address{Department of Mathematics and Statistics, University of Calgary, 
2500 University Drive NW, Calgary, Alberta, Canada}
\email{ahmed.moussaoui@ucalgary.ca}
\author[M. Solleveld]{Maarten Solleveld}
\address{IMAPP, Radboud Universiteit Nijmegen, Heyendaalseweg 135, 
6525AJ Nijmegen, the Netherlands}
\email{m.solleveld@science.ru.nl}
\date{\today}
\thanks{The second author gratefully acknowledges support from the Pacific Institute for the Mathematical Sciences (PIMS). The third author is supported by a NWO Vidi-grant, No. 639.032.528.}
\maketitle

\begin{abstract}
Let $\cH$ be any reductive $p$-adic group. We introduce a notion of cuspidality for 
enhanced Langlands parameters for $\cH$, which conjecturally puts supercuspidal 
$\cH$-representations in bijection with such L-parameters. We also define a 
cuspidal support map and Bernstein components for enhanced L-parameters, in
analogy with Bernstein's theory of representations of $p$-adic groups. We check
that for several well-known reductive groups these analogies are actually precise. 

Furthermore we reveal a new structure in the space of enhanced L-parameters for
$\cH$, that of a disjoint union of twisted extended quotients. This is an analogue
of the ABPS conjecture (about irreducible $\cH$-representations) on the Galois side
of the local Langlands correspondence. Only, on the Galois side it is no longer conjectural.
These results will be useful to reduce the problem of finding a local Langlands 
correspondence for $\cH$-representations to the corresponding problem for 
supercuspidal representations of Levi subgroups of $\cH$.

The main machinery behind this comes from perverse sheaves on algebraic groups.
We extend Lusztig's generalized Springer correspondence to disconnected complex
reductive groups $G$. It provides a bijection between, on the one hand, pairs
consisting of a unipotent element $u$ in $G$ and an irreducible representation of the
component group of the centralizer of $u$ in $G$, and, on the other hand, irreducible 
representations of a set of twisted group algebras of certain finite groups.
Each of these twisted group algebras contains the group algebra of a Weyl group, which
comes from the neutral component of $G$.

In 2025 an erratum was added, to repair Theorem 3.1.a.
\end{abstract}

\tableofcontents

\section*{Introduction}

As the title suggests, this paper consists of two parts. The first part is purely in
complex algebraic geometry, and is accessible without any knowledge of the Langlands
program or $p$-adic groups. We start with discussing the second part though,
which is an application of and a motivation for the first part.

The local Langlands correspondence (LLC) predicts a relation between two rather 
different kinds of objects: on the one hand irreducible representations of 
reductive groups over a local field $F$, on the other hand some sort of 
representations of the Weil--Deligne group of $F$. According to the original setup 
\cite{Bor,Lan1}, it should be possible to associate to every L-parameter a finite 
packet of irreducible admissible representations. Later this was improved by 
enhancing L-parameters \cite{Lus0,KaLu}, and the modern interpretation \cite{ABPS7,Vog} 
says that the LLC should be a bijection (when formulated appropriately).

We consider only non-archimedean local fields $F$, and we speak of the Galois
side versus the $p$-adic side of the LLC. The conjectural bijectivity makes it
possible to transfer many notions and ideas from on side of the LLC to the other. 
Indeed, a main goal of this paper is to introduce an analogue, on the Galois side,
of the Bernstein theory \cite{BeDe} for smooth representations of reductive 
$p$-adic groups. 

Bernstein's starting point is the notion of a supercuspidal representation. For
a long time it has been unclear how to translate this to the Galois side.
In \cite[Def. 4.12]{Mou} the second author discovered the (probably) correct notion for
split reductive $p$-adic groups, which we generalize here.

For maximal generality, we adhere to the setup for L-parameters from  \cite{Art2}.
Let $\mb W_F$ be the Weil group of $F$, let $\cH$ be a connected reductive group 
over $F$ and let ${}^L \cH = \cH^\vee \rtimes \mb W_F$ be its dual L-group.
Let $\cH^\vee_\ad$ be the adjoint group of $\cH^\vee$, and let $\cH^\vee_\sc$ be the 
simply connected cover of the derived group of $\cH^\vee_\ad$. 
Let $\phi : \mb W_F \times \SL_2 (\C) \to {}^L \cH$ be an L-parameter, let 
$Z_{\cH^\vee_\ad}(\phi (\mb W_F))$ be the centralizer of $\phi (\mb W_F)$ in 
$\cH^\vee_\ad$ and let 
\[
G = Z^1_{\cH^\vee_\sc}(\phi (\mb W_F))
\]
be its inverse image in $\cH^\vee_\sc$. To $\phi$ we associate the finite group 
$\cS_\phi := \pi_0 (Z^1_{\cH^\vee_\sc}(\phi))$, where $Z^1$ is again defined via 
$\cH^\vee_\ad$. We call any irreducible representation 
of $\cS_\phi$ an \emph{enhancement} of $\phi$ . The group $\cS_\phi$ coincides 
with the group considered by both Arthur in \cite{Art2} and Kaletha in \cite[\S 4.6]{Kal}.
A remarkable fact is that the group $\cS_\phi$ is isomorphic to the group 
$A_{G}(u_\phi) := \pi_0 (Z_G (u_\phi))$, where $u_\phi:= \phi (1,\matje{1}{1}{0}{1})$.

We propose (see Definition \ref{def:7.1}) to call an enhanced L-parameter
$(\phi,\rho)$ for $\cH$ \emph{cuspidal} if $u_\phi$ and $\rho$,
considered as data for the complex reductive group $G$,
form a cuspidal pair. By definition this means that the restriction of $\rho$ from
$A_G (u) = \pi_0 (Z_G (u))$ to $A_{G^\circ}(u)$ is a direct sum of cuspidal
representations in Lusztig's sense \cite{Lus1}. Intuitively, it says that $\rho$
or $\rho |_{A_{G^\circ} (u)}$ cannot be obtained (via an appropriate notion of 
parabolic induction) from any pair $(u',\rho')$ that can arise from a proper Levi
subgroup of $G^\circ$.
We emphasize that it is essential 
to use L-parameters enhanced with a representation of a suitable component group, 
for cuspidality cannot be detected from the L-parameter alone.

Let $\Irr_\cusp (\cH)$ be the set of supercuspidal $\cH$-representations (up to
isomorphism) and let $\Phi_\cusp (\cH)$ be the set of $\cH^\vee$-conjugacy 
classes of cuspidal L-parameters for $\cH$. It is known that in many cases
such cuspidal L-parameters do indeed parametrize supercuspidal representations, and
that moreover there is a nice bijection $\Irr_\cusp (\cH) \to \Phi_\cusp (\cH)$.

We call the enhanced L-parameters $(\phi,\rho)$ such that $\phi$ restricts trivially 
to the inertia group $\bI_F$ \emph{unipotent}. 
A representation $\pi$ of $\cH$ is said to be unipotent (or sometimes to have unipotent 
reduction) if for some parahoric subgroup $\fP$ of $\cH$, and the inflation $\sigma$ to $\fP$ 
of some unipotent cuspidal representation of its reductive quotient, the space 
$\Hom_\fP (\sigma,\pi)$ is nonzero.  
When $\cH$ is simple of adjoint type, Lusztig has proved in \cite{Lus2,Lus3} that the 
$\cH^\vee$-conjugacy classes of unipotent enhanced L-parameters are in bijection with the 
equivalence classes of unipotent irreducible representations of $\cH$. Under 
Lusztig's bijection the unipotent cuspidal enhanced L-parameters correspond to the 
unipotent supercuspidal irreducible representations of $\cH$.

When $(\phi,\rho)$ is a unipotent enhanced L-parameter, the group $G$ coincides 
with the group $Z^1_{\cH^\vee_\sc}(\phi (\Fr))$. In contrast, for $(\phi,\rho)$ arbitrary, 
the group $G$ is usually a proper subgroup of $Z^1_{\cH^\vee_\sc}(\phi (\Fr))$. 

Based on the notion of cuspidality that we have defined, we construct a 
\emph{cuspidal support map} for L-parameters 
(Definition \ref{def:7.8}). It assigns to every enhanced L-parameter for $\cH$ a
Levi subgroup $\cL \subset \cH$ and a cuspidal L-parameter for $\cL$, unique up 
to conjugation. We conjecture that this map is a precise analogue of Bernstein's
cuspidal support map for irreducible $\cH$-representations, in the sense that 
these cuspidal support maps commute with the respective local Langlands 
correspondences (assuming that these exist of course). 

A result for $p$-adic groups which has already been transferred to the Galois
side is the Langlands classification \cite{SiZi}. On the $p$-adic side it
reduces $\Irr (\cH)$ to the tempered duals of Levi subgroups of $\cH$, while
on the Galois side it reduces general (enhanced) L-parameters to bounded L-parameters
for Levi subgroups. We show (Lemma \ref{lem:8.3}) that our cuspidal support map
factors through through the Langlands classification on the Galois side, just 
like Bernstein's cuspidal support map on the $p$-adic side.\\

Recall that a crucial role in the Bernstein decomposition is played by inertial
equivalence classes of (super)cuspidal pairs for $\cH$. These consist of a Levi subgroup
$\cL \subset \cH$ and a supercuspidal representation thereof, up to equivalence by
$\cH$-conjugation and twists by unramified characters. Since the LLC for unramified
characters is known, we can easily translate this to a notion of inertial equivalence
classes of enhanced L-parameters (Definition \ref{def:8.4}). Using the cuspidal
support map, we can also partition the set of enhanced L-parameters $\Phi_e (\cH)$
into countably many Bernstein components $\Phi_e (\cH)^{\fs^\vee}$, parametrized
by the inertial equivalence classes $\fs^\vee$, see \eqref{eq:8.4}. 

Let $\cL \subset \cH$ be a Levi subgroup, and let 
\begin{equation*}
W(\cH,\cL) = N_{\cH}(\cL) / \cL
\end{equation*}
be its "Weyl" group. In \cite{ABPS7} it was shown to be naturally isomorphic to\\
$N_{\cH^\vee}(\cL^\vee \rtimes \mb W_F) / \cL^\vee$, so it acts on both 
$\Irr_\cusp (\cL)$ and $\Phi_\cusp (\cL)$.

Our main result provides a complete description of the space of enhanced
L-parameters $\Phi_e (\cH)$ in terms of cuspidal L-parameters for Levi subgroups,
and the associated Weyl groups. It discovers a new structure in $\Phi_e (\cH)$,
that of a union of extended quotients. It improves on both the Langlands 
classification and the theory of the Bernstein centre (on the Galois side of
the LLC).

Fix a character $\zeta_\cH$ of $Z(\cH^\vee_\sc)$ whose restriction to 
$Z(\cH^\vee_\sc)^{\mb W_F}$ corresponds via the Kottwitz isomorphism to the class 
of $\cH$ as an inner twist of its quasi-split inner form.
We indicate the subset of enhanced L-parameters 
$(\phi,\rho)$ such that $\rho$ extends $\zeta_\cH$ with a subscript $\zeta_\cH$.  
This $\zeta_\cH$ only plays a role when $Z(\cH^\vee_\sc)$ is not fixed by $\mb W_F$,
in particular it is redundant for inner twists of split groups.

\begin{thmintro}\label{thm:A} (See Theorem \ref{thm:10.3}) \\
Let $\mf{Lev}(\cH)$ be a set of representatives for the conjugacy classes of
Levi subgroups of $\cH$. There exists a bijection
\[
\Phi_{e,\zeta_\cH} (\cH) \; \longleftrightarrow \; \bigsqcup\nolimits_{\cL \in 
\mf{Lev}(\cH)} \big( \Phi_{\cusp,\zeta_\cH} (\cL) \q W(\cH,\cL) \big)_\kappa . 
\]
\end{thmintro}

Here $( \cdot \q \cdot )_\kappa$ denotes a twisted extended quotient, as defined 
in \eqref{eq:extquot}. The bijection is not entirely canonical, but we provide
a sharp bound on the non-canonicity. We note that the bijection is not based on
the earlier cuspidal support map, but rather on a modification thereof, which
preserves boundedness of L-parameters.

We expect that Theorem \ref{thm:A} will turn out to be an analogue of the ABPS 
conjecture \cite{ABPS7} on the Galois side of the LLC. To phrase this precisely in 
general, we need yet another ingredient. 

\begin{conjintro}\label{conj:B}
Let $\cH$ be a connected reductive group over a local non-archimedean field, and 
let $\Irr(\cH)$ denote the set of its irreducible smooth representations.
There exists a commutative bijective diagram
\[
\xymatrix{
\Irr (\cH) \ar@{<->}[r] \ar@{<->}[d] & \Phi_{e,\zeta_\cH} (\cH) \ar@{<->}[d] \\
\bigsqcup_{\cL \in \mf{Lev}(\cH)} \big( \Irr_\cusp (\cL) \q W(\cH,\cL) 
\big)_\kappa \ar@{<->}[r] & \bigsqcup_{\cL \in \mf{Lev}(\cH)} 
\big( \Phi_{\cusp,\zeta_\cH} (\cL) \q W(\cH,\cL) \big)_\kappa
}
\]
with the following maps:
\begin{itemize}
\item The right hand side is Theorem \ref{thm:A}.
\item The upper horizontal map is a local Langlands correspondence for $\cH$.
\item The lower horizontal map is obtained from local Langlands correspondences
for $\Irr_\cusp (\cL)$ by applying $( \cdot \q  W(\cH,\cL) )_\kappa$.
\item The left hand side is the bijection in the ABPS conjecture \cite[\S 2]{ABPS7}.
\end{itemize}
\end{conjintro}

With this conjecture one can reduce the problem of finding a LLC for $\cH$ to that
of finding local Langlands correspondences for supercuspidal representations of its
Levi subgroups. Conjecture \ref{conj:B} is currently known in the following cases:
\begin{itemize}
\item inner forms of $\GL_n (F)$ \cite[Theorem 5.3]{ABPS6},
\item inner forms of $\SL_n (F)$ \cite[Theorem 5.6]{ABPS6},
\item split classical groups \cite[\S 5.3]{Mou},
\item principal series representations of split groups \cite[\S 16]{ABPS5}.
\end{itemize}
In \cite{AMS}, we extensively use the results of the present paper in order to construct
(twisted) graded Hecke algebras $\mh H$ based on a (possibly disconnected) complex reductive 
group $G$ and a cuspidal local system $\cL$ on a unipotent orbit of a Levi subgroup $M$ of $G$, 
and to develop their representation theory. The algebras $\mh H$ generalize the graded Hecke 
algebras defined and investigated by Lusztig for connected $G$.\\[2mm]

Now we come to the main technique behind the above: generalizations of the Springer
correspondence. Let $G^\circ$ be a connected complex reductive group with a maximal
torus $T$ and Weyl group $W(G^\circ,T)$. Recall that the original Springer 
correspondence \cite{Spr} is a bijection between the irreducible representations
of $W(G^\circ,T)$ and $G^\circ$-conjugacy classes of pairs $(u,\eta)$, where 
$u \in G^\circ$ is unipotent and $\eta$ is an irreducible representation of 
$A_{G^\circ}(u) = \pi_0 (Z_{G^\circ}(u))$ which appears in the homology of the 
variety of Borel subgroups of $G^\circ$ containing $u$. 

Lusztig \cite{Lus1} generalized this to a setup which includes all pairs $(u,\eta)$
with $u \in G^\circ$ unipotent and $\eta \in \Irr (A_{G^\circ}(u))$. On the other 
side of the correspondence he replaced $\Irr (W (G^\circ,T))$ by a disjoint union 
$\sqcup_{\ft^\circ} \Irr (W_{\ft^\circ})$, where $\ft^\circ = [L,v,\epsilon]_{G^\circ}$ 
runs through cuspidal pairs $(v,\epsilon)$ for Levi subgroups $L$ of $G^\circ$, and 
$W_{\ft^\circ} = W(G^\circ,L)$ is the Weyl group associated to $\ft^\circ$.

More precisely, Lusztig first attaches to $(u,\eta)$ a cuspidal support $\ft^\circ =
\Psi_{G^\circ}(u,\eta)$, and then he constructs a bijection $\Sigma_{\ft^\circ}$ 
between $\Psi_{G^\circ}^{-1}(\ft^\circ)$ and $\Irr (W_{\ft^\circ})$. In Section 
\ref{sec:genSpr} we recall these constructions in more detail, and we prove:

\begin{thmintro}\label{thm:C}
The maps $\Psi_{G^\circ}$ and $\Sigma_{\ft^\circ}$ are equivariant with respect 
to algebraic automorphisms of the group $G^\circ$.
\end{thmintro}

Given a Langlands parameter $\phi$ for $\cH$, we would like to apply this machinery 
to $G = Z^1_{\cH^\vee_\sc}(\phi (\mb W_F))$. However, we immediately run into the 
problem that this complex reductive group is usually not connected. Thus we need 
a generalization of Lusztig's correspondence to disconnected reductive groups. 
Although there exist generalizations of the Springer correspondence in various
directions \cite{AcHe,AHJR,AcSa,Lus1,Lus4,LS,Sor}, this particular issue has not 
yet been addressed in the literature.

We would like to have a version which transforms every pair $(u,\eta)$ for $G$
into an irreducible representation of some Weyl group. But this
turns out to be impossible! The problem is illustrated by Example \ref{ex:A}:
we have to use twisted group algebras of groups $W_\ft$ which are not necessarily
Weyl groups.

When $G$ is disconnected, we define the cuspidal support map by
\[
\Psi_G (u,\eta) = \Psi_{G^\circ}(u,\eta^\circ) / G\text{-conjugacy} , 
\]
where $\eta^\circ$ is any constituent of $\Res_{A_{G^\circ}(u)}^{A_G (u)} \eta$.
This is well-defined by the Ad$(G)$-equivariance of $\Psi_{G^\circ}$ from 
Theorem \ref{thm:C}.

For a cuspidal support $\ft = [L,v,\epsilon]_G$ (where $L$ is a Levi subgroup of
$G^\circ$), we put 
\[
W_\ft = N_G (L,v,\epsilon) / L \quad \text{and} \quad 
\ft^\circ = [L,v,\epsilon]_{G^\circ}. 
\]
Then $W_\ft$ contains $W_{\ft^\circ} = W(G^\circ,L)$ as a normal subgroup.

\begin{thmintro}\label{thm:D} 
(See Theorem \ref{thm:5.5} and Proposition \ref{prop:5.3}) \\
Let $\ft = [L,v,\epsilon]_G$ be a cuspidal support for $G$. There exist:
\begin{itemize}
\item a 2-cocycle $\natural_\ft\colon  W_\ft / W_{\ft^\circ} \times  
W_\ft / W_{\ft^\circ} \to \C^\times$,
\item a twisted group algebra $\C [W_\ft,\natural_\ft]$,
\item a bijection $\Psi_G^{-1}(\ft) \to \Irr (\C [W_\ft,\natural_\ft])$
which extends \cite{Lus1}.
\end{itemize}
Moreover the composition of the bijection with 
$\Res^{\C [W_\ft,\natural_\ft]}_{\C [W_{\ft^\circ}]}$ is canonical.
\end{thmintro}

Of course the proof of Theorem \ref{thm:D} starts with Lusztig's generalized
Springer correspondence for $G^\circ$. Ultimately it involves a substantial
part of the techniques and objects from \cite{Lus1}, in particular we consider
similar varieties and sheafs. In Section \ref{sec:cusp} we provide an expression
for the 2-cocycle $\natural_\ft$, derived from the cuspidal case $L = G^\circ$.

Yet $\Psi_G$ and Theorem \ref{thm:D} still do not suffice for our plans with 
Langlands parameters. Namely, suppose that $(\phi,\rho)$ is an enhanced L-parameter
for $\cH$ and apply $\Psi_G$ with $G = Z^1_{\cH^\vee_\sc}(\phi (\mb W_F))$ and 
$(u,\eta) = \big( \phi (1,\matje{1}{1}{0}{1}), \rho \big)$. We end up with 
$\ft = [L,v,\epsilon]_G$, where $L$ is a Levi subgroup of $G^\circ$. But the
cuspidal support map for L-parameters should produce an enhanced L-parameter for
a Levi subgroup $\cL$ of $\cH$, and that would involve a possibly disconnected
group $Z^1_{\cL^\vee_\sc}(\phi (\mb W_F))$ instead of $L$.

To resolve this problem, we consider \emph{quasi-Levi} subgroups of $G$. These are
groups of the form $M = Z_G (Z(L)^\circ)$, where $L \subset G^\circ$ is a Levi
subgroup (and hence $M^\circ = L$). With these one can define a \emph{quasi-cuspidal} 
support, a triple $(M,v,q\epsilon)$ with $v \in M^\circ$ unipotent and 
$q \epsilon \in \Irr (A_M (v))$ such that $\Res^{A_M (v)}_{A_{M^\circ}(v)} q \epsilon$
is a sum of cuspidal representations. The cuspidal support map $\Psi_G$ can be
adjusted to a canonical quasi-cuspidal support map $q \Psi_G$, see \eqref{eq:6.1}.
It is this map that gives us the cuspidal support of enhanced L-parameters.

To a quasi-cuspidal support $q \ft = [M,v,q \epsilon]_G$ we associate the group
$W_{q \ft} = N_G (M,v,q\epsilon) / M$, which (again) contains
$W_{\ft^\circ} = N_{G^\circ}(M^\circ) / M^\circ$.

\begin{thmintro}\label{thm:E} (See Theorem \ref{thm:6.3} and Lemma \ref{lem:6.2}) \\
Theorem \ref{thm:D} also holds with quasi-Levi subgroups and with the quasi-cuspidal
support $q \ft$ instead of $\ft$. It gives a bijection
$q \Psi_G^{-1}(q \ft) \to \Irr (\C [W_{q \ft},\kappa_{q \ft}])$ 
which is canonical in the same degree as for $\ft$.
\end{thmintro}

The derivation of Theorem \ref{thm:E} from Theorem \ref{thm:D} relies to a large
extent on (elementary) results about twisted group algebras, which we put in
Section \ref{sec:twisted}. The bijection from Theorem \ref{thm:E} is extensively
used in Section \ref{sec:extquot}, for Theorem \ref{thm:A}.\\

\noindent \textbf{Acknowledgements.}
The authors thank Anthony Henderson for pointing out a mistake in an earlier version,
and the referee, for his extremely thorough and helpful report.

\section{Twisted group algebras and normal subgroups}
\label{sec:twisted}

Throughout this section $\Gamma$ is a finite group and $K$ is an algebraically 
closed field whose characteristic does not divide the order of $\Gamma$.
Suppose that $\natural : \Gamma \times \Gamma \to K^\times$ is a 2-cocycle, that is,
\begin{equation}\label{eq:1.1}
\natural (\gamma_1,\gamma_2 \gamma_3) \natural (\gamma_2,\gamma_3) =
\natural (\gamma_1,\gamma_2) \natural (\gamma_1 \gamma_2,\gamma_3)
\quad \forall \gamma_1,\gamma_2,\gamma_3 \in \Gamma .
\end{equation}
The $\natural$-twisted group algebra of $\Gamma$ is defined to be the $K$-vector 
space $K[\Gamma,\natural]$ with basis $\{ T_\gamma : \gamma \in \Gamma \}$ 
and multiplication rules
\begin{equation}\label{eq:1.2}
T_\gamma T_{\gamma'} = \natural (\gamma,\gamma') T_{\gamma \gamma'} \quad
\gamma, \gamma' \in \Gamma .
\end{equation}
Its representations can be considered as projective $\Gamma$-representations.
Schur showed (see \cite[Theorem 53.7]{CuRe}) that there exists a finite central
extension $\tilde \Gamma$ of $\Gamma$, such that 
\begin{itemize}
\item char$(K)$ does not divide $|\tilde \Gamma|$,
\item every irreducible projective $\Gamma$-representation over $K$ lifts to
an irreducible $K$-linear representation of $\tilde \Gamma$.
\end{itemize}
Then $K[\Gamma,\natural]$ is a direct summand of $K[\tilde \Gamma]$, namely the
image of a minimal idempotent in $K[\ker (\tilde \Gamma \to \Gamma)]$. The condition
on char$(K)$ ensures that $K[\tilde \Gamma]$ is semisimple, so $K[\Gamma,\natural]$
is also semisimple.

Let $N$ be a normal subgroup of $\Gamma$ and $(\pi,V_\pi)$ an irreducible
representation of $N$ over $K$. We abbreviate this to $\pi \in \Irr_K (N)$.
We want to analyse the set of irreducible $\Gamma$-representations whose 
restriction to $N$ contains $\pi$. 

More generally, suppose that $\natural$ is a 2-cocycle of $\Gamma / N$. We identify
it with a 2-cocycle $\Gamma \times \Gamma \to K^\times$ that factors through 
$(\Gamma / N)^2$. We also want to analyse the irreducible representations of
$K [\Gamma,\natural]$ that contain $\pi$.

For $\gamma \in \Gamma$ we define $\gamma \cdot \pi \in \Irr_K (N)$ by
\begin{equation}\label{eq:1.8}
(\gamma \cdot \pi) (n) = \pi (\gamma^{-1} n \gamma) . 
\end{equation}
This determines an action of $\Gamma$ and of $\Gamma / N$ on $\Irr (N)$. Let
$\Gamma_\pi$ be the isotropy group of $\pi$ in $\Gamma$. For every $\gamma \in 
\Gamma_\pi$ we choose a $I^\gamma = I^\gamma_\pi \in \Aut_K (V_\pi)$ such that
\begin{equation}\label{eq:1.3}
I^\gamma \circ \pi (\gamma^{-1} n \gamma) = \pi (n) \circ I^\gamma
\quad \forall n \in N .
\end{equation}
Thus $I^\gamma \in \Hom_N (\gamma \cdot \pi,\pi)$. Given another 
$\gamma' \in \Gamma$, we can regard $I^\gamma$ also as an element of 
$\Hom_N (\gamma' \cdot \gamma \cdot \pi, \gamma' \cdot \pi)$, and then it can be 
composed with $I^{\gamma'} \in \Hom_N (\gamma' \cdot \pi,\pi)$. By Schur's lemma
all these maps are unique up to scalars, so there exists a 
$\kappa_\pi (\gamma,\gamma') \in K^\times$ with
\begin{equation}\label{eq:1.4}
I^{\gamma \gamma'} = \kappa_\pi (\gamma,\gamma') I^\gamma \circ I^{\gamma'} . 
\end{equation}
On comparing this with \eqref{eq:1.1}, one sees that $\kappa_\pi : \Gamma_\pi \times
\Gamma_\pi \to K^\times$ is a 2-cocycle. Notice that the algebra 
$K [\Gamma_\pi,\kappa_\pi^{-1}]$ acts on $V_\pi$ by $T_\gamma \mapsto I^\gamma$.
Let $[\Gamma_\pi / N] \subset \Gamma_\pi$ be a set of representatives for 
$\Gamma_\pi / N$. We may pick the $I^\gamma$ such that 
\begin{equation}\label{eq:1.7}
I^{\tilde \gamma n} = I^{\tilde \gamma} \circ \pi (n) \quad 
\forall \tilde \gamma \in [\Gamma_\pi / N] , n \in N .
\end{equation}
It follows from \eqref{eq:1.3} that $I^{n \tilde \gamma} = \pi (n) \circ 
I^{\tilde \gamma}$ and that $\kappa_\pi$ factors as
\[
\kappa_\pi : \Gamma_\pi \times \Gamma_\pi \to \Gamma_\pi /N \times
\Gamma_\pi / N \to K^\times
\]
Let $\natural : \Gamma / N \times \Gamma / N \to K^\times$ be a 2-cocycle. Thus 
we can construct the twisted group algebras $K[\Gamma,\natural]$
and $K[\Gamma_\pi / N,\natural \kappa_\pi]$. To avoid confusion we denote the
standard basis elements of $K[\Gamma,\natural]$ by $S_\gamma$.

\begin{prop}\label{prop:1.1}
Let $(\tau,M)$ be a representation of $K[\Gamma_\pi / N,\natural \kappa_\pi]$.
\enuma{
\item The algebra $K[\Gamma_\pi,\natural]$ acts on $M \otimes_K V_\pi$ by
\[
S_\gamma (m \otimes v) = \tau (T_{\gamma N}) m \otimes I^\gamma (v) 
\quad h \in K[\Gamma_\pi / N,\natural \kappa_\pi], v \in V_\pi .
\]
\item The $K$-linear map 
\[
\begin{array}{lcccl}
T : & \ind_{K [N]}^{K[\Gamma_\pi,\natural]} (V_\pi) & \to & 
K[\Gamma_\pi / N,\natural \kappa_\pi] \otimes_K V_\pi & \\
& S_{\tilde \gamma} \otimes v & \mapsto & T_{\tilde \gamma N} 
\otimes I^{\tilde \gamma} (v) & \quad \tilde \gamma \in [\Gamma / N] 
\end{array} 
\]
is an isomorphism of $K[\Gamma_\pi,\natural]$-representations.
\item The map $M \mapsto \ind_{K[\Gamma_\pi,\natural]}^{K[\Gamma,\natural]}
(T^{-1} (M \otimes V_\pi))$ is an equivalence between the following categories:
\begin{itemize}
\item subrepresentations of the left regular representation of 
$K[\Gamma_\pi / N,\natural \kappa_\pi]$;
\item $K[\Gamma,\natural]$-subrepresentations of
$\ind_{K [N]}^{K[\Gamma,\natural]} (V_\pi)$. 
\end{itemize}
\item We write $\tau \ltimes \pi := 
\ind_{K [\Gamma_\pi]}^{K[\Gamma,\natural]} (M \otimes V_\pi)$.
For any representation $V$ of $K[\Gamma,\natural]$ there is an isomorphism
\[
\Hom_{K[\Gamma,\natural]}(\tau \ltimes \pi,V) \cong 
\Hom_{K[\Gamma_\pi / N, \natural \kappa_\pi]}(\tau, \Hom_N (\pi,V)) .
\]
}
\end{prop}
\begin{proof}
(a) By \eqref{eq:1.3} and \eqref{eq:1.4}
\begin{align*}
S_\gamma (S_{\gamma'} (m \otimes v)) & = S_\gamma (\tau(T_{\gamma' N}) m \otimes
I^{\gamma'}(v)) \\
& = \tau(T_{\gamma N} T_{\gamma' N}) m \otimes I^{\gamma N} \circ I^{\gamma'}(v) \\
& = (\natural \kappa_\pi) (\gamma,\gamma') \tau(T_{\gamma \gamma' N}) m \otimes
\kappa_\pi (\gamma,\gamma')^{-1} I^{\gamma \gamma'}(v) \\
& = \natural (\gamma,\gamma') \tau (T_{\gamma \gamma' N}) m \otimes I^{\gamma \gamma'}(v) \\
& = \natural (\gamma,\gamma') S_{\gamma \gamma'} (m \otimes v) = 
(S_\gamma S_{\gamma'})(m \otimes v) .
\end{align*}
(b) Since every $I^{\tilde \gamma} : V_\pi \to V_\pi$ is bijective, so is $T$. 
For any $n \in N$:
\begin{multline}\label{eq:1.5}
T (S_n (S_{\tilde \gamma} \otimes v)) = T (S_{\tilde \gamma} \otimes \pi ({\tilde \gamma}^{-1}
n \tilde \gamma ) v)  
= T_{\tilde \gamma N} \otimes I^{\tilde \gamma} \circ \pi ({\tilde \gamma}^{-1} n 
\tilde \gamma ) (v) = \\
T_{\tilde \gamma N} \otimes \pi(n) I^{\tilde \gamma} (v) 
= S_n (T_{\tilde \gamma N} \otimes I^{\tilde \gamma} (v)) = 
S_n (T (S_{\tilde \gamma} \otimes v)) ,
\end{multline}
so $T$ is $N$-equivariant.
Let $\gamma_1 \in \Gamma$ and write $\gamma_1 \tilde \gamma = n {\tilde \gamma}_2$ with
$n \in N$ and ${\tilde \gamma}_2 \in [\Gamma / N ]$. By \eqref{eq:1.5}
\begin{align*}
T( S_{\gamma_1} (S_{\tilde \gamma} \otimes v)) & = T (\natural (\gamma_1, \tilde{\gamma})
S_n S_{{\tilde \gamma}_2} \otimes v) = \natural (\gamma_1, \tilde{\gamma}) S_n 
T (S_{{\tilde \gamma}_2} \otimes v) \\
& = \natural (\gamma_1, \tilde{\gamma}) S_n (T_{{\tilde \gamma}_2 N} \otimes 
I^{{\tilde \gamma}_2}(v)) = \natural (\gamma_1, \tilde{\gamma}) T_{{\tilde \gamma}_2 N} 
\otimes \pi (n) I^{{\tilde \gamma}_2}(v) \\
& = \natural (\gamma_1, \tilde{\gamma}) T_{n {\tilde \gamma}_2 N} 
\otimes I^{n {\tilde \gamma}_2}(v) = \natural (\gamma_1, \tilde{\gamma}) 
T_{\gamma_1 \tilde{\gamma} N} \otimes I^{\gamma_1 \tilde{\gamma}}(v) \\
& = \kappa_\pi (\gamma_1, \tilde{\gamma})^{-1} T_{\gamma_1 N} T_{\tilde{\gamma} N} \otimes 
\kappa_\pi (\gamma_1, \tilde{\gamma}) I^{\gamma_1} I^{\tilde{\gamma}}(v)
= T_{\gamma_1 N} T_{\tilde{\gamma} N} \otimes I^{\gamma_1} I^{\tilde{\gamma}}(v) \\
& = S_{\gamma_1} (T_{\tilde{\gamma} N} \otimes I^{\tilde{\gamma}}(v)) =
S_{\gamma_1} (T (S_{\tilde \gamma} \otimes v)) .
\end{align*}
(c) See \cite[Theorem 11.2.b]{SolGHA}. The proof over there applies because we already
have established parts (a) and (b).\\
(d) We already saw that all these algebras are semisimple. In particular $V$ is
completely reducible. Let $V'$ be the $\pi$-isotypical component of 
$\Res^{K[\Gamma,\natural]}_{K[N]}(V)$. Every $K[\Gamma,\natural]$-homomorphism from
$\tau \ltimes \pi$ has image in $K[\Gamma,\natural] \cdot V'$, so we may assume that
$V = K[\Gamma,\natural] \cdot V'$. Then $V$ can be embedded in a direct sum of copies
of $\ind_{K [N]}^{K[\Gamma,\natural]} (V_\pi)$. Hence it suffices to prove the claim
in the case that $V = \ind_{K [N]}^{K[\Gamma,\natural]} (V_\pi)$. 

By part (b) and the irreducibility of $\pi$
\begin{equation}\label{eq:1.11}
\Hom_N \Big( V_\pi, \ind_{K [N]}^{K[\Gamma,\natural]} (V_\pi) \Big) =
\Hom_N \Big( V_\pi, \ind_{K [N]}^{K[\Gamma_\pi,\natural]} (V_\pi) \Big)
\cong K[\Gamma_\pi / N, \natural \kappa_\pi ].
\end{equation}
By Frobenius reciprocity
\begin{equation}\label{eq:1.9}
\Hom_{K[\Gamma,\natural]} \Big( \tau \ltimes \pi, \ind_{K [N]}^{K[\Gamma,\natural]} 
(V_\pi) \Big) \cong \Hom_{K[\Gamma_\pi,\natural]} \Big( M \otimes V_\pi, 
\ind_{K [N]}^{K[\Gamma,\natural]} (V_\pi) \Big) 
\end{equation}
By \eqref{eq:1.11} the right hand side simplifies to
\begin{equation}\label{eq:1.12}
\Hom_{K[\Gamma_\pi,\natural]} \Big( M \otimes V_\pi, 
\ind_{K [N]}^{K[\Gamma_\pi,\natural]} (V_\pi) \Big) 
\cong \Hom_{K[\Gamma_\pi,\natural]} \Big( M \otimes V_\pi, K[\Gamma_\pi / N, 
\natural \kappa_\pi ] \otimes V_\pi \Big) . 
\end{equation}
As we have seen in part (b), $K[N]$ acts only on the second tensor legs, so
\begin{equation}\label{eq:1.10}
\Hom_{K[N]} \big( M \otimes V_\pi, K[\Gamma_\pi / N, 
\natural \kappa_\pi ] \otimes V_\pi \big) = 
\Hom_K \big( M, K[\Gamma_\pi / N, 
\natural \kappa_\pi ] \big) \otimes K \mathrm{Id}_{V_\pi} .
\end{equation}
An element $\phi = \phi' \otimes \mathrm{Id}_{V_\pi}$ of \eqref{eq:1.10} is a 
$K[\Gamma_\pi, \natural]$-homomorphism if and only if commutes with the action 
described in part (a). On $V_\pi$ it automatically commutes with the $I^\gamma$, 
so the condition becomes that $\phi'$ commutes with left multiplication by 
$T_{\gamma N}$. In other words, $\phi'$ needs to be in $\Hom_{K[\Gamma_\pi / N, 
\natural \kappa_\pi ]} (M, K[\Gamma_\pi / N, \natural \kappa_\pi ])$. In view of 
\eqref{eq:1.11}, \eqref{eq:1.9} is isomorphic with
\[
\Hom_{K[\Gamma_\pi / N, \natural \kappa_\pi ]} \Big( M, 
\Hom_N \big( V_\pi, \ind_{K [N]}^{K[\Gamma,\natural]} (V_\pi) \big) \Big) . \qedhere
\]
\end{proof}

This result leads to a version of Clifford theory.
We will formulate it in terms of extended quotients, see \cite[\S 2]{ABPS5} or
\cite[Appendix B]{ABPS6}. We briefly recall the necessary definitions.

Suppose that $\Gamma$ acts on some set $X$.
Let $\kappa$ be a given function which assigns to each 
$x \in X$ a 2-cocycle $\kappa_x : \Gamma_x \times \Gamma_x \to \C^\times$,
where $\Gamma_x = \{\gamma \in \Gamma : \gamma x = x\}$.   
It is assumed that $\kappa_{\gamma x}$ and $\gamma_* \kappa_x$ define the same
class in $H^2 (\Gamma_x , K^\times)$, where $\gamma_* : \Gamma_x \to \Gamma_{\gamma x}, 
\alpha \mapsto \gamma \alpha \gamma^{-1}$. Define 
\[
\widetilde X_\kappa : = \{(x,\rho) : x \in X, \rho \in \Irr \, K[\Gamma_x, \kappa_x] \}.
\]
We require, for every $(\gamma,x) \in \Gamma \times X$, a definite  algebra isomorphism
\begin{equation}\label{eq:can}
\phi_{\gamma,x} : K [\Gamma_x,\kappa_x] \to K [\Gamma_{\gamma x},\kappa_{\gamma x}]
\end{equation}
such that:
\begin{itemize}
\item $\phi_{\gamma,x}$ is inner if $\gamma x = x$;
\item $\phi_{\gamma',\gamma x} \circ \phi_{\gamma,x} = 
\phi_{\gamma' \gamma,x}$ for all $\gamma',\gamma \in \Gamma, x \in X$.
\end{itemize}
We call these maps connecting homomorphisms, because they are reminiscent 
of a connection on a vector bundle.
Then we can define $\Gamma$-action on $\widetilde X_\kappa$ by
\[
\gamma \cdot (x,\rho) = (\gamma x, \rho \circ \phi_{\gamma,x}^{-1}).
\]
We form the \emph{twisted extended quotient}
\begin{equation}\label{eq:extquot}
(X\q \Gamma)_\kappa : = \widetilde{X}_\kappa / \Gamma.
\end{equation}
Let us return to the setting of Proposition \ref{prop:1.1}.

\begin{thm}\label{thm:1.2}
Let $\kappa \natural$ be the family of 2-cocycles which assigns $\kappa_\pi \natural$
to $\pi \in \Irr_K (N)$. There is a bijection
\[
\begin{array}{ccc}
( \Irr_K (N) \q \, \Gamma / N )_{\kappa \natural} & \to & \Irr (K [\Gamma,\natural]) \\
(\pi,\tau) & \mapsto & \tau \ltimes \pi := 
\ind_{K[\Gamma_\pi,\natural]}^{K[\Gamma,\natural]} (V_\tau \otimes V_\pi) 
\end{array}
\]
\end{thm}
\begin{proof}
With Proposition \ref{prop:1.1}, \cite[Appendix]{SolGHA} becomes valid in our situation.
The theorem is a reformulation of parts (d) and (e) of \cite[Theorem 11.2]{SolGHA}.
For completeness we note that the connecting homomorphism
\[
K [ \Gamma_\pi / N ,\kappa_\pi \natural] \to K [\Gamma_{\gamma \cdot \pi} / N,
\kappa_{\gamma \cdot \pi} \natural ]
\]
is given by conjugation with $I^\gamma_\pi$, as in \cite[(3)]{ABPS5}.
\end{proof}

For convenience we record the special case $\natural = 1$ of the above
explicitly. It is very similar to \cite[p. 24]{RaRa} and \cite[\S 51]{CuRe}.
\begin{align}
& \ind_N^{\Gamma_\pi}(\pi) \cong K[\Gamma_\pi / N,\kappa_\pi] \otimes_K V_\pi
\quad \text{as } \Gamma_\pi \text{-representations} , \\
& \Irr_K (\Gamma) \longleftrightarrow ( \Irr_K (N) \q \, \Gamma / N )_\kappa .
\label{eq:1.6}
\end{align}

It will also be useful to analyse the structure of $K[\Gamma,\natural]$ as 
a bimodule over itself. Let $K[\Gamma,\natural]^{op}$ be the opposite algebra,
and denote its standard basis elements by $S_\gamma \; (\gamma \in \Gamma)$.
\begin{lem}\label{lem:1.3}
\enuma{
\item There is a $K$-algebra isomorphism
\[
\begin{array}{cccc}
* : & K[\Gamma,\natural^{-1}] & \to & K[\Gamma,\natural]^{op} \\
& T_\gamma & \mapsto & T_\gamma^* = S_\gamma^{-1} .
\end{array}
\]
\item There is a bijection
\[
\begin{array}{ccc}
\Irr (K[\Gamma,\natural]) & \to & \Irr (K[\Gamma,\natural^{-1}]) \\
V & \mapsto & V^* = \Hom_K (V,K) ,
\end{array}
\]
where $(h \cdot \lambda)(v) = \lambda (h^* \cdot v)$ for 
$v \in V, \lambda \in V^*$ and $h \in K[\Gamma,\natural^{-1}]$.
\item Let $K[\Gamma,\natural] \oplus K[\Gamma,\natural^{-1}]$ act on
$K[\Gamma,\natural]$ by $(a,h) \cdot b = a b h^*$. \\
As $K[\Gamma,\natural] \oplus K[\Gamma,\natural^{-1}]$-modules
\[
K[\Gamma,\natural] \cong 
\bigoplus\nolimits_{V \in \Irr (K[\Gamma,\natural])} V \otimes V^* . 
\]
}
\end{lem}
\begin{proof}
(a) The map is $K$-linear by definition, and it clearly is bijective. For
$\gamma, \gamma' \in \Gamma$:
\[
T_\gamma^* \cdot T_{\gamma'}^* = S_\gamma^{-1} \cdot S_{\gamma'}^{-1} =
(S_{\gamma'} \cdot S_\gamma )^{-1} = (\natural (\gamma,\gamma') 
S_{\gamma \gamma'})^{-1} = \natural (\gamma,\gamma')^{-1} T_{\gamma \gamma'}^* ,
\]
so * is an algebra homomorphism. \\
(b) Trivial, it holds for any finite dimensional algebra and its opposite.\\
(c) Let $\tilde \Gamma$ be a Schur extension of $\Gamma$, as on page \pageref{eq:1.2}. 
As a representation of $\tilde \Gamma \times (\tilde \Gamma)^{op}$, 
$K[\tilde \Gamma]$ decomposes in the asserted manner. Hence the same holds for its
direct factor $K[\Gamma,\natural]$.
\end{proof}

\section{The generalized Springer correspondence}
\label{sec:genSpr}

Let $G$ be a connected complex reductive group.
The generalized Springer correspondence for $G$ has been constructed by Lusztig. 
We will recall the main result of \cite{Lus1}, and then we prove that Lusztig's
constructions are equivariant with respect to automorphisms of algebraic groups.

Let $l$ be a fixed prime number, and let $\overline{\Q_{\ell}}$
be an algebraic closure of $\Q_{\ell}$. For compatibility with the literature we
phrase our results with $\overline{\Q_{\ell}}$-coefficients. However, by their
algebro-geometric nature everything works just as well with coefficients in
any other algebraically closed field of characteristic zero.

For $u$ a unipotent element in $G$, we denote by $A_{G}(u)$ the group of 
components $\Cent_{G}(u)/\Cent_{G}(u)^\circ$ of the centralizer in 
$G$ of $u$. We set
\[
\cN_{G}^+ := \{ (u,\eta ) : u \in G \text{ unipotent} , 
\eta \in \Irr_{\overline{\Q_{\ell}}} (A_{G} (u)) \} / G\text{-conjugacy}.
\]
The set $\cN_{G}^+$ is canonically in bijection with the set of pairs 
$(\cC_u^{G},\cF)$, where $\cC_u^{G}$ is the $G$-conjugacy class 
of a unipotent element $u\in G$ 
and $\cF$ is an irreducible $G$-equivariant local system on $\cC_u^{G}$.
The bijection associates to $(\cC_u^{G},\cF)$ an element $u \in \cC_u^G$
and the representation of $A_G (u)$ on the stalk $\cF_u$.

Let $P$ be a parabolic subgroup of $G$ with unipotent radical $U$, and let $L$ be 
a Levi factor of $P$. Let $v$ be a unipotent element in $L$. The group 
$\Cent_{G}(u)\times\Cent_{L}(v)U$ acts on the variety 
\begin{equation}\label{eq:2.2}
Y_{u,v} := \left\{y\in G \;:\, y^{-1}uy\in vU\right\} 
\end{equation}
by $(g,p)\cdot y=gyp^{-1}$, with $g\in\Cent_{G}(u)$, $p\in\Cent_{L}(v) U$, 
and $y\in Y_{u,v}$. We have 
\[
\dim Y_{u,v}\le d_{u,v} :=
\frac{1}{2}(\dim \Cent_{G}(u)+\dim\Cent_{L}(v))+\dim U.
\] 
The group $A_{G}(u)\times A_{L} (v)$ acts on the set of irreducible components of 
$Y_{u,v}$ of dimension $d_{u,v}$; we denote by $\sigma_{u,v}$ the corresponding 
permutation representation.
 
Let $\langle \;,\;\rangle_{A_{G} (u)}$ be the usual scalar product of the set of class 
functions on the finite group $A_{G} (u)$ with values in $\overline{\Q_{\ell}}$. 
An irreducible representation $\eta$ of $A_{G} (u)$ is called \emph{cuspidal} 
(see \cite[Definition 2.4]{Lus1} and \cite[\S 0.4]{LS}) if 
\begin{equation}
\langle \eta,\sigma_{u,v}\rangle_{A_{G} (u)}\ne 0 \text{ implies that } P=G.
\end{equation} 
If $A_G (u)$ has a cuspidal representation, then \cite[Proposition 2.8]{Lus1} implies
that $u$ is a distinguished unipotent element of $G$, i.e. not contained in any
proper Levi subgroup of $G$. However, in general not every distinguished unipotent
element supports a cuspidal representation.
The set of irreducible cuspidal representations of $A_{G} (u)$ 
(over $\overline{\Q_{\ell}}$) is denoted by $\Irr_\cusp (A_{G} (u))$, and we write
\[
\cN_{G}^0 = \{ (u,\eta ) : u \in G \text{ unipotent, } 
\eta \in \Irr_\cusp (A_{G} (u)) \} 
/ G\text{-conjugacy} .
\]
Given a pair $(u,\eta)\in\cN_{G}^+$, there exists a triple $(P,L,v)$ 
as above and an 
\[
\epsilon\in \Irr_\cusp (A_{L} (v)) \quad \text{such that} \quad 
\langle \eta \otimes \epsilon^*,\sigma_{u,v}\rangle_{A_G (u) \times A_L (v)}\ne 0,
\]
where $\epsilon^*$ is the dual of $\epsilon$ (see \cite[\S~6.2]{Lus1} and 
\cite[\S 0.4]{LS}). Moreover $(P,L,v,\epsilon)$ is unique up to $G$-conjugation 
(see \cite[Prop.~6.3]{Lus1} and \cite{LS}). We denote by 
$\ft:=[L,\cC_v^{L},\epsilon]_{G}$ the $G$-conjugacy class of 
$(L,v,\epsilon)$ and we call it the \emph{cuspidal support} of the pair $(u,\eta)$.
The centre $Z(G)$ maps naturally to $A_{G}(u)$ and to $A_L (v)$. 
By construction \cite[Theorem 6.5.a]{Lus1}
\begin{equation}\label{eq:2.13}
\eta \text{ and } \epsilon \text{ have the same } Z(G)\text{-character.}  
\end{equation}
Let $\cS_{G}$ denote the set consisting of all triples $(L,\cC_v^{L},\epsilon)$ 
(up to $G$-conjugacy) where $L$ is a Levi subgroup of a parabolic subgroup of $G$,
$\cC_v^{L}$ is the $L$-conjugacy class of a unipotent element $v$ in $L$ and 
$\epsilon\in\Irr_\cusp (A_{L} (v))$. Let 
\begin{equation}\label{eq:2.1}
\Psi_{G}\colon\cN_{G}^+\to \cS_{G}
\end{equation}
be the map defined by sending the $G$-conjugacy class of $(u,\eta)$ 
to its cuspidal support. By \eqref{eq:2.13} this map preserves the 
$Z(G)$-characters of the involved representations.

In \cite[§3.1]{Lus1}, Lusztig defined a partition of $G$ in a finite number of 
irreducible, smooth, locally closed subvarieties, stable under conjugation. 
For all $g \in G$, we denote by $g_s$ the semisimple part of $g$. We say that 
$g \in G$ (or its conjugacy class) is isolated if $Z_{G}(g_s)^{\circ}$ is not 
contained in any proper Levi subgroup of $G$. In particular every unipotent
conjugacy class is isolated.

Let $L$ be a Levi subgroup of $G$ and $S \subset L$ the inverse image of an isolated 
conjugacy class of $L/Z_{L}^{\circ}$ by the natural projection map 
$L \twoheadrightarrow L/Z_{L}^{\circ}$. Denote by 
\[
S_{\reg}=\{g \in S, \,\, Z_G(g_s)^{\circ} \subset L \}
\] 
the set of regular elements in $S$.
Consider the irreducible, smooth, locally closed subvariety of $G$ defined by 
\[ 
Y_{(L,S)} = \bigcup\nolimits_{g\in G} g S_{\reg} g^{-1} = 
\bigcup\nolimits_{x \in S_{\reg}} \mathcal{C}_{x}^G. 
\] 
We remark that $Y_{(L,S)}$ depends only on the $G$-conjugacy class of $(L,S)$.
 
Now, let $P=L U_P$ a parabolic subgroup of $G$ with Levi factor $L$,  denote  
$\overline{\mathbf{c}}=(P,L,S), \,\, \mathbf{c}=(L,S)$ and let 
\begin{align*}
& \widehat{X}_{\overline{\mathbf{c}}}=\{(g,x) \in G \times G, 
\,\, x^{-1}g x \in \overline{S} \cdot U_P \}, \\
& X_{\overline{\mathbf{c}}}=\{(g,xP) \in G \times G/P, \,\, 
x^{-1}g x \in \overline{S} \cdot U_P \},
\end{align*} 
where $\overline{S}$ is the closure of $S$. The subgroup $P$ acts freely by 
translation on right on the second coordinate of an element of 
$\widehat{X}_{\overline{\mathbf{c}}}$  and $\widehat{X}_{\overline{\mathbf{c}}}/P = 
X_{\overline{\mathbf{c}}}$. After \cite[4.3]{Lus1}, the projection on the 
first coordinate $\phi_{\overline{\mathbf{c}}} : X_{\overline{\mathbf{c}}} 
\longrightarrow G$ is proper and its image is $\overline{Y}_{\mathbf{c}}$.

The group $Z_L^{\circ}$ acts on $\overline{S}$ by translation and $L$ acts on 
$\overline{S}$ by conjugation. This gives rises to an action of $Z_L^{\circ}\times L$ 
on $\overline{S}$. The orbits form a stratification of $\overline{S}$, in which $S$ 
is the unique open stratum. Denote by $\sigma_{\overline{\mathbf{c}}} : 
\widehat{X}_{\overline{\mathbf{c}}} \longrightarrow \overline{S}$ the map which 
associates to $(g,x)$ the projection of $x^{-1}gx \in \overline{S} \cdot U_P$ 
on the factor $\overline{S}$ and $\varpi_{P} : \widehat{X}_{\overline{\mathbf{c}}} 
\longrightarrow X_{\overline{\mathbf{c}}}$ the map defined for all 
$(g,x) \in \widehat{X}_{\overline{\mathbf{c}}}$ by $\varpi_{P}(g,x)=(g,xP)$. 
To sum up, we have the following diagram: 
\[
\xymatrix{
 & & \ar[lldd]^{\sigma_{\overline{\mathbf{c}}}} \widehat{X}_{\overline{\mathbf{c}}} 
 \ar[rrd]^{\varpi_P}& & \\
 & &  & & X_{\overline{\mathbf{c}}} \ar[ld]^{\phi_{\overline{\mathbf{c}}}} \\
\overline{S}& &  & \overline{Y}_{\mathbf{c}}
}
\]
By taking image inverse under $\sigma_{\overline{\mathbf{c}}}$, the stratification of 
$\overline{S}$ gives a stratification of $\widehat{X}_{\overline{\mathbf{c}}}$. 
The stratum $\widehat{X}_{\overline{\mathbf{c}},\alpha}$ (corresponding to the open 
stratum $S$) is open and dense. We denote by $\sigma_{\overline{\mathbf{c}},\alpha}$ 
the restriction of $\sigma_{\overline{\mathbf{c}}}$ to 
$\widehat{X}_{\overline{\mathbf{c}},\alpha}$. Every stratum of 
$\widehat{X}_{\overline{\mathbf{c}}}$ is $P$-invariant and their images in 
$X_{\overline{\mathbf{c}}}=\widehat{X}_{\overline{\mathbf{c}}}/P$ form a stratification 
of $X_{\overline{\mathbf{c}}}$, with $X_{\overline{\mathbf{c}},\alpha}=
\widehat{X}_{\overline{\mathbf{c}},\alpha}/P$ open and dense.

Let $\mathcal{E}$ be an irreducible $L$-equivariant cuspidal local system on $S$. Then 
$(\sigma_{\overline{\mathbf{c}},\alpha})^{*}\mathcal{E}$ is a $G\times P$-equivariant 
local system on $\widehat{X}_{\overline{\mathbf{c}},\alpha}$. There exists a 
unique $G$-equivariant local system on $X_{\overline{\mathbf{c}},\alpha}$, denoted by 
$\overline{\mathcal{E}}$, such that $(\sigma_{\overline{\mathbf{c}},\alpha})^{*} 
\mathcal{E}=(\varpi_{P})^{*} \overline{\mathcal{E}}$.

We denote by $\widetilde{Y}_{\mathbf{c}}=\phi_{\overline{\mathbf{c}}}^{-1}(Y_{\mathbf{c}})$, 
$\pi_{\mathbf{c}}=\restriction{\phi_{\overline{\mathbf{c}}}}{\widetilde{Y}_{\mathbf{c}}}$, 
$\widetilde{\mathcal{E}}=\restriction{\overline{\mathcal{E}}}{\widetilde{Y}_{\mathbf{c}}}$ 
and $$\mathcal{A}_{\mathcal{E}}=\End_{\mathcal{D} 
Y_{\mathbf{c}}}((\pi_{\mathbf{c}})_{*}\widetilde{\mathcal{E}}) \simeq \End_{\mathcal{D}_{G} 
Y_{\mathbf{c}}}((\pi_{\mathbf{c}})_{*}\widetilde{\mathcal{E}}), $$ where 
$\mathcal{D} Y_{\mathbf{c}}$ (resp. $\mathcal{D}_{G} Y_{\mathbf{c}}$) 
is the bounded derived category of $\overline{\Q}_{\ell}$-constructible sheaves 
(resp. $G$-equivariant) on $Y_{\mathbf{c}}$. We denote by $\Irr(\mathcal{A}_{\mathcal{E}})$ 
the set of (isomorphism classes of)  simple $\mathcal{A}_{\mathcal{E}}$-modules and 
$\overline{\Q}_{\ell}$ the constant sheaf.

Let $K_{\overline{\mathbf{c}}}=\IC(X_{\overline{\mathbf{c}}},\overline{\mathcal{E}})$ 
the intersection cohomology complex of Deligne--Goresky--MacPherson on 
$X_{\overline{\mathbf{c}}}$, with coefficients in $\overline{\mathcal{E}}$. Then 
$(\phi_{\overline{\mathbf{c}}})_{!}K_{\overline{\mathbf{c}}}$ is a complex on 
$\overline{Y}_{\mathbf{c}}$.

\begin{thm}\label{thmlus}
\cite[Theorem 6.5]{Lus1} \ \\
Let $\mathfrak{t}=[L,\mathcal{C}_{v}^{L},\mathcal{E}] \in \mathcal{S}_{G}$, 
$(S,\mathcal{E})=(Z_{L}^{\circ} \cdot \mathcal{C}_{v}^{L}, 
\overline{\Q}_{\ell} \boxtimes \mathcal{E})$ the corresponding cuspidal pair for 
$L$ and $P$ a parabolic subgroup of $G$ with Levi factor $L$. As before, 
we denote by $\overline{\mathbf{c}}=(P,L,S), \,\, \mathbf{c}=(L,S)$ and 
$(\phi_{\overline{\mathbf{c}}})_{!}K_{\overline{\mathbf{c}}}$ the corresponding complex 
on $\overline{Y}_{\mathbf{c}}$.
\begin{enumerate}
\item Let $(\mathcal{C}_{u}^{G},\mathcal{F}) \in \mathcal{N}_{G}^{+}$. 
Then $\Psi_{G}(\mathcal{C}_{u}^{G},\mathcal{F})=(L,\mathcal{C}_{v}^{L},\mathcal{E})$, 
if and only if the following conditions are satisfied : 
\begin{enumerate}
\item $\mathcal{C}_{u}^{G} \subseteq \overline{Y}_{\mathbf{c}}$ ;
\item $\mathcal{F}$ is a direct summand of 
$\restriction{R^{2 d_{\mathcal{C}_{u}^{G},\mathcal{C}_{v}^{L}}} 
(f_{\overline{\mathbf{c}}})_{!}(\overline{\mathcal{E}})}{\mathcal{C}_{u}^{G}}$, 
where $f_{\overline{\mathbf{c}}}$ is the restriction of $\phi_{\overline{\mathbf{c}}}$ 
to $X_{\overline{\mathbf{c}},\alpha} \subset X_{\overline{\mathbf{c}}}$, 
$d_{\mathcal{C}_{u}^{G},\mathcal{C}_{v}^{L}}=(\nu_G-\frac{1}{2} 
\dim \mathcal{C}_{u}^{G})-(\nu_L-\frac{1}{2} \dim \mathcal{C}_{v}^{L})$, and 
$\nu_{G}$ (resp. $\nu_{L}$) is the number of positive roots of $G$ (resp. $L$).
\end{enumerate}
\item The natural morphism $$\restriction{R^{2 d_{\mathcal{C}_{u}^{G},
\mathcal{C}_{v}^{L}}} (f_{\overline{\mathbf{c}}})_{!}(\overline{\mathcal{E}})}{
\mathcal{C}_{u}^{G}} \longrightarrow\restriction{ \mathcal{H}^{2d_{\mathcal{C}_{u}^{G},
\mathcal{C}_{v}^{L}}}((\phi_{\overline{\mathbf{c}}})_{!}K_{\overline{\mathbf{c}}})}{
\mathcal{C}_{u}^{G}}$$ given by the imbedding of $X_{\overline{\mathbf{c}},\alpha}$ 
into $X_{\overline{\mathbf{c}}}$ as an open subset, is an isomorphism.
\item For all $\rho \in \Irr(\mathcal{A}_{\mathcal{E}})$, let $((\phi_{\overline{
\mathbf{c}}})_{!}K_{\overline{\mathbf{c}}})_{\rho}$ the $\rho$-isotypical component 
of $(\phi_{\overline{\mathbf{c}}})_{!}K_{\overline{\mathbf{c}}}$, i.e. 
\[
(\phi_{\overline{\mathbf{c}}})_{!}K_{\overline{\mathbf{c}}}=
\bigoplus\nolimits_{\rho \in \Irr(\mathcal{A}_{\mathcal{E}})} \rho \boxtimes 
((\phi_{\overline{\mathbf{c}}})_{!}K_{\overline{\mathbf{c}}})_{\rho}.
\]
Let $\overline{Y}_{\mathbf{c},\uni}$ be the variety of unipotent elements in 
$\overline{Y}_{\mathbf{c}}$. There exists an unique pair $(\mathcal{C}_{u}^{G},\mathcal{F}) 
\in \mathcal{N}_{G}^{+}$ which satisfies the following conditions: 
\begin{enumerate}
\item $\mathcal{C}_{u}^{G} \subset \overline{Y}_{\mathbf{c}}$;
\item $\restriction{((\phi_{\overline{\mathbf{c}}})_{!}K_{\overline{\mathbf{c}}})_{\rho}}{
\overline{Y}_{\mathbf{c},\uni}}$ is isomorphic to 
$\IC(\overline{\mathcal{C}_{u}^{G}},\mathcal{F})[2d_{\mathcal{C}_{u}^{G},\mathcal{C}_{v}^{L}}]$ 
extended by $0$ on $\overline{Y}_{\mathbf{c},\uni}-\overline{\mathcal{C}_{u}^{G}}$.
\end{enumerate}
In particular, $\displaystyle \mathcal{F}=
\restriction{\mathcal{H}^{2d_{\mathcal{C}_{u}^{G},\mathcal{C}_{v}^{L}}}
\left(((\phi_{\overline{\mathbf{c}}})_{!}K_{\overline{\mathbf{c}}})_{\rho} \right)}{
\mathcal{C}_{u}^{G}}$ and \\
$\rho = \Hom_G \big( \cF, \cH^{2d_{\mathcal{C}_{u}^{G}, \mathcal{C}_{v}^{L}}} \left( 
(\phi_{\overline{\mathbf{c}}})_{!} K_{\overline{\mathbf{c}}} \right) \big|_{\cC_u^G} \big)$. 
The map 
\[
\Sigma_{\ft} : \Psi_G^{-1}(\ft) \to \Irr (\mathcal{A}_{\mathcal{E}}) 
\]
which associates $\rho$ to $(\mathcal{C}_{u}^{G},\mathcal{F})$ is a bijection.
\end{enumerate}
\end{thm}

The relation of Theorem \ref{thmlus} with the classical Springer correspondence 
goes via $\mc A_\cE$, which turns out to be isomorphic to the group algebra of 
a Weyl group. We define
\begin{equation} \label{eqn Wtcirc}
W_{\ft} := N_{G}(\ft) / L = N_G (L,\cC_v^L,\cE) / L .
\end{equation} 

\begin{thm}\label{thm:2.2} \cite[Theorem 9.2]{Lus1} 
\enuma{
\item $W_{\ft} = N_{G}(L) / L$.
\item $N_{G}(L) / L$ is the Weyl group of the root system $R(G,Z(L)^\circ)$.
\item There exists a canonical algebra isomorphism $\mc A_\cE \cong \overline{\Q_{\ell}}[W_\ft]$.
Together with Theorem \ref{thmlus}.(3) this gives a canonical bijection 
$\Psi_G^{-1}(\ft) \to \Irr_{\Q_{\ell}}(W_\ft)$.}
\end{thm}

In fact there exist two such canonical algebra isomorphisms, for one can always twist 
with the sign representation of $W_\ft$. When we employ generalized Springer correspondences 
in relation with the local Langlands correspondence, we will always use the isomorphism 
$\mc A_\cE \cong \overline{\Q_{\ell}}[W_\ft]$ such that the trivial $W_\ft$-representation 
is the image of $(\cC_v^G, \cE)$ under Theorems \ref{thmlus} and \ref{thm:2.2}. 
(Here we extend $\cE \; G$-equivariantly to $\cC_v^G$, compare with \cite[9.5]{Lus1}.)

Let $H$ be a group which acts on the connected complex reductive group $G$ by
algebraic automorphisms.
Then $H$ acts also on $\mathcal{N}_{G}^{+}$ and $\mathcal{S}_{G}$. Indeed, let $h \in H$,  
$(\mathcal{C}^{G}_{u},\mathcal{F}) \in \mathcal{N}_{G}^{+}$, $\mathfrak{t}=
[L,\mathcal{C}_{v}^{L},\cE]_{G} \in \mathcal{S}_{G}$ and $\rho \in \Irr(W_{\mathfrak{t}})$. 
Since $h (G) = G$, ${}^h \mathcal{C}_{u}^{G} = \mathcal{C}_{h \cdot u}^{G}$ 
is a unipotent orbit of $G$. Similarly, ${}^h L$ is a Levi subgroup of 
$G$, ${}^h \mathcal{C}_{v}^{L}$ is a unipotent orbit of ${}^h L$, etc. 

We denote by $h^*$ the pullback of sheaves along the isomorphism $h^{-1} : G \to G$. 
Thus $h^{*} \mathcal{F}$ (resp. $h^{*}\mathcal{L}$) is a local system on 
$\mathcal{C}^{G}_{h \cdot u}$ (resp. ${}^{h} \mathcal{C}_{v}^{L}$). Keeping the 
above notation, the action of $H$ on $\mathcal{N}_{G}^{+}, \,\, 
\mathcal{S}_{G}$ and $\Irr(W)$ is given by
\[
h \cdot (\mathcal{C}^{G}_{u},\mathcal{F})=(\mathcal{C}^{G}_{h \cdot u},h^{*} 
\mathcal{F}), \qquad h \cdot [L,\mathcal{C}_{v}^{L},\mathcal{L}]=[{}^{h} L,
\mathcal{C}^{{}^{h} L}_{h \cdot v}, h^{*} \mathcal{L}] 
\]
and $h \cdot \rho = \rho^{h} \in \Irr(W_{h \cdot \mathfrak{t}})$.

\begin{thm}\label{thm:3.2}
The Springer correspondence for $G$ is $H$-equivariant. More precisely, for all 
$h \in H$, the following diagrams are commutative:
\[
\xymatrix{
\mathcal{N}_{G}^{+} \ar[rr]^{\Psi_{G}} \ar[d]^{h} & & \mathcal{S}_{G}^{+} \ar[d]^{h} & & 
\Psi_G^{-1}(\ft) \ar[d]^{h} \ar[rr]_{\Sigma_{\ft}} & & \Irr(W_{\mathfrak{t}}) \ar[d]^h\\
\mathcal{N}_{G}^{+} \ar[rr]^{\Psi_{G}} & & \mathcal{S}_{G}^{+} & & \Psi_G^{-1}(h \cdot \ft) 
\ar[rr]_{\Sigma_{h \cdot \ft}} & & \Irr(W_{h \cdot \mathfrak{t}})
}
\]
In other words, for all $h \in H$, $(\mathcal{C}_{u}^{G},\mathcal{F}) \in 
\Psi_G^{-1}(\ft) \subset \mathcal{N}_{G}^{+}$:
\[
\Psi_{G}(h \cdot (\mathcal{C}_{u}^{G}, \mathcal{F})) = 
h \cdot \Psi_{G}(\mathcal{C}_{u}^{G},\mathcal{F}) \quad 
\mbox{and} \quad \Sigma_{h \cdot \ft}(h \cdot (\mathcal{C}_{u}^{G},\mathcal{F}))=
h \cdot \Sigma_{\ft}(\mathcal{C}_{u}^{G},\mathcal{F}).
\]
\end{thm}
\begin{proof}
We keep the notations of Theorem \ref{thmlus}. 
Let $h \in H$, $(\mathcal{C}_{u}^{G},\mathcal{F}) \in \mathcal{N}_{G}^{+}$, $P$ 
a parabolic subgroup of $G$ with Levi factor $L$, $v \in L$ a unipotent element and 
$\cE$ an irreducible cuspidal $L$-equivariant local system on $\mathcal{C}_{v}^{L}$ 
such that  
\[
\Psi_{G}(\mathcal{C}_{u}^{G},\mathcal{F})=[L,\mathcal{C}_{v}^{L},\cE] \in 
\mathcal{S}_{G}.
\]
As in Theorem \ref{thmlus}, let $(S,\mathcal{E})=(Z_{L}^{\circ} \cdot 
\mathcal{C}_{v}^{L}, \overline{\Q}_{\ell} \boxtimes \cE)$ be the corresponding 
cuspidal pair for $L$ and $\overline{\mathbf{c}}=(P,L,S), \,\, \mathbf{c}=(L,S)$.
After (1) in Theorem \ref{thmlus}, $\mathcal{C}_{u}^{G} \subset \overline{Y}_{\mathbf{c}}$, 
so ${}^h\mathcal{C}_{u}^{G} \subset {}^{h} \overline{Y}_{\mathbf{c}}=
\overline{Y}_{h \cdot \mathbf{c}}$, where $h \cdot \mathbf{c}=({}^{h} L,{}^{h} S) $.
Consider the maps 
\[
\begin{array}[t]{rcl}
\widehat{X}_{h \cdot \overline{\mathbf{c}}} & \longrightarrow & 
\widehat{X}_{\overline{\mathbf{c}}}\\
(g,x) & \longmapsto &  ({}^{h^{-1}}g,{}^{h^{-1}}x)
\end{array}, \,\,\begin{array}[t]{rcl}
X_{h \cdot \overline{\mathbf{c}}} & \longrightarrow & X_{\overline{\mathbf{c}}}\\
(g,x \,{}^{h}P) & \longmapsto &  ({}^{h^{-1}}g,{}^{h^{-1}}xP)
\end{array}, \,\,
\begin{array}[t]{rcl}
G & \longrightarrow & G \\
g & \longmapsto & {}^{h^{-1}}g
\end{array}.
\]
and the following diagrams:
\[
\xymatrix{
\widehat{X}_{h \cdot \overline{\mathbf{c}},\alpha} \ar[r]^{h} \ar[d]^{\sigma_{h 
\cdot \overline{\mathbf{c}},\alpha}} & \widehat{X}_{\overline{\mathbf{c}},\alpha} 
\ar[d]^{\sigma_{\overline{\mathbf{c}},\alpha}} \\
{}^{h} S \ar[r]^{h} & S
},\,\, 
\xymatrix{
\widehat{X}_{h \cdot \overline{\mathbf{c}},\alpha} \ar[r]^{h} \ar[d]^{\varpi_{{}^{h} P}} & 
\widehat{X}_{\overline{\mathbf{c}},\alpha} \ar[d]^{\varpi_{P}} \\
X_{h \cdot \overline{\mathbf{c}},\alpha} \ar[r]^{h} & X_{\overline{\mathbf{c}},\alpha}
},\,\, 
\xymatrix{
X_{h \cdot \overline{\mathbf{c}},\alpha} \ar[r]^{h} \ar[d]^{f_{h \cdot 
\overline{\mathbf{c}}}} & X_{\overline{\mathbf{c}},\alpha} \ar[d]^{f_{\overline{\mathbf{c}}}} \\
\overline{Y}_{h \cdot \mathbf{c}} \ar[r]^{h} & \overline{Y}_{\mathbf{c}}
}.
\]
The first two commutative diagrams show that:
\begin{eqnarray*}
(\sigma_{h \cdot \overline{\mathbf{c}},\alpha})^{*} (h^{*} \mathcal{E}) & = & 
h^{*} (\sigma_{\overline{\mathbf{c}},\alpha})^{*} (\mathcal{E})\\
& = & h^{*} (\varpi_P)^{*} (\overline{\mathcal{E}})
\; = \; (\varpi_{{}^{h} P})^{*} (h^{*} \overline{\mathcal{E}}) .
\end{eqnarray*}
By unicity, this shows that $h^{*} \overline{\mathcal{E}}=\overline{h^{*}\mathcal{E}}$. 
The third cartesian diagram shows, by the proper base change theorem, that
\[
h^{*} R^{2d_C} (f_{\overline{\mathbf{c}}})_{!}(\overline{\mathcal{E}}) \cong
R^{2d_C} (f_{h \cdot \overline{\mathbf{c}}})_{!}(h^{*} \overline{\mathcal{E}}) =
R^{2d_C} (f_{h \cdot \overline{\mathbf{c}}})(\overline{h^{*} \mathcal{E}}).
\]
Because 
\begin{eqnarray*} 
0 & \neq \Hom_{\mathcal{D}\mathcal{C}_{u}^{G}}(\mathcal{F},
\restriction{R^{2 d_C} (f_{\overline{\mathbf{c}}})_{!}(\overline{\mathcal{E}})}{
\mathcal{C}_{u}^{G}}))  & \cong \Hom_{\mathcal{D}{}^{h}\mathcal{C}_{u}^{G}}(h^{*} 
\mathcal{F},h^{*} \restriction{R^{2 d_C} (f_{\overline{\mathbf{c}}})_{!}
(\overline{\mathcal{E}})}{{}^{h} \mathcal{C}_{u}^{G}}))\\
& & \cong \Hom_{\mathcal{D}{}^{h}\mathcal{C}_{u}^{G}}(h^{*} \mathcal{F},
\restriction{R^{2 d_C} (f_{h \cdot \overline{\mathbf{c}}})_{!}(\overline{h^{*}
\mathcal{E}})}{{}^{h} \mathcal{C}_{u}^{G}})) \neq 0,
\end{eqnarray*}
with $d_C = d_{\mathcal{C}_{u}^{G},\mathcal{C}_{v}^{L}}=d_{{}^{h} 
\mathcal{C}_{u}^{G},{}^{h} \mathcal{C}_{v}^{L}}$. Thus $h^{*} \mathcal{F}$ 
is a direct summand of $\restriction{R^{2 d} (f_{h \cdot \overline{\mathbf{c}}})_{!}
(\overline{h^{*} \mathcal{E}})}{{}^{h} \mathcal{C}_{u}^{G}}$ and after 
Theorem \ref{thmlus}, $\Psi_{G}$ is $H$-equivariant.

According to \cite[Proposition 5.4]{Gor}
\[
h^{*} K_{\overline{\mathbf{c}}} = 
h^{*} \IC(X_{\overline{\mathbf{c}}},\overline{\mathcal{E}}) = \IC(h^{*} 
X_{\overline{\mathbf{c}}},h^{*} \overline{\mathcal{E}}) = \IC (X_{h \cdot 
\overline{\mathbf{c}}},\overline{h^{*}  \mathcal{E}}) = K_{h \cdot \overline{\mathbf{c}}}.
\]
Let $\rho \in \Irr(\mathcal{A}_{\mathcal{E}})$. By functoriality, 
$\mathcal{A}_{h^{*} \mathcal{E}} \simeq \mathcal{A}_{\mathcal{E}}$ and by considering 
the third commutative diagram, we get:
\begin{eqnarray*}
\Hom_{\mathcal{A}_{\mathcal{E}}}(\rho,(\phi_{\overline{\mathbf{c}}})_{!}
K_{\overline{\mathbf{c}}})  \cong & \Hom_{\mathcal{A}_{h^{*} \mathcal{E}}}(h^{*} \rho,
h^{*}  (\phi_{\overline{\mathbf{c}}})_{!}K_{\overline{\mathbf{c}}})\\
\cong &  \Hom_{\mathcal{A}_{h^{*} \mathcal{E}}}(\rho^{h}, (\phi_{h \cdot 
\overline{\mathbf{c}}})_{!}K_{h \cdot\overline{\mathbf{c}}}) \\
h^{*} ((\phi_{\overline{\mathbf{c}}})_{!}K_{\overline{\mathbf{c}}})_{\rho} \cong & 
((\phi_{h \cdot \overline{\mathbf{c}}})_{!}K_{h \cdot \overline{\mathbf{c}}})_{h^{*} \rho}
\end{eqnarray*}
Since $\restriction{((\phi_{\overline{\mathbf{c}}})_{!}K_{\overline{\mathbf{c}}})_{\rho}}
{\overline{Y}_{\mathbf{c},\uni}} \simeq \IC(\overline{\mathcal{C}_{u}^{G}},
\mathcal{F})[2d_{\mathcal{C}_{u}^{G},\mathcal{C}_{v}^{L}}]$, we have 
\begin{eqnarray*}
h^{*}\restriction{((\phi_{\overline{\mathbf{c}}})_{!}K_{\overline{\mathbf{c}}})_{\rho}}{
\overline{Y}_{\mathbf{c},\uni}} \cong h^{*} \IC(\overline{\mathcal{C}_{u}^{G}},
\mathcal{F})[2d_{\mathcal{C}_{u}^{G},\mathcal{C}_{v}^{L}}]\\
((\phi_{h \cdot \overline{\mathbf{c}}})_{!}K_{h \cdot \overline{\mathbf{c}}})_{h^{*} \rho} 
\cong \IC(\overline{{}^h\mathcal{C}_{u}^{G}},h^{*} \mathcal{F})[2d_{{}^{h} 
\mathcal{C}_{u}^{G},{}^{h} \mathcal{C}_{v}^{L}}]
\end{eqnarray*}
According to the characterization (3) of Theorem \ref{thmlus}, this shows that 
$\Sigma_{\ft}$ is $H$-equivariant.
\end{proof}

\section{Disconnected groups: the cuspidal case}
\label{sec:cusp}

First we recall Lusztig's classification of unipotent cuspidal pairs 
for a connected reductive group. \texttt{See Erratum in the appendix!}

\begin{thm}\label{thm:4.1} (Lusztig) \\
Let $G^\circ$ be a connected complex reductive group and write 
$Z = Z(G^\circ) / Z(G^\circ)^\circ$. 
\enuma{
\item Fix an Aut$(G^\circ)$-orbit $X$ of characters $Z \to \overline{\Q_\ell}^\times$.
There is at most one unipotent conjugacy class $\cC_u^{G^\circ}$ which carries a 
cuspidal local system on which $Z$ acts as an element of $X$. Moreover
$\cC_u^{G^\circ}$ is Aut$(G^\circ)$-stable and distinguished in $G^\circ$.
\item Every cuspidal local system $\cE$ on $\cC_u^{G^\circ}$ is uniquely determined
by the character by which $Z$ acts on it.
\item The dimension of the cuspidal representation $\cE_u$ of $A_{G^\circ} (u)$
is a power of two (possibly $2^0 = 1$). It is one if $G^\circ$ contains no factors 
which are isomorphic to spin or half-spin groups. 
} 
\end{thm}
\begin{proof}
In \cite[\S 2.10]{Lus1} it is explained how the classification can be reduced to 
simply connected, almost simple groups. Namely, first one notes that dividing out 
$Z(G^\circ)^\circ$ does not make an essential difference. Next everything is lifted 
to the simply connected cover $\tilde G$ of the semisimple group $G^\circ / 
Z(G^\circ)^\circ$. Since every automorphism of $G^\circ / Z(G^\circ)^\circ$ can be
lifted to one of $\tilde G$, the canonical image of $X$ is contained in a unique 
Aut$(\tilde G)$-orbit $\tilde X$ on $\Irr_{\overline{\Q_\ell}}(\tilde Z)$, where 
$\tilde Z$ is the $Z$ for $\tilde G$. Furthermore $\tilde G$ is a direct product of 
almost simple, simply connected groups, and $\tilde X$ decomposes as an analogous
product. Therefore it suffices to establish the theorem for simple, simply connected 
groups $G_{\sc}$.

(a) and (b) are shown in the case-by-case calculations in \cite[\S 10 and \S 14--15]{Lus1}.
But (a) is not made explicit there, so let us comment on it. There are only few cases in
which one really needs an Aut$(G_\sc)$-orbit $X_\sc$ in $\Irr_{\overline{\Q_\ell}}(Z(G_\sc))$.
Namely, only the spin groups $\mr{Spin}_N (\C)$ where $N > 1$ is simultaneously a square and 
a triangular number. These groups have precisely two unipotent conjugacy classes, say
$\cC_+$ and $\cC_-$, that carry a cuspidal local system. Let $\{1,-1\}$ be the kernel of 
$\mr{Spin}_N (\C) \to SO_N (\C)$, a characteristic subgroup of $\mr{Spin}_N (\C)$.
Lusztig's classification shows that $-1$ acts as $\epsilon$ on every cuspidal local system
supported on $\cC_\epsilon$. As $-1$ is fixed by Aut$(G_\sc)$, $X_\sc$ determines a unique
character of $\{1,-1\}$ and thus specifies $\cC_+$ or $\cC_-$.

This proves the first part of (a). For the second part, we notice that every algebraic 
automorphism of $G^\circ$ maps a cuspidal local system on a unipotent conjugacy class in 
$G^\circ$ with $Z$-character in $X$ to another such local system. By 
\cite[Proposition 2.8]{Lus1} every such conjugacy class is distinguished in $G^\circ$. 

(c) is obvious in types $A_n, C_n$ and $E_6$, for then $A_{G_{\sc}} (u)$
is abelian. For the root systems $E_8, F_4$ and $G_2$, $A_{G_{\sc}} (u)$ is a symmetric 
group and $\cE_u$ is the sign representation \cite[\S 15]{Lus1}. In type $E_7$ 
\cite[Table 9]{Miz} shows that $A_{G_{\sc}} (u) \cong S_3 \times S_2$. According
to \cite[\S 15.6]{Lus1}, $(\cE)_u$ again has dimension one (it is the tensor product
of the sign representations of $S_3$ and $S_2$).

In types $B_n$ and $D_n ,\; G_\sc = \mathrm{Spin}_N (\C)$ is a spin group. All the
cuspidal local systems $\cE$ for which the action of $Z(G_\sc)$ factors through 
$Z(\SO_N (\C))$ are one-dimensional, for $A_{\SO_N (\C)}(u)$ is abelian.
If the character by which $Z(G_\sc)$ acts on $\cE$ is not of this kind, then
\cite[Proposition 14.4]{Lus1} says that dim$(\cE_u)$ is a power of two. In that case
the original $G^\circ$ has an almost direct factor isomorphic to $\mr{Spin}_N (\C)$ or 
to a half-spin group $\mathrm{HSpin}_N (\C) = \mathrm{Spin}_N (\C) / \{1,\omega\}$ 
(here $N \in 4 \N$ and $\omega \in Z(\mathrm{Spin}_N (\C)) \setminus \{1,-1\}$). 
\end{proof}

Let $G$ be a disconnected complex reductive group with neutral component $G^\circ$.
We want to classify unipotent cuspidal pairs for $G$ in terms of those for $G^\circ$.

First we define them properly. For $u \in G^\circ$ we call an irreducible 
representation of $A_G (u)$ cuspidal if its restriction to $A_{G^\circ}(u)$ is a 
direct sum of irreducible cuspidal $A_{G^\circ}(u)$-representations. The set of
irreducible cuspidal representations of $A_G (u)$ (over $\overline{\Q_{\ell}}$) is
denoted by $\Irr_\cusp (A_G (u))$. We write
\[
\cN_G^0 = \{ (u,\eta) : u \in G \text{ unipotent, } \eta \in \Irr_\cusp (A_G (u)) \} 
/ G\text{-conjugacy} .
\]
Notice that the unipotency forces $u \in G^\circ$. Every $(u,\eta) \in \cN_G^0$
gives rise to a unique $G$-equivariant local system $\cF$ on $\cC_u^G$. We call any
$G$-equivariant local system on $\cC_u^G$ cuspidal if and only if it arises in this
way. Thus we may identify $\cN_G^0$ with the set of pairs $(\cC_u^G,\cF)$ where
$\cC_u^G$ is a unipotent conjugacy class in $G$ and $\cF$ is a cuspidal local
system on it.
For example, if $G^\circ$ is a torus, then $u=1$ and every irreducible representation
of $A_G (u) = G / G^\circ$ is cuspidal.

It follows from \eqref{eq:1.6} that there is a bijection
\[
\Irr_\cusp (A_G (u)) \longleftrightarrow \Big( \Irr_\cusp (A_{G^\circ}(u)) \q 
A_G (u) / A_{G^\circ}(u) \Big)_\kappa .
\]
So we want to identify the 2-cocycles $\kappa_\epsilon$ for $\epsilon \in 
\Irr_\cusp (A_{G^\circ}(u))$.

We note that there are natural isomorphisms
\begin{equation}\label{eq:4.1}
A_G (u) / A_{G^\circ} (u) \leftarrow Z_G (u) / Z_{G^\circ}(u) \to G / G^\circ . 
\end{equation}
In fact Theorem \ref{thm:4.1}.a implies that $\cC_u^G = \cC_u^{G^\circ}$,
which accounts for the surjectivity of the map to the right.

Recall from \cite[Lemma 4.2]{ABPS5} that the short exact sequence
\begin{equation}\label{eq:4.5}
1 \to \pi_0 (Z_{G^\circ}(u) / Z (G^\circ)) \to \pi_0 (Z_G (u) / Z(G^\circ)) 
\to G / G^\circ \to 1
\end{equation}
is split. However, the short exact sequence
\begin{equation}\label{eq:4.9}
1 \to \pi_0 (Z_{G^\circ}(u) / Z (G^\circ)^\circ ) \to 
\pi_0 (Z_G (u) / Z(G^\circ)^\circ) \to G / G^\circ \to 1
\end{equation}
need not be split. We choose a map 
\begin{equation}\label{eq:4.10}
s : G / G^\circ \to Z_G (u)  
\end{equation}
such that the induced map $G / G^\circ \to \pi_0 (Z_G (u) / Z(G^\circ))$ 
is a group homomorphism that splits \eqref{eq:4.5}. The proof of 
\cite[Lemma 4.2]{ABPS5} shows that we can take $s (g G^\circ)$ in $Z_G (G^\circ)$ 
whenever the conjugation action of $g$ on $G^\circ$ is an
inner automorphism of $G^\circ$. For all $\gamma,\gamma' \in G / G^\circ$
\[
s(\gamma) s(\gamma') s(\gamma \gamma')^{-1} \in Z(G^\circ) Z_{G^\circ}(u)^\circ ,
\]
because it represents the neutral element of $\pi_0 (Z_G (u) / Z(G^\circ))$.

Let $(\cC_u^{G^\circ},\cE) \in \cN_{G^\circ}^0$. 
The group $Z_{G^\circ}(u)^\circ$ acts trivially on $\epsilon = \cE_u$ and by 
cuspidality $Z(G^\circ) \subset Z(L)$ acts according to a character. Therefore
\begin{equation}\label{eq:4.6}
\natural_\cE (\gamma,\gamma') := \epsilon \big( s(\gamma) s(\gamma') 
s(\gamma \gamma')^{-1} \big)  
\end{equation}
lies in $\overline{\Q_{\ell}}^\times$. Comparing with \eqref{eq:1.1}, 
one checks easily that 
\begin{equation}\label{eq:4.12}
\natural_\cE : G / G^\circ \times G / G^\circ \to \overline{\Q_{\ell}}^\times
\end{equation}
is a 2-cocycle. We note that another element $u' \in \cC_u^G$ would give the same 
cocycle: just conjugate $s$ with a $g \in G^\circ$ such that $g u g^{-1} = u'$ and use 
the same formulas. Although $\natural_\cE$ depends on the choice of $s$, its 
class in $H^2 (G/G^\circ, \overline{\Q_{\ell}}^\times)$ does not. Indeed, 
suppose that $s'$ is another splitting as in \eqref{eq:4.10}. Since $s' (\gamma)$ 
and $s(\gamma)$ represent the same element of $\pi_0 (Z_G (u) / Z(G^\circ))$, there exist 
\[
z(\gamma) \in Z(G^\circ) \text{ such that } s' (\gamma) s(\gamma)^{-1} \in 
z(\gamma) Z_{G^\circ}(u)^\circ .
\]
As $Z_{G^\circ}(u)^\circ$ is normal in $Z_G (u)$ and contained in the kernel of $\epsilon$,
\begin{multline*}
\epsilon \big( s'(\gamma) s'(\gamma') s'(\gamma \gamma')^{-1} \big) =
\epsilon \big( s(\gamma) z(\gamma) s(\gamma') z(\gamma') s(\gamma \gamma')^{-1} 
z(\gamma \gamma')^{-1} \big) = \\ \natural_\cE (\gamma,\gamma') \,
\epsilon (z(\gamma)) \epsilon( z(\gamma')) \epsilon \big( z(\gamma \gamma')^{-1} \big) . 
\end{multline*}
Therefore $s'$ gives rise to a 2-cocycle that differs from \eqref{eq:4.12} by a coboundary,
and the cohomology class of $\natural_\cE$ depends only on $\cE$. 
Via the isomorphism \eqref{eq:4.1} we also get a 2-cocycle
\[
\natural_\cE : A_G (u) / A_{G^\circ} (u) \times A_G (u) / A_{G^\circ} (u) 
\to \overline{\Q_{\ell}}^\times .
\]
It will turn out that the 2-cocycles $\natural_\cE$ are trivial in many cases,
in particular whenever $Z(G^\circ)$ acts trivially on $\cE$. But sometimes
their cohomology class is nontrivial. 

\begin{ex}\label{ex:A}
Consider the following subgroup of $\SL_2 (\C)^5$:
\[
Q = \big\{ (\pm I_2) \times I_8, \matje{\pm i}{0}{0}{\mp i} \times I_4 \times -I_4,
\matje{0}{\pm i}{\pm i}{0} \times I_2 \times -I_2 \times I_2 \times -I_2,
\matje{0}{\mp 1}{\pm 1}{0} \times I_2 \times -I_4 \times I_2 \big\}
\]
It is isomorphic to the quaternion group of order 8. We take
$G = N_{\SL_{10}(\C)}(Q)$. Then
\begin{align*}
& G^\circ = Z_{\SL_{10}(\C)}(Q) = 
\big( Z(\GL_2 (\C)) \times \GL_2 (\C)^4 \big) \cap \SL_{10}(\C) , \\
& Z(G^\circ) = \Big\{ (z_j)_{j=1}^5 \in 
Z (\GL_2 (\C))^5 : \prod\nolimits_{j=1}^5 z_j^2 = 1 \Big\} .
\end{align*}
By \cite[\S 10.1--10.3]{Lus1} there exists a unique cuspidal pair for $G^\circ$,
namely \\ $(u = I_2 \times \matje{1}{1}{0}{1}^{\otimes 4}, \epsilon)$
with $\epsilon$ the nontrivial character of                                             
\[
A_{G^\circ}(u) = Z(G^\circ) / Z(G^\circ)^\circ \cong \{ \pm 1 \}.
\]
We note that the canonical map $Q \to A_G (u)$ is an isomorphism and that
\[
G / G^\circ \cong A_G (u) / A_{G^\circ}(u) \cong Q / \{\pm 1\} \cong (\Z / 2 \Z)^2. 
\]
There is a unique irreducible representation of $A_G (u)$ whose restriction to
$A_{G^\circ}(u)$ contains $\epsilon$, and it has dimension 2.

The group $S_5$ acts on $\GL_2 (\C)^5$ by permutations. Let $P_\sigma \in 
\GL_{10}(\C)$ be the matrix corresponding to a permutation $\sigma \in S_5$. 
Representatives for $G / G^\circ$ in $Z_G (u)$ are
\begin{equation}\label{eq:4.8}
\big\{ 1, \matje{i}{0}{0}{-i} P_{(23)(45)}, \matje{0}{i}{i}{0} P_{(24)(35)},
\matje{0}{-1}{1}{0} P_{(25)(34)} \big\} .
\end{equation}
The elements \eqref{eq:4.8} provide a splitting of \eqref{eq:4.5}, but \eqref{eq:4.9} 
is not split in this case. Then $\natural_\cE$ is the nontrivial cocycle of 
$G / G^\circ$ determined by the 2-dimensional projective representation with image 
$\big\{1, \matje{i}{0}{0}{-i}, \matje{0}{i}{i}{0},\matje{0}{-1}{1}{0} \big\}$. 

The twisted group algebra $\overline{\Q_{\ell}}[G / G^\circ,\natural_\cE]$ is isomorphic
with $M_2 (\overline{\Q_{\ell}})$. In particular it has precisely one irreducible
representation. This agrees with the number of representations of $A_G (u)$ that 
we want to obtain. Notice that, without the twisting, $\overline{\Q_{\ell}}[G/G^\circ]$
would have four inequivalent irreducible representations, too many for this situation.
\end{ex}

We return to our general setup.
Let $G_\cE$ be the subgroup of $G$ that stabilizes $\cE$ (up to isomorphism).
It contains $G^\circ$ and by Theorem \ref{thm:4.1}.b it coincides with the stabilizer 
of the $Z(G^\circ)$-character of $\cE$. By \eqref{eq:4.1} there are group
isomorphisms
\begin{equation}\label{eq:4.3}
A_G (u)_\epsilon / A_{G^\circ}(u) \leftarrow Z_G (u)_\epsilon / Z_{G^\circ}(u)
\to G_\cE / G^\circ .
\end{equation}

\begin{lem}\label{lem:4.2}
Let $(u,\epsilon) \in \cN_{G^\circ}^0$. Then we can take $\kappa_\epsilon = 
\natural_\cE^{-1}$ as 2-cocycles of $A_G (u)_\epsilon / A_{G^\circ}(u)$.
\end{lem}
\begin{proof}
With \eqref{eq:4.3} we translate the lemma to a statement about cocycles of \\
$Z_G (u)_\ep / Z_{G^\circ}(u)$. For $g \in Z_G (u)_\epsilon$ we have to find 
$I^g_\epsilon : V_\epsilon \to V_\epsilon$ such that
\begin{equation}\label{eq:3.1}
I^g_\epsilon \circ \epsilon (h) \circ (I^g_\epsilon)^{-1} = 
\epsilon (g h g^{-1}) \qquad \forall h \in Z_{G^\circ}(u) . 
\end{equation}
Since $Z_G (u) = s(G / G^\circ) Z_{G^\circ}(u)$, it suffices to find 
$I^{s (\gamma)}_\ep$ for $\gamma \in G_\cE / G^\circ$. Namely, then we can put
$I^{s(\gamma) h}_\epsilon = I^{s (\gamma)}_\epsilon \circ \epsilon (h)$ for
$h \in Z_{G^\circ}(u)$, as in \eqref{eq:1.7}.

Let us consider $(\cC_u^{G^\circ},\cE)$ as a cuspidal local system for the
simply connected cover $G_{\sc}$ of $G^\circ / Z(G^\circ)^\circ$. The action of
$G$ on $G^\circ$ by conjugation lifts to an action on $G_{\sc}$ and $Z(G^\circ)^\circ$
acts trivially on $\epsilon$. Hence it suffices to construct $I^{s(\gamma)}_\epsilon$
for $\epsilon$ as a representation of $A_{G_{\sc}}(u)$. 

Then $(A_{G_{\sc}}(u),\epsilon)$
decomposes as a direct product over almost simple factors of $G_{\sc}$. Factors
with different cuspidal local systems have no interaction, so we may assume that 
$G_{\sc} = H^n , \epsilon = \sigma^{\otimes n}$ with $H$ simply connected and almost
simple. The conjugation action of $G$ on $H^n$ is a combination of permutations of
$\{1,2,\ldots,n\}$ and automorphisms of $H$. If $g \in G$ permutes the factors of 
$H^n$ according to $\tau \in S_n$, then we can construct $I^g_{\sigma^{\otimes n}}$ 
as the permutation $\tau$ of $V_\sigma^{\otimes n}$, combined with some automorphisms
of the vector space $V_\sigma$. In this way we reduce to the case where $G_{\sc}$ is
almost simple. 

Whenever $\epsilon$ is one-dimensional, we simply put 
\begin{equation}\label{eq:3.2}
I^\gamma_\epsilon = I^g_\epsilon = \mathrm{Id}_{V_\epsilon} \quad \text{for} 
\quad g = s(\gamma) \in s (G_\cE / G^\circ) .
\end{equation}
To deal with the remaining cases, we recall from Theorem \ref{thm:4.1}.c that in
all those instances $G_\sc = \mathrm{Spin}_N (\C)$ is a spin group and that the
action of $Z(G_\sc)$ on $\cE$ does not factor through $Z(\SO_N (\C))$.

Suppose first that $N \geq 3$ is odd. Then $G_\sc$ is of type $B_{(N-1)/2}$ and all
its automorphisms are inner. As explained after \eqref{eq:4.10}, we can take
$s (G_\cE / G^\circ)$ in $Z_{G_\cE}(G^\circ)$. Thus \eqref{eq:3.1} can be fulfilled
by defining $I^g_\epsilon = \mathrm{Id}_{V_\epsilon}$. 

Next we suppose that $N$ is even, so $G_\sc$ is of type $D_{N/2}$. By 
\cite[Proposition 14.6]{Lus1} $N = j (j+1) / 2$ for some $j \geq 2$, and in particular
$G_\sc$ is not isomorphic to the group $\mathrm{Spin}_8 (\C)$ of type $D_4$. 
Let us write $Z(G_\sc) = \{1,-1,\omega,-\omega\}$, where
\[
\{1,-1\} = \ker \big( \mathrm{Spin}_N (\C) \to \SO_N (\C) \big) . 
\]
Our assumptions entail that $\epsilon (-1) \neq 1$. For both characters of $Z(G_\sc)$
with $\epsilon (-1) = -1$ there is exactly one cuspidal pair $(\cC_u^{G_\sc},\cE)$ on 
which $Z (G_\sc)$ acts in this way \cite[Proposition 14.6]{Lus1}. The group of outer 
automorphisms of $G_\sc$ has precisely two elements. It interchanges $\omega$ and
$-\omega$, and hence it interchanges the two cuspidal pairs in question. Therefore 
the conjugation action of $G_\cE$ on $G_\sc$ is by inner automorphisms
of $G_\sc$. Now the same argument as in the $N$ odd case shows that we may take
$I^g_\epsilon = \mathrm{Id}_{V_\epsilon}$. 

Thus \eqref{eq:3.2} works in all cases under consideration. The defining property 
of $s$ entails that
\begin{equation}\label{eq:4.2}
I^\gamma_\epsilon \circ I^{\gamma'}_\epsilon = \natural_\cE (\gamma,\gamma') 
I^{\gamma \gamma'}_\epsilon .
\end{equation}
Together \eqref{eq:4.3} this shows that the lemma holds when $G_{\sc}$ is almost 
simple. In view of our earlier reduction steps, that implies the general case.
\end{proof}

Notice that $Y := \cC_u^G Z(G^\circ)^\circ$ is a union of $G$-conjugacy classes in 
$G^\circ$. Tensoring $\cE$ with the constant sheaf on $Z(G^\circ)^\circ$, we obtain a 
$G^\circ$-equivariant cuspidal local system on $Y$. We also denote that by $\cE$.

Next we build a $G$-equivariant local system on $Y$ which contains 
every extension of $\epsilon$ to an irreducible representation of $A_G (u)$. 
The construction is the same as in \cite[\S 3.2]{Lus1}, only for a disconnected group. 
Via the map
\begin{equation}\label{eq:4.13}
Y \times G \to Y : (y,g) \mapsto g^{-1} y g 
\end{equation}
we pull $\cE$ back to a local system $\hat{\cE}$ on $Y \times G$.
It is $G \times G^\circ$-equivariant for the action
\begin{equation}\label{eq:4.14}
(h_1,h_0) \cdot (y,g) = (h_1 y h_1^{-1},h_1 g h_0^{-1}) . 
\end{equation}
The $G^\circ$ action is free, so we can divide it out and obtain a 
$G$-equivariant local system $\tilde{\cE}$ on $Y \times G / G^\circ$ such that its pull 
back under the natural quotient map is isomorphic to $\hat{\cE}$, see \cite[2.6.3]{BeLu}.
Let $\pi : Y \times G / G^\circ \to Y$ be the projection on the first coordinate. It is a
$G$-equivariant fibration, so the direct image $\pi_* \tilde{\cE}$ is a $G$-equivariant
local system on $Y$. With \eqref{eq:4.1} we see that its stalk at $y \in Y$ 
is isomorphic, as $Z_{G^\circ}(y)$-representation, to
\[
\bigoplus_{g \in G / G^\circ} (\tilde \cE )_{y,g G^\circ} \cong
\bigoplus_{g \in Z_G (y) / Z_{G^\circ}(y)} (\cE _{g^{-1} y g}) \cong
\bigoplus_{g \in Z_G (y) / Z_{G^\circ}(y)} g \cdot (\cE )_y .
\]
The elements of $Z_G (y)$ permute these subspaces $\cE_y$ in the expected way, so
\begin{equation}\label{eq:4.4}
(\pi_* \tilde{\cE})_y \cong \ind_{Z_{G^\circ}(y)}^{Z_G (y)} (\cE)_y
\quad \text{as} \; Z_G(y)\text{-representations.}
\end{equation}
In other words, we can consider $\pi_* \tilde{\cE}$ as the induction of $(\cE )_Y$
from $G^\circ$ to $G$.

\begin{lem}\label{lem:4.3}
The $G$-endomorphism algebra of $\pi_* \tilde{\cE}$ is isomorphic with
$\overline{\Q_{\ell}}[G_\cE / G^\circ,\natural_\cE]$. Once $\natural_\cE$ has been chosen, 
the isomorphism is canonical up to twisting by characters of $G_\cE / G^\circ$.
\end{lem}
\begin{proof}
By \cite[Proposition 3.5]{Lus1}, which applies also in the disconnected case, \\
$\End_G (\pi_* \tilde{\cE})$ is canonically a direct sum of one-dimensional
subspaces $\mathcal A_\gamma$ with $\gamma \in G_\cE / G^\circ$. We need to
specify one element in each of these subspaces to obtain a twisted group algebra.
Recall the isomorphisms \eqref{eq:4.3} and the map $s$ from \eqref{eq:4.10}. 
For $g = s(\gamma) \in s (G_\cE / G^\circ)$ we define 
\[
I^\gamma_\epsilon = I^g_\epsilon : (\cE)_u \to (\cE )_u
\]
as in the proof of Lemma \ref{lem:4.2}. We already saw in \eqref{eq:4.2} that
the $I^\gamma_\epsilon$ span an algebra isomorphic to $\overline{\Q_{\ell}}[G_\cE / 
G^\circ,\natural_\cE]$. Each $I^\gamma_\epsilon$ extends uniquely to an 
isomorphism of $G$-equivariant local systems 
\begin{equation}\label{eq:4.17}
I^\gamma_\cE : (\cE )_Y \to \mathrm{Ad}(\gamma)^* (\cE )_Y. 
\end{equation}
We can consider this as a family of $\overline{\Q_{\ell}}$-linear maps 
\[
I^\gamma_\cE : (\hat \cE )_{(y,g)} = (\cE )_{g^{-1} y g} \to
(\hat \cE )_{(y,g \gamma^{-1})} = (\cE )_{\gamma g^{-1} y g \gamma^{-1}} .
\]
Consequently the $I^\gamma_\cE$ induce automorphisms of $\hat \cE$, of 
$\tilde \cE$ and of $\pi_* \tilde{\cE}$. The latter automorphism belongs to
$\mathcal A_\gamma$ and we take it as element of the
required basis of $\End_G (\pi_* \tilde{\cE})$.

Any other choice of an isomorphism as in the lemma would differ from the first
one by an automorphism of $\overline{\Q_{\ell}}[G_\cE / G^\circ,\natural_\cE]$
which stabilizes each of the subspaces $\overline{\Q_{\ell}} T_\gamma$. Every such
automorphism is induced by a character of $G_\cE / G^\circ$.
\end{proof}

We note that the isomorphism in Lemma \ref{lem:4.3} is in general not canonical,
because $s$ and the constructions in the proof of Lemma \ref{lem:4.2} are not. 
In the final result of this section, we complete the classification of unipotent
cuspidal local systems on $G$.

\begin{prop}\label{prop:4.4}
There exists a canonical bijection
\[
\begin{array}{ccc}
\Irr ( \End_G (\pi_* \tilde{\cE}) ) & \longleftrightarrow & 
\{ \cF : (\cC_u^G,\cF) \in \cN_G^0 ,\, \Res^G_{G^\circ} \cF \text{ contains } \cE \} \\
\rho & \mapsto & \Hom_{\End_G (\pi_* \tilde{\cE})}(\rho,\pi_* \tilde{\cE}) \\
\Hom_G (\cF, \pi_* \tilde \cE) & \reflectbox{$\mapsto$} & \cF
\end{array}
\]
Upon choosing an isomorphism as in Lemma \ref{lem:4.3}, we obtain a bijection
\[
\Irr ( \overline{\Q_{\ell}}[G_\cE/G^\circ,\natural_\cE] ) \longleftrightarrow
\{ (u,\eta) \in \cN_G^0 : \Res^{A_G (u)}_{A_{G^\circ}(u)} \eta 
\text{ contains } (\cE )_u \} .
\]
\end{prop}
\begin{proof}
The first map is canonical because its definition does not involve any arbitrary
choices. To show that it is a bijection, we fix an isomorphism as in Lemma 
\ref{lem:4.3}. By $G$-equivariance, it suffices to consider the claims at the stalk
over $u$. Then we must look for irreducible $A_G (u)$-representations that contain
$\epsilon$. By \eqref{eq:4.3}, \eqref{eq:4.4} and Proposition \ref{prop:1.1}.b
\[
(\pi_* \tilde{\cE} )_u \cong \ind_{A_{G^\circ}(u)}^{A_G (u)} (\cE )_u \cong
\ind_{A_G (u)_\epsilon}^{A_G (u)} \big( \overline{\Q_{\ell}}[A_G (u)_\epsilon /
A_{G^\circ}(u),\kappa_\epsilon] \otimes \epsilon \big) .
\]
By Lemma \ref{lem:4.2} the right hand side is
\begin{equation}\label{eq:4.7}
\ind_{A_G (u)_\epsilon}^{A_G (u)} \big( \overline{\Q_{\ell}}[A_G (u)_\epsilon /
A_{G^\circ}(u),\natural_\cE^{-1}] \otimes \epsilon \big) .  
\end{equation}
By Frobenius reciprocity and the definition of $A_G (u)_\epsilon$, the 
$A_G (u)$-endomorphism algebra of \eqref{eq:4.7} is
\begin{equation}\label{eq:4.11}
\End_{A_G (u)_\epsilon} \big( \overline{\Q_{\ell}}[A_G (u)_\epsilon /
A_{G^\circ}(u),\natural_\cE^{-1}] \otimes \epsilon \big) .
\end{equation}
The description of the $A_G (u)_\epsilon$-action in Proposition \ref{prop:1.1}.a shows 
that it is \\
$\overline{\Q_{\ell}}[A_G (u)_\epsilon / A_{G^\circ}(u), \natural_\cE^{-1}]^{op}$, 
acting by multiplication from the right.
By \eqref{eq:4.3} and Lemma \ref{lem:1.3}.a, \eqref{eq:4.11} can be 
identified with $\overline{\Q_{\ell}}[G_\cE/G^\circ,\natural_\cE]$. Lemma \ref{lem:4.3} 
shows that this matches precisely with $\End_G (\pi_* \tilde{\cE})$. With
\eqref{eq:4.7} and Lemma \ref{lem:1.3}.c it follows that
\begin{align}\label{eq:3.3}
& \big( \pi_* \tilde \cE \big)_u \cong \bigoplus\nolimits_{\rho \in 
\Irr ( \End_G (\pi_* \tilde{\cE}) )}
\rho \otimes \ind_{A_G (u)_\epsilon}^{A_G (u)} (\rho^* \otimes \epsilon) , \\
\nonumber & \Hom_{\End_G (\pi_* \tilde{\cE})}(\rho, (\pi_* \tilde{\cE})_u ) \cong
\ind_{A_G (u)_\epsilon}^{A_G (u)} (\rho^* \otimes \epsilon) , 
\end{align}
where $\rho^* \in \Irr (\overline{\Q_{\ell}}[A_G (u)_\epsilon /
A_{G^\circ}(u),\kappa_\epsilon])$ is the contragredient of $\rho$. By Lemma
\ref{lem:1.3}.b and Proposition \ref{prop:1.1}.c every irreducible 
$A_G (u)$-representation containing $\epsilon$ is of the form 
$\ind_{A_G (u)_\epsilon}^{A_G (u)} (\rho^* \otimes \epsilon)$, for a unique
$\rho \in \Irr ( \End_G (\pi_* \tilde{\cE}) )$. Hence the maps from left to right
in the statement are bijective.

Let $(\cC_u^G,\cF) \in \cN_G^0$ be such that $\Res^G_{G^\circ} \cF$ contains $\cE$.
By what we have just shown, $\cF_u \cong \ind_{A_G (u)_\epsilon}^{A_G (u)} 
(\rho^* \otimes \epsilon)$, for a unique $\rho$. By \eqref{eq:3.3}
\[
\Hom_G (\cF, \pi_* \tilde \cE) = 
\Hom_{Z_G (u)} \big( \cF_u, ( \pi_* \tilde \cE )_u \big) =
\Hom_{A_G (u)} \big( \ind_{A_G (u)_\epsilon}^{A_G (u)} 
(\rho^* \otimes \epsilon) ,( \pi_* \tilde \cE )_u \big) \cong \rho ,
\]
which provides the formula for the inverse of the above bijection.
\end{proof}

\section{Disconnected groups: the non-cuspidal case}
\label{sec:noncusp}

We would like to extend the generalized Springer correspondence
for $G^\circ$ to $G$. First define the source and target properly.

\begin{defn}\label{defn:5.1}
For $\cN_G^+$ we use exactly the same definition as in the connected case:
\[
\cN_G^+ = \{ (u,\eta) : u \in G \text{ unipotent }, 
\eta \in \Irr_{\overline{\Q_{\ell}}} (A_G (u)) \} / G\text{-conjugacy} .
\]
As $\cS_G$ we take the same set as for $G^\circ$, but now considered up to
$G$-conjugacy:
\[
\cS_G = \{ \text{unipotent cuspidal supports for } G^\circ \}
/ G\text{-conjugacy} .
\]
For $\ft = [L,\cC_v^L,\cE]_G \in \cS_G$, let $N_G (\ft)$ be the stabilizer of 
$(L,\cC_v^L,\cE)$ in $G$. We define $W_\ft$ as the component group of $N_G (\ft)$.
\end{defn}

In the above notations, the group $L$ stabilizes $(L,\cC_v^L,\cE)$ and any 
element of $G$ which stabilizes $(L,\cC_v^L,\cE)$ must normalize $L$. Hence 
$L$ is the neutral component of $N_G (\ft)$ and $W_\ft = N_G (\ft) / L$ is a 
subgroup of $W(G,L) = N_G (L) / L$. 

As in the connected case, $\cN_G^+$ is canonically in bijection with the
set of pairs $(\cC^G_u,\cF)$, where $\cC^G_u$ is the $G$-conjugacy class of
a unipotent element $u$ and $\cF$ is an irreducible $G$-equivariant 
local system on $\cC^G_u$.

We define a map $\Psi_G : \cN_G^+ \to \cS_G$ in the following way. Let
$(u,\eta) \in \cN_G^+$. With Theorem \ref{thm:1.2} we can write
$\eta = \eta^\circ \rtimes \tau$ with $\eta^\circ \in \Irr (A_{G^\circ}(u))$.
Moreover $\eta^\circ$ is uniquely determined by $\eta$ up to $A_G (u)$-conjugacy.
Then $(u,\eta^\circ) \in \cN_{G^\circ}^+$. Using \eqref{eq:2.1} we put
\begin{equation}\label{eq:5.1}
\Psi_G (u,\eta) := \Psi_{G^\circ} (u,\eta^\circ) / G\text{-conjugacy}
\end{equation}
By the $G$-equivariance of $\Psi_{G^\circ}$ (Theorem \ref{thm:3.2}), 
$\Psi_G (u,\eta)$ does not depend on the choice of $\eta^\circ$. Write 
\[
\ft^\circ = [L,\cC_v^L,\cE]_{G^\circ}
\]
and consider $\Sigma_{\ft^\circ} (u,\eta^\circ) \in \Irr (W_{\ft^\circ})$. 
Just as in \eqref{eq:5.1}, $(L,\cC_v^L,\cE,\Sigma_{\ft^\circ} (u,\eta^\circ))$ 
is uniquely determined by $(u,\eta)$, up to $G$-conjugacy.

We would like to define $\Sigma_\ft$ such that $\Sigma_\ft(u,\eta)$
is a representation of $W_\ft$ whose restriction to $W_{\ft^\circ}$ contains
$\Sigma_{\ft^\circ} (u,\eta^\circ)$. However, in general this does not work. 
It turns out that we have to twist the group algebra $\overline{\Q_{\ell}} [W_\ft]$ 
with a certain 2-cocycle, which is trivial on $W_{\ft^\circ}$. 
In fact we have already seen this in Section \ref{sec:cusp}. Over there
$L = G^\circ ,\; W_{\ft^\circ} = 1 ,\; W_\ft = G_\cE / G^\circ$ and in 
Example \ref{ex:A} the group algebra of $W_\ft$ had to be twisted by a 
nontrivial 2-cocycle.

This twisting by nontrivial cocycles is only caused by the relation between 
irreducible representations of $A_{G^\circ}(u)$ and $A_G (u)$. The next
two results show that the group $W_\ft$, considered on its own, would not
need such twisting.

\begin{lem}\label{lem:5.1}
There exists a subgroup $\cR_\ft \subset W_\ft$ such that
$W_\ft = \cR_\ft \ltimes W_{\ft^\circ}$. 
\end{lem}
\begin{proof}
Thanks to Theorem \ref{thm:2.2}, we know that $W_{\ft^\circ}$ equals 
$W(G^\circ,L)=N_{G^\circ}(L) / L$. On the other hand, $W_\ft \subset W(G,L)$ acts 
on the root system $R(G^\circ,Z(L)^\circ)$.
Fix a positive subsystem and let $\cR_\ft$ be its stabilizer in
$W_\ft$. Since $W(G^\circ,L)$ is the Weyl group of the root system 
$R(G^\circ,Z(L)^\circ)$ (see Theorem \ref{thm:2.2}), it acts simply transitively
on the collection of positive systems in $R(G^\circ,Z(L)^\circ)$. As
$W(G^\circ,L)$ is normal in $W(G,L)$, we obtain the decomposition of
$W_\ft$ as a semidirect product.
\end{proof}

\begin{prop}\label{prop:5.7}
Let $\pi \in \Irr_{\overline{\Q_{\ell}}} (W_{\ft^\circ})$. 
The cohomology class of $\kappa_\pi$ in \\ 
$H^2 (W_{\ft,\pi} / W_{\ft^\circ},\overline{\Q_{\ell}}^\times)$ is trivial. 
\end{prop}
\begin{proof}
This is the statement of \cite[Proposition 4.3]{ABPS5}, which is applicable
by Lemma \ref{lem:5.1}.
\end{proof}

Let $N_G^+ (\ft)$ be the inverse image of $\cR_\ft$ in $N_G (\ft) \subset N_G (L)$.
Then $L = N_G^+ (\ft)^\circ$ and $\cR_\ft \cong N_G^+ (\ft) / N_G^+ (\ft)^\circ$.
Thus $(\cC_v^L,\cE)$ can be considered as a cuspidal pair for $N_G^+ (\ft)^\circ$. 
In \eqref{eq:4.2} we constructed a 2-cocycle 
$\natural_\cE : \cR_\ft \times \cR_\ft \to \overline{\Q_{\ell}}^\times$.
With Lemma \ref{lem:5.1} we can also consider it as a 2-cocycle of
$W_\ft$, trivial on $W_{\ft^\circ}$:
\begin{equation}\label{eq:5.3}
\natural_\cE : W_\ft / W_{\ft^\circ} \times  W_\ft / W_{\ft^\circ}  
\to \overline{\Q_{\ell}}^\times .
\end{equation}

\begin{lem}\label{lem:5.2}
Let $\cF^\circ$ be the $G^\circ$-equivariant local system on $\cC_u^{G^\circ}$ 
corresponding to $\eta^\circ \in \Irr_{\overline{\Q_{\ell}}}(A_{G^\circ}(u))$.
There are natural isomorphisms 
\[
W_{\ft,\Sigma_{\ft^\circ}(u,\eta^\circ)} / W(G^\circ,L) \to 
G_{(\cC_u^{G^\circ},\cF^\circ)} / G^\circ \leftarrow 
Z_G (u)_{\eta^\circ} / Z_{G^\circ}(u) \to A_G (u)_{\eta^\circ} / A_{G^\circ}(u) .
\]
\end{lem}
\begin{proof}
There is a natural injection 
\begin{equation}\label{eq:5.4}
W_\ft / W(G^\circ,L) \cong N_G (L,\cC_v^L,\cE) / N_{G^\circ}(L) \to G / G^\circ .
\end{equation}
By Theorem \ref{thm:3.2} an element of $W_\ft / W(G^\circ,L)$ stabilizes 
$\Sigma_{\ft^\circ}(u,\eta^\circ) = \Sigma_{\ft^\circ} (\cC_u^{G^\circ},\cF^\circ)$ 
if and only if its image in 
$G / G^\circ$ stabilizes $(\cC_u^{G^\circ},\cF^\circ)$. The second isomorphism is 
a direct consequence of the relation between $\cF^\circ$ and $\eta^\circ$.
\end{proof}
With this lemma we transfer \eqref{eq:5.3} to a 2-cocycle
\begin{equation}\label{eq:5.8}
\natural_\cE : A_G (u)_{\eta^\circ} / A_{G^\circ}(u) \times
A_G (u)_{\eta^\circ} / A_{G^\circ}(u) \to \overline{\Q_{\ell}}^\times.
\end{equation}
Our construction of $\Sigma_\ft$ will generalize that of $\Sigma_{\ft^\circ}$ in
\cite{Lus1}, in particular we use similar equivariant local systems. 
Recall that $(L,\cC_v^L,\cE)$ is a cuspidal support. As in \cite[\S 3.1]{Lus1}
we put $S = \cC_v^L Z(L)^\circ$ and we extend $\cE$ to  a local system on $S$.
We say that an element $y \in S$ is regular if $Z_G (y_s)^\circ$, the connected 
centralizer of the semisimple part of $y$, is contained in $L$.
Consider the variety $Y = Y_{(L,S)}$ which is the union of all conjugacy classes
in $G$ that meet the set of regular elements $S_\reg$. We build equivariant local 
systems $\hat \cE$ on 
\[
\hat Y := \{ (y,g) \in Y \times G : g^{-1} y g \in S_\reg \}
\]
and $\tilde \cE$ on $\tilde Y := \hat Y / L$ as in 
\eqref{eq:4.13} and \eqref{eq:4.14}, only with $L$ instead of $G^\circ$. 
The projection map 
\[
\pi : \tilde Y \to Y ,\; (y,g) \mapsto y
\]
is a fibration with fibre $N_G (L) / L$, so
\begin{equation}\label{eq:5.20}
\pi : \tilde Y \to Y \text{ gives a $G$-equivariant local system }
\pi_* \tilde{\cE} \text{ on } Y .
\end{equation}
By Theorem \ref{thm:4.1}.a $N_G (L)$ stabilizes $\cC_v^L$, so $N_G (L) / L \cong 
Z_{N_G (L)} (v) / Z_L (v)$. The stalk of $\pi_* \tilde{\cE}$ at $y \in S_{\reg}$ 
is isomorphic, as representation of $Z_L (y) = Z_L (v)$, to
\begin{equation}\label{eq:5.21}
(\pi_* \tilde{\cE})_y \cong \bigoplus_{g \in N_G (L) / L} (\tilde{\cE})_{y,gL}
\cong \bigoplus_{g \in Z_{N_G (L)} (v) / Z_L (v)} \cE_{g^{-1} y g} \cong
\bigoplus_{g \in Z_{N_G (L)} (v) / Z_L (v)} g \cdot \cE_y .
\end{equation}
On the part $Y^\circ$ of $Y$ that is $G^\circ$-conjugate to $S_{\reg}$, $\pi_* 
\tilde{\cE}$ can also be considered as a $G^\circ$-equivariant local system.
As such $(\pi_* \tilde{\cE} )_{Y^\circ}$ contains the analogous local system
$\pi_* \tilde{\cE}^\circ$ for $G^\circ$ as a direct summand. 

The following result generalizes Lemma \ref{lem:4.3}.

\begin{prop}\label{prop:5.3}
The $G$-endomorphism algebra of $\pi_* \tilde{\cE}$ is isomorphic with
$\overline{\Q_{\ell}}[W_\ft,\natural_\cE]$. Once $\natural_\cE$ has been chosen
via \eqref{eq:4.6}, the isomorphism is canonical up to twisting by characters
of $W_\ft / W_{\ft^\circ}$.
\end{prop}
\begin{proof}
First we note that the results and proofs of \cite[\S 3]{Lus1} are also valid
for the disconnected group $G$. By \cite[Proposition 3.5]{Lus1} 
$\End_{\overline{\Q_{\ell}}}(\pi_* \tilde{\cE}) = \End_G (\pi_* \tilde{\cE})$, and 
according to \cite[Remark 3.6]{Lus1} it is a twisted group algebra of $W_\ft$. 
It remains to determine the 2-cocycle. Again by \cite[Proposition 3.5]{Lus1}, 
$\End_G (\pi_* \tilde{\cE})$ is naturally a direct sum of one-dimensional
$\overline{\Q_{\ell}}$-vector spaces $\mc A_{\cE,w} \; (w \in W_\ft)$. An element of
$\mc A_{\cE,w}$ consists of a system of $\overline{\Q_{\ell}}$-linear maps 
\begin{equation}\label{eq:5.22}
\tilde{\cE}_{y,g} = \cE_{g^{-1} y g} \to \tilde{\cE}_{y,g w^{-1}} =
\cE_{w g^{-1} y g w^{-1}}
\end{equation}
and is determined by a single $L$-intertwining map 
$\cE \to \mathrm{Ad}(w)^* \cE$.

For $w \in W_{\ft^\circ}$ any element $b_w \in \mc A_{\cE,w}$ also acts on 
$\pi_* \tilde{\cE}^\circ$. In \cite[Theorem 9.2.d]{Lus1} a canonical isomorphism
\begin{equation}\label{eq:5.5}
\End_{G^\circ} (\pi_* \tilde{\cE}^\circ) \cong \overline{\Q_{\ell}}[W_{\ft^\circ}] , 
\end{equation}
was constructed. Via this isomorphism we pick the $b_w \; (w \in W_{\ft^\circ})$,
then
\begin{equation}\label{eq:5.6}
w \mapsto b_w \text{ is a group homomorphism } W_{\ft^\circ} \to
\Aut_G (\pi_* \tilde{\cE}) .
\end{equation}
In view of Lemma \ref{lem:5.1}, we still to have find suitable $b_\gamma \in 
\mc A_{\cE,\gamma}$ for $\gamma \in \cR_\ft$. Let $n_\gamma \in N_G^+ (\ft)$ be
a lift of $\gamma \in N_G^+ (\ft) / L$. By \cite[\S 3.4-3.5]{Lus1} the
choice of $b_\gamma$ is equivalent to the choice of an automorphism 
$I^\gamma_\cE$ of $(\cE )_S$ that lifts the map
\[
S \to S : g \mapsto n_\gamma g n_\gamma^{-1} . 
\]
Precisely such an automorphism was constructed (with the group $N_G^+ (\ft)$ in
the role of $G$) in \eqref{eq:4.17}. We pick the unique $b_\gamma \in 
\mc A_{\cE,\gamma}$ corresponding to this $I^\gamma_\cE$. Then the multiplication
rules for the $b_\gamma$ are analogous to those for the $I^\gamma_\cE$, so by 
Lemma \ref{lem:4.3} we get
\begin{equation}\label{eq:5.7}
b_\gamma \cdot b_{\gamma'} = \natural_\cE (\gamma,\gamma') b_{\gamma \gamma'}
\qquad \gamma, \gamma' \in \cR_\ft .
\end{equation}
Using Lemma \ref{lem:5.1} we define $b_{\gamma w} = b_\gamma b_w$ for
$\gamma \in \cR_\ft, w \in W_{\ft^\circ}$. Now \eqref{eq:5.6} and \eqref{eq:5.7}
imply that $b_w \cdot b_{w'} = \natural_\cE (w,w') b_{w w'}$ for all
$w, w' \in W_\ft$.

The only noncanonical part in the construction of the above isomorphism is the
choice of the $b_\gamma \in \mc A_{\cE,\gamma}$ with $\gamma \in \cR_\ft$. 
Any other choice would differ from the above by an automorphism of 
$\overline{\Q_{\ell}}[\cR_\ft,\natural_\cE]$ which stabilizes each of the 
one-dimensional subspaces $\mc A_{\cE,w}$. Every such automorphism is induced 
by a character of $\cR_\ft \cong W_\ft / W_{\ft^\circ}$.
\end{proof}

Let $(u,\eta^\circ) \in \cN_{G^\circ}^+$.
Recall the cocycle $\kappa_{\eta^\circ}$ of $A_G (u)_{\eta^\circ} / A_{G^\circ}(u)$
constructed from $\eta^\circ \in \Irr (A_{G^\circ}(u))$ in \eqref{eq:1.4}. Like 
$\natural_\cE$ it depends on some choices, but its cohomology class does not.

\begin{lem}\label{lem:5.4}
We can choose $\kappa_{\eta^\circ}$ equal to $\natural_\cE^{-1}$ from \eqref{eq:5.8}.
\end{lem}
\begin{proof}
Let $s : G_{\cC_u^{G^\circ}} /G^\circ \to Z_G (u)$ be as in \eqref{eq:4.10}. 
As a $G^\circ$-equivariant local system on $Y^\circ$, 
\[
( \pi_* \tilde{\cE} )_{Y^\circ} = \bigoplus\nolimits_{\gamma \in s (G_{\cC_u^{G^\circ}} 
/ G^\circ)} \mathrm{Ad}(\gamma)^* (\pi_* \tilde{\cE}^\circ ) .
\]
Every summand is of the same type as $\pi_* \tilde{\cE}^\circ$, so we can apply all 
the constructions of \cite{Lus1} to $\pi_* \tilde{\cE}$. In particular we can build
\begin{equation}\label{eq:5.10}
\cH^{2d_C} \big( \IC (\overline{Y},\pi_* \tilde{\cE}) \big) | \cC_u^{G^\circ} 
\cong \bigoplus\nolimits_{\gamma \in s (G_{\cC_u^{G^\circ}} / G^\circ)} 
\mathrm{Ad}(\gamma)^* \cH^{2d_C} 
\big( \IC (\overline{Y^\circ},\pi_* \tilde{\cE}^\circ) \big) | \cC_u^{G^\circ} ,
\end{equation}
see \cite[Theorem 6.5]{Lus1}. Write 
\[
\rho^\circ = \Sigma_{\ft^\circ}(u,\eta^\circ) \in \Irr (W_{\ft^\circ}) .
\]
Let $d_C = d_{\cC_u^{G^\circ},\cC_v^L}$ be as in Theorem \ref{thmlus}. 
Then $A_{G^\circ}(u)$ acts on 
\[
\mathrm{Ad}(\gamma)^* V_{\eta^\circ} = \mathrm{Ad}(\gamma)^* \cH^{2d_C} \big( 
\IC (\overline{Y^\circ},\pi_* \tilde{\cE}^\circ)_{\rho^\circ} \big) | \cC_u^{G^\circ} 
\]
as $\gamma \cdot \eta^\circ$. Let $r (\gamma) \in \cR_\ft \cong W_\ft / W_{\ft^\circ}$
correspond to $\gamma G^\circ \in G / G^\circ$ under Lemma \ref{lem:5.2}. 
By construction $b_{r(\gamma)} \in \End_G (\pi_* \tilde{\cE})$ maps the 
$G^\circ$-local system $\mathrm{Ad}(\gamma)^* (\pi_* \tilde{\cE}^\circ )$ to
$\pi_* \tilde{\cE}^\circ$. Suppose that $\gamma$ stabilizes $\eta^\circ$.
For $I^\gamma_{\eta^\circ}$ we take the map
\[
\mathrm{Ad}(\gamma)^* \cH^{2d_C} \Big( \IC (\overline{Y^\circ},\pi_* 
\tilde{\cE}^\circ)_{\rho^\circ} \Big) | \cC_u^{G^\circ} \to
\cH^{2d_C} \Big( \IC (\overline{Y^\circ},\pi_* \tilde{\cE}^\circ)_{\rho^\circ} \Big)
| \cC_u^{G^\circ} 
\]
induced by $b_{r(\gamma)}$. It commutes with the action of $Z_G (u)$, so it can be
regarded as an element of $\Hom_{A_{G^\circ} (u)}(\gamma \cdot \eta^\circ,
\eta^\circ)$. Then
\begin{multline*}
\kappa_{\eta^\circ}^{-1}(\gamma,\gamma') = I^\gamma_{\eta^\circ} \circ 
I^{\gamma'}_{\eta^\circ} \circ (I^{\gamma \gamma'}_{\eta^\circ})^{-1} = 
b_{r(\gamma)} b_{r(\gamma')} b_{r(\gamma \gamma')}^{-1} =
\natural_\cE (r(\gamma),r(\gamma')) = \natural_\cE (\gamma,\gamma') ,
\end{multline*}
where we used \eqref{eq:5.7} for the third equality.
\end{proof}

Now we can state the main result of the first part of the paper.

\begin{thm}\label{thm:5.5}
Let $\ft = [L,\cC_v^L,\cE]_G \in \cS_G$. There exists a canonical bijection
\[
\begin{array}{cccc}
\Sigma_\ft : & \Psi_G^{-1}(\ft) & \to & \Irr (\End_G (\pi_* \tilde{\cE})) \\
& (\cC_u^G ,\cF) & \mapsto & \Hom_G \Big( \cF , \cH^{2d_C} 
\big( \IC (\overline{Y},\pi_* \tilde{\cE}) \big) | \cC_u^G \Big) .
\end{array}
\]
Suppose that $\rho \in \Irr \big( \End_G (\pi_* \tilde{\cE})\big)$ contains 
$\rho^\circ \in \Irr \big( \End_{G^\circ} (\pi_* \tilde{\cE}^\circ) \big)$ and 
that the unipotent conjugacy class of 
$\Sigma_{\ft^\circ}^{-1}(\rho^\circ)$ is represented by $u \in G^\circ$. Then
\[
\Sigma_\ft^{-1}(\rho) = \Big( \cC_u^G, \cH^{2d_C} 
\big( \IC (\overline{Y},\pi_* \tilde{\cE})_\rho \big) | \cC_u^G \Big) ,
\]
where $d_C$ is as in Theorem \ref{thmlus}.

Upon choosing an isomorphism as in Proposition \ref{prop:5.3}, we obtain a bijection
\[
\Psi_G^{-1}(\ft) \to \Irr ( \overline{\Q_{\ell}} [W_\ft,\natural_\cE]) .
\]
\end{thm}
\begin{proof}
First we show that there exists a bijection $\Sigma_\ft$ between the indicated sets. 
To this end we may fix an isomorphism 
\begin{equation}\label{eq:5.9}
\End_G (\pi_* \tilde{\cE}) \cong \overline{\Q_{\ell}} [W_\ft,\natural_\cE]
\end{equation}
as in Proposition \ref{prop:5.3}. In particular it restricts to
\[
\End_{G^\circ} (\pi_* \tilde{\cE}^\circ) \cong \overline{\Q_{\ell}}[W_{\ft^\circ}].
\]
Let us compare $\Psi_G^{-1}(\ft)$ with $\Psi_{G^\circ}^{-1}(\ft^\circ)$. For every
$(u,\eta^\circ) \in \Psi_{G^\circ}^{-1}(\ft^\circ)$ we can produce an element of 
$\Psi^{-1}_G (\ft)$ by extending $\eta^\circ$ to an irreducible representation 
$\eta$ of $A_G (u)$. By Lemma \ref{lem:5.4} and Proposition
\ref{prop:1.1}.c the only way to do so is taking $\eta$ of the form
\begin{equation}\label{eq:5.24}
\eta^\circ \rtimes \tau' = \ind_{A_G (u)_{\eta^\circ}}^{A_G (u)} 
(\eta^\circ \otimes \tau') \text{ with } \tau' \in \Irr \big (\overline{\Q_{\ell}}[
A_G (u)_{\eta^\circ} / A_{G^\circ} (u), \natural_\cE^{-1}] \big) .
\end{equation}
In view of Theorem \ref{thm:1.2} and Lemma \ref{lem:5.2} that yields a bijection
\begin{equation}\label{eq:5.11}
\begin{array}{ccc}
\big( \Psi_{G^\circ}^{-1}(\ft^\circ) \q \, W_\ft / W_{\ft^\circ} 
\big)_{\natural_\cE^{-1}} & \longleftrightarrow & \Psi_G^{-1}(\ft)  \\
((u,\eta^\circ),\tau') & \mapsto & (u,\eta^\circ \rtimes \tau') . 
\end{array}
\end{equation}
By Lemma \ref{lem:1.3}.b there is a bijection
\begin{equation}\label{eq:5.12}
\begin{array}{ccc}
\Irr \big (\overline{\Q_{\ell}}[W_\ft / W_{\ft^\circ}, \natural_\cE] \big) & 
\longleftrightarrow &
\Irr \big (\overline{\Q_{\ell}}[W_\ft / W_{\ft^\circ}, \natural_\cE^{-1}] \big) \\
V & \mapsto & V^* = \Hom_{\overline{\Q_{\ell}}}(V,\overline{\Q_{\ell}}) .
\end{array}
\end{equation}
Recall from Proposition \ref{prop:5.7} that for any $\rho^\circ \in 
\Irr_{\overline{\Q_{\ell}}} (W_{\ft^\circ})$ the cohomology class of $\kappa_{\rho^\circ}$ 
in $H^2 (W_{\ft,\rho^\circ} / W_{\ft^\circ},\overline{\Q_{\ell}}^\times)$ is
trivial. With Theorem \ref{thm:1.2} we get a bijection
\begin{equation}\label{eq:5.13}
\begin{array}{ccc}
\big( \Irr_{\overline{\Q_{\ell}}} (W_{\ft^\circ})) \q \, W_\ft / W_{\ft^\circ} 
\big)_{\natural_\cE} & \longleftrightarrow & 
\Irr (\overline{\Q_{\ell}}[W_\ft,\natural_\cE])  \\
(\rho^\circ,\tau) & \mapsto & \rho^\circ \rtimes \tau . 
\end{array} 
\end{equation}
From \eqref{eq:5.11}, \eqref{eq:5.12} and \eqref{eq:5.13} we obtain a bijection
\begin{equation}\label{eq:5.14}
\begin{array}{ccc}
\Psi_G^{-1}(\ft) & \longleftrightarrow &
\Irr (\overline{\Q_{\ell}}[W_\ft,\natural_\cE]) \\
(u,\eta^\circ \rtimes \tau') & \mapsto & 
\Sigma_{\ft^\circ}(u,\eta^\circ) \rtimes \tau'^* \\
\Sigma_{\ft^\circ}^{-1}(\rho^\circ) \rtimes \tau^* & 
\text{\reflectbox{$\mapsto$}} & \rho^\circ \rtimes \tau
\end{array}
\end{equation}
Together with \eqref{eq:5.9} we get a candidate for $\Sigma_\ft$, and we know that this
candidate is bijective. To prove that it is canonical, it suffices to see that it 
satisfies the given formula for $\Sigma_\ft^{-1}(\rho)$. That formula involves a 
$G$-equivariant local system on $\cC_u^G$. Since $\End_G (\pi_* \tilde \cE)$ is
semisimple, we only have to determine its stalk at $u$, as a $A_G (u)$-representations. 
It follows from \eqref{eq:5.10} that this stalk is
\[
\cH^{2d_C} \big( \IC (\overline{Y},\pi_* \tilde{\cE}) \big)_u 
\cong \ind_{A_{G^\circ}(u)}^{A_G (u)} \cH^{2d_C} \big( \IC 
(\overline{Y^\circ},\pi_* \tilde{\cE}^\circ) \big)_u .
\]
We abbreviate $\Sigma (u) = \{ \rho^\circ = \Sigma_{\ft^\circ}(u,\eta^\circ) : 
(u,\eta^\circ) \in \Psi_{G^\circ}^{-1}(\ft^\circ)\}$. Decomposing \\$\cH^{2d_C} \big( 
\IC (\overline{Y^\circ},\pi_* \tilde{\cE}^\circ) \big)_u$ as a representation of 
$\overline{\Q_{\ell}} [A_{G^\circ} (u)] \times \End_{G^\circ}(\pi_* (\tilde{\cE}^\circ))$, 
like in \cite[\S 3.7]{Lus1}, the right hand side becomes 
\begin{equation}\label{eq:5.15}
\ind_{A_{G^\circ}(u)}^{A_G (u)} \Big( \bigoplus\nolimits_{\rho^\circ \in \Sigma (u)} 
V_{\eta^\circ} \otimes V_{\rho^\circ} \Big) .
\end{equation}
By Proposition \ref{prop:1.1} and Lemma \ref{lem:5.4} this is isomorphic to
\begin{equation}\label{eq:5.16}
\bigoplus_{\rho^\circ \in \Sigma (u)} \ind_{A_G (u)_{\eta^\circ}}^{A_G (u)} \Big(
\overline{\Q_{\ell}} [A_G (u)_{\eta^\circ} / A_{G^\circ}(u),\natural_\cE^{-1}] \otimes
V_{\eta^\circ} \otimes V_{\rho^\circ} \Big) := 
\bigoplus_{\rho^\circ \in \Sigma (u)} B_{\rho^\circ} .
\end{equation}
(This equality defines $B_{\rho^\circ}$.)
Let us analyse the action of $\End_G (\pi_* \tilde{\cE})$ on \eqref{eq:5.16}. 
By \eqref{eq:5.9} and Lemma \ref{lem:5.2} there is a subalgebra $\overline{\Q_{\ell}} 
[W_{\ft,\rho^\circ},\natural_\cE]$, which stabilizes $B_{\rho^\circ}$. Moreover,
by Lemma \ref{lem:1.3}.a
\[
\overline{\Q_{\ell}} [A_G (u)_{\eta^\circ} / A_{G^\circ}(u),\natural_\cE^{-1}] \cong
\overline{\Q_{\ell}} [W_{\ft,\rho^\circ} / W_{\ft^\circ},\natural_\cE^{-1}] \cong 
\overline{\Q_{\ell}} [W_{\ft,\rho^\circ} / W_{\ft^\circ},\natural_\cE]^{op} .
\]
By Lemma \ref{lem:1.3}.c it decomposes as
\begin{equation}\label{eq:5.17}
\overline{\Q_{\ell}} [A_G (u)_{\eta^\circ} / A_{G^\circ}(u),\natural_\cE^{-1}] \cong
\bigoplus_{\tau' \in \Irr (\overline{\Q_{\ell}} [W_{\ft,\rho^\circ} / W_{\ft^\circ},
\natural_\cE^{-1}])} \hspace{-1cm} V_{\tau'} \otimes V_{\tau'}^* \cong 
\bigoplus_{\tau \in \Irr (\overline{\Q_{\ell}} [W_{\ft,\rho^\circ} / W_{\ft^\circ},
\natural_\cE])} \hspace{-1cm} V_{\tau}^* \otimes V_\tau .
\end{equation}
Recall that $W_{\ft^\circ}$ is a normal subgroup of $W_\ft$ and that $W_\ft$
acts on $\Irr_{\overline{\Q_{\ell}}}(W_{\ft^\circ})$, as in \eqref{eq:1.8}. Then
$W_\ft / W_{\ft,\rho^\circ}$ is in bijection with the $W_\ft$-orbit of 
$\rho^\circ$, so
\begin{equation}\label{eq:5.18}
\bigoplus\nolimits_{\rho^i \in W_\ft \cdot \rho^\circ} B_{\rho^i} =
\overline{\Q_{\ell}} [W_\ft,\natural_\cE] B_{\rho^\circ} \cong
\ind_{\overline{\Q_{\ell}}[W_{\ft,\rho^\circ},\natural_\cE]}^{\overline{\Q_{\ell}}
[W_\ft,\natural_\cE]} B_{\rho^\circ} .
\end{equation}
It follows from \eqref{eq:5.16}, \eqref{eq:5.17} and \eqref{eq:5.18} that
\eqref{eq:5.15} is isomorphic to 
\begin{multline}\label{eq:5.19}
\bigoplus_{\rho^\circ \in \Sigma (u) / W_\ft} \ind_{\overline{\Q_{\ell}}[W_{\ft,\rho^\circ},
\natural_\cE]}^{\overline{\Q_{\ell}} [W_\ft,\natural_\cE]} 
\ind_{A_G (u)_{\eta^\circ}}^{A_G (u)}
\Big( \bigoplus_{\tau \in \Irr (\overline{\Q_{\ell}} [W_{\ft,\rho^\circ} / W_{\ft^\circ},
\natural_\cE])} \hspace{-1cm} V_\tau^* \otimes V_\tau \otimes V_{\eta^\circ} \otimes 
V_{\rho^\circ} \Big) = \\
\bigoplus_{\rho^\circ \in \Sigma (u) / W_\ft} \bigoplus_{\tau \in \Irr 
(\overline{\Q_{\ell}} [W_{\ft,\rho^\circ} / W_{\ft^\circ},\natural_\cE])} \hspace{-1cm} 
\ind_{A_G (u)_{\eta^\circ}}^{A_G (u)} (V_\tau^* \otimes V_{\eta^\circ}) \otimes
\ind_{\overline{\Q_{\ell}}[W_{\ft,\rho^\circ},\natural_\cE]}^{\overline{\Q_{\ell}}
[W_\ft,\natural_\cE]} (V_\tau \otimes V_{\rho^\circ}) .
\end{multline}
Let $\rho = \rho^\circ \rtimes \tau = \ind_{\overline{\Q_{\ell}}[W_{\ft,\rho^\circ},
\natural_\cE]}^{\overline{\Q_{\ell}} [W_\ft,\natural_\cE]} 
(V_\tau \otimes V_{\rho^\circ})$. By \eqref{eq:5.19} 
\begin{multline*}
\cH^{2d_C} \big( \IC (\overline{Y},\pi_* \tilde{\cE})_\rho \big) \big)_u =
\Hom_{\overline{\Q_l}[W_\ft,\natural_\cE]} \Big( \rho, \cH^{2d_C} \big( 
\IC (\overline{Y},\pi_* \tilde{\cE}) \big)_u \Big) \cong \\
\ind_{A_G (u)_{\eta^\circ}}^{A_G (u)} (V_\tau^* \otimes V_{\eta^\circ}) =
\tau^* \ltimes \eta^\circ = \Sigma_{\ft^\circ}^{-1}(\rho^\circ) \rtimes \tau^*.
\end{multline*}
Hence the formula for $\Sigma_\ft^{-1}$ given in the theorem agrees with
the bijection \eqref{eq:5.14}.

Let us also compare the given formula for $\Sigma_\ft$ with the above constructions. 
By Theorem \ref{thm:1.2} $\cF_u \cong \eta^\circ \rtimes \tau'$ with $\eta^\circ \in 
\Irr_{\overline{\Q_l}}(A_{G^\circ}(u))$. We rewrite
\[
\Hom_G \Big( \cF , \cH^{2d_C} \big( \IC (\overline{Y},\pi_* \tilde{\cE}) \big) 
| \cC_u^G \Big) = \Hom_{A_G (u)} \Big( \eta^\circ \rtimes \tau' , 
\cH^{2d_C} \big( \IC (\overline{Y}, \pi_* \tilde{\cE}) \big)_u \Big) .
\]
By \eqref{eq:5.19} this equals 
\[
\ind_{\overline{\Q_{\ell}}[W_{\ft,\rho^\circ},\natural_\cE]}^{\overline{\Q_{\ell}}
[W_\ft,\natural_\cE]} (V^*_{\tau'} \otimes V_{\rho^\circ}) =
\tau'^* \ltimes \rho^\circ = \Sigma_{\ft^\circ}(u,\eta^\circ) \rtimes \tau'^* .
\]
Hence $\Sigma_\ft$ as given agrees with \eqref{eq:5.14}.
\end{proof}

The maps $\Psi_G$ and $\Sigma_\ft$ are compatible with restriction to Levi
subgroups, in the following sense. Let $H \subset G$ be an algebraic subgroup
such that $H \cap G^\circ$ is a Levi subgroup of $G^\circ$. Suppose that
$u \in H^\circ$ is unipotent. By \cite[\S 3.2]{Ree1} 
\begin{equation}\label{eq:5.23}
Z_G (u)^\circ \cap H = Z_{G^\circ}(u)^\circ \cap H \quad \text{equals} \quad
Z_{H^\circ}(u)^\circ = Z_H (u)^\circ .
\end{equation}
Hence the natural map $A_H (u) \to A_G (u)$ is injective and we can regard
$A_H (u)$ as a subgroup of $A_G (u)$. Let $\pi_*^H \widetilde{\cE_H}$ be the
$H$-equivariant local system on $\cC_u^H$ constructed like $\pi_* \tilde{\cE}$ but
for the group $H$. By Proposition \ref{prop:5.3} $\End_H (\pi_*^H \widetilde{\cE_H})$
is naturally a subalgebra of $\End_G (\pi_* \tilde{\cE})$.

\begin{thm}\label{thm:5.6}
Let $\eta \in \Irr_{\overline{\Q_{\ell}}}(A_G (u))$ and 
$\eta_H \in \Irr_{\overline{\Q_{\ell}}}(A_H (u))$. 
\enuma{
\item If $\eta_H$ appears in $\Res^{A_G (u)}_{A_H (u)}(\eta)$, then
$\Psi_G (u,\eta) = \Psi_H (u,\eta_H) / G$-conjugacy.
\item Suppose that $\Psi_G (u,\eta) = \Psi_H (u,\eta_H) / G$-conjugacy.
Then $\Sigma_{\Psi_H (u,\eta_H)}(u,\eta_H)$ is a constituent of 
$\Res_{\End_H (\pi_*^H \widetilde{\cE_H})}^{ \End_G (\pi_* \tilde{\cE})} 
\Sigma_\ft (u,\eta)$ if and only if $\eta_H$
is a constituent of $\Res^{A_G (u)}_{A_H (u)}(\eta)$.
} 
\end{thm}
\textbf{Remark.} In \cite[\S 8]{Lus1} both parts were proven in the connected case, 
for $G^\circ$ and $H^\circ$. As in this source, it is likely that in part (b) the 
multiplicity of $\eta_H$ in $\eta$ equals the multiplicity of 
$\Sigma_{\Psi_H (u,\eta_H)}(u,\eta_H)$ in $\Sigma_\ft (u,\eta)$. However, it seems
difficult to prove that with the current techniques. We will return to this
issue in \cite{AMS}.
\begin{proof}
(a) Let $\eta_H^\circ$ be an irreducible constituent of 
$\Res^{A_H (u)}_{A_{H^\circ} (u)}(\eta_H)$ and let $\eta^\circ$ be an irreducible
constituent of $\Res^{A_G (u)}_{A_{G^\circ} (u)}(\eta)$ which contains 
$\eta_H^\circ$. By the definition \eqref{eq:5.1} and by \cite[Theorem 8.3.a]{Lus1}
there are equalities up to $G$-conjugacy:
\[
\Psi_G (u,\eta) = \Psi_{G^\circ} (u,\eta) = 
\Psi_{H^\circ}(u,\eta_H^\circ) = \Psi_H (u,\eta_H) .
\]
(b) Write $\eta = \eta^\circ \rtimes \tau^*$ as in \eqref{eq:5.24}. Similarly,
we can write any irreducible representation of $A_H (u)$ as 
$\eta_H = \eta_H^\circ \rtimes \tau_H^*$ with $\eta_H^\circ \in 
\Irr_{\overline{\Q_{\ell}}}(A_{H^\circ}(u))$ and \\$\tau_H^* \in \Irr (\overline{\Q_{\ell}}
[W_{\ft_H,\eta_H^\circ} / W_{\ft_H^\circ},\natural_\cE^{-1}] )$. 
As representations of $A_{G^\circ} (u)$:
\begin{equation}\label{eq:5.25}
\eta = \ind_{A_G (u)_{\eta^\circ}}^{A_G (u)}(V_{\eta^\circ} \otimes V_\tau^*) \cong
\bigoplus_{a \in A_G (u) / A_G (u)_{\eta^\circ}} (a \cdot \eta^\circ) \otimes
(a \cdot V_\tau^*) ,
\end{equation}
where $A_{G^\circ}(u)$ acts trivially on the parts $a \cdot V_\tau^*$.
Using Proposition \ref{prop:1.1}.d and \eqref{eq:5.25} we compute
\begin{multline*}
\Hom_{A_H (u)}(\eta_H,\eta) \cong \Hom_{\overline{\Q_{\ell}}
[A_H (u)_{\eta_H^\circ} / A_{H^\circ}(u),\natural_\cE^{-1}]} \Big( \tau_H^* , 
\Hom_{A_{H^\circ}(u)}(\eta_H^\circ, \eta) \Big) \cong \\
\bigoplus_{a \in A_G (u) / A_G (u)_{\eta^\circ}} \Hom_{\overline{\Q_{\ell}}
[A_H (u)_{\eta_H^\circ} / A_{H^\circ}(u),\natural_\cE^{-1}]} \Big( \tau_H^* ,
\Hom_{A_{H^\circ}(u)}(\eta_H^\circ, a \cdot \eta^\circ) \otimes a \cdot V_\tau^* \Big) .
\end{multline*}
Here $\overline{\Q_{\ell}} [A_H (u)_{\eta_H^\circ} / A_{H^\circ}(u),\natural_\cE^{-1}]$ does
not act on $\Hom_{A_{H^\circ}(u)}(\eta_H^\circ, a \cdot \eta^\circ)$, so we can
rearrange the last line as
\begin{equation}\label{eq:5.26}
\bigoplus_{a \in A_G (u) / A_G (u)_{\eta^\circ}} \Hom_{A_{H^\circ}(u)}(\eta_H^\circ, 
a \cdot \eta^\circ) \otimes \Hom_{\overline{\Q_{\ell}} [A_H (u)_{\eta_H^\circ} / 
A_{H^\circ}(u),\natural_\cE^{-1}]} (\tau_H^*,a \cdot V_\tau^* ) .
\end{equation}
Notice that $\eta \cong a \cdot \eta \cong a \cdot \eta^\circ \rtimes a \cdot \tau^*$.
We conclude from \eqref{eq:5.26} that $\Hom_{A_H (u)}(\eta_H,\eta)$ is nonzero if and 
only if $\eta \cong \eta^\circ \rtimes \tau^*$ where $\Hom_{A_{H^\circ}(u)}(\eta_H^\circ, 
\eta^\circ) \neq 0$ and \\ $\Hom_{\overline{\Q_{\ell}} [A_H (u)_{\eta_H^\circ} / 
A_{H^\circ}(u),\natural_\cE^{-1}]} (\tau_H^*,\tau^* ) \neq 0$.

Write $\rho = \Sigma_\ft (u,\eta)$ and let $\rho_H = \rho_H^\circ \rtimes \tau_H \in
\Irr (\End_H (\pi_*^H \widetilde{\cE_H}))$. Just as in \eqref{eq:5.26} one shows
that $\Hom_{\End_H (\pi_*^H \widetilde{\cE_H})}(\rho_H,\rho)$ is nonzero if and only
if $\rho \cong \rho^\circ \rtimes \tau$ with $\Hom_{\overline{\Q_{\ell}}[W_{\ft_H^\circ}]}
(\rho_H^\circ,\rho^\circ) \neq 0$ and $\Hom_{\overline{\Q_{\ell}}[W_{\ft_H,\rho_H^\circ} /
W_{\ft_H^\circ},\natural_\cE]}(\tau_H,\tau) \neq 0$.

Write $\ft_H = \Psi_H (u,\eta_H), \ft_H^\circ = \Psi_{H^\circ} (u,\eta_H^\circ)$ and 
consider $\rho_H^\circ = \Sigma_{\ft_H^\circ}(u,\eta_H^\circ)$. Then
\[
\rho_H = \rho_H^\circ \rtimes \tau_H \quad \text{equals} \quad  
\Sigma_{\ft_H}(u,\eta_H) = \Sigma_{\ft_H}(u,\eta_H^\circ \rtimes \tau_H^*) .
\]
By \cite[Theorem 8.3.b]{Lus1} 
\[                              
\dim_{\overline{\Q_{\ell}}} \Hom_{A_{H^\circ}(u)}(\eta_H^\circ, \eta^\circ) =
\dim_{\overline{\Q_{\ell}}} \Hom_{\overline{\Q_{\ell}}[W_{\ft_H^\circ}]} 
(\rho_H^\circ,\rho^\circ)
\]
and from Lemmas \ref{lem:1.3} and \ref{lem:5.2} we see that
\[                              
\dim_{\overline{\Q_{\ell}}} \Hom_{\overline{\Q_{\ell}} [A_H (u)_{\eta_H^\circ} / 
A_{H^\circ}(u),\natural_\cE^{-1}]} (\tau_H^*,\tau^* ) =  
\dim_{\overline{\Q_{\ell}}} \Hom_{\overline{\Q_{\ell}}[W_{\ft_H,\rho_H^\circ} /
W_{\ft_H^\circ},\natural_\cE]}(\tau_H,\tau) .
\]
The above observations entail that $\Hom_{A_H (u)}(\eta_H,\eta)$ is nonzero if 
and only if \\$\Hom_{\End_H (\pi_*^H \widetilde{\cE_H})}(\rho_H,\rho)$ is nonzero.
\end{proof}

\section{A version with quasi-Levi subgroups}
\label{sec:quasiLevi}

For applications to Langlands parameters we need a version of the generalized
Springer correspondence which involves a disconnected version of Levi subgroups. 
Recall that every Levi subgroup $L$ of $G^\circ$ is of the form 
$L = Z_{G^\circ} (Z(L)^\circ)$.

\begin{defn}
Let $G$ be a possibly disconnected complex reductive algebraic group, and
let $L \subset G^\circ$ be a Levi subgroup. Then we call $Z_G (Z(L)^\circ)$
a quasi-Levi subgroup of $G$. 
\end{defn}

Notice that $Z_G (Z(L)^\circ)$ also has neutral component $L$ and connected
centre $Z(L)^\circ$. Hence there is canonical bijection between Levi subgroups 
and quasi-Levi subgroups of $G$.
We will also need some variations on other previous notions.

\begin{defn}
A unipotent cuspidal quasi-support for $G$ is a triple $(M,v,q \epsilon)$
where $M \subset G$ is a quasi-Levi subgroup, $v \in M^\circ$ is unipotent
and $q \epsilon \in \Irr_\cusp (A_M (v))$. We write
\[
q \cS_G = \{ \text{cuspidal unipotent quasi-supports for } G \} / 
G\text{-conjugacy} . 
\]
\end{defn}
Like before, we will also think of unipotent cuspidal quasi-supports as
triples $(M,\cC_v^M,q \cE)$, where $q \cE$ is a cuspidal local system on $\cC_v^M$.
We want to define a canonical map 
\[
q \Psi_G : \cN_G^+ \to q \cS_G ,
\]
and to analyse its fibers. Of course this map should just be $\Psi_G$ if $G$ is
connected. 

Let $\ft = [M^\circ,\cC_v^{M^\circ},\cE ]_G$ and suppose that $(u,\eta) \in
\cN_G^+$ with $\Psi_G (u,\eta) = \ft$. Obviously, the cuspidal quasi-support of 
$(u,\eta)$ will involve the quasi-Levi subgroup $M = Z_G (Z(M^\circ)^\circ)$.
From Theorem \ref{thm:5.5} we get
\[
\rho = \Sigma_\ft (u,\eta) \in \Irr (\End_G (\pi_* \tilde{\cE})) . 
\]
Let $\pi_*^M \widetilde{\cE_M}$ be the $M$-equivariant local system on $\cC_v^M$
built from $\cE$ in the same way as $\pi_* \tilde{\cE}$, only with $M$ instead of $G$.
From Proposition \ref{prop:5.3} we see that $\End_G (\pi_* \tilde{\cE}) \cong
\overline{\Q_{\ell}}[W_\ft , \natural_\cE]$ naturally contains a subalgebra 
\[
\End_M (\pi_*^M \widetilde{\cE_M}) \cong \overline{\Q_{\ell}}[M_\cE / M^\circ, \natural_\cE] .
\]
As $M_\cE / M^\circ$ is normal in $W_\ft = N_G (\ft) / M^\circ$, the latter group
acts on $\End_M (\pi_*^M \widetilde{\cE_M})$ by conjugation in $\overline{\Q_{\ell}}[W_\ft , 
\natural_\cE]$. Let $\rho_M \in \Irr \big( \End_M (\pi_*^M \widetilde{\cE_M}) \big)$ be a 
constituent of $\rho$ as $\End_M (\pi_*^M \widetilde{\cE_M})$-representation. 
By the irreducibility of $\rho$ as $\overline{\Q_{\ell}}[W_\ft ,\natural_\cE]$-representation,
$\rho_M$ is unique up to conjugation by $W_\ft$. Let us write 
$\ft_M = [M^\circ,\cC_v^{M^\circ}, \cE ]_M \in \cN_M^0$. By Proposition \ref{prop:4.4}
\begin{equation}\label{eq:6.11}
q \cE := \Hom_{\End_M (\pi_*^M \tilde{\cE_M})} 
\big( \rho_M, \pi_*^M \widetilde{\cE_M} \big) = (\pi_*^M \tilde{\cE})_{\rho_M} 
\end{equation}
is a cuspidal local system on $\cC_v^M = \cC_v^{M^\circ}$, and 
$\Sigma_{\ft_M}(\cC_v^M, q \cE) = \rho_M$. Since any other choice $\rho'_M$ is 
conjugate to $\rho_M$ by an element of $N_G (\ft) ,\; (M,\cC_v^M,q \cE)$ is 
determined by $(u,\eta)$, up to $G$-conjugacy. Thus we can canonically define
\begin{equation}\label{eq:6.1}
q \Psi_G (u,\eta) = [M,\cC_v^M,q \cE]_G . 
\end{equation}
Let $\cF$ be the irreducible local system on $\cC_u^G$ with $\cF_u = \eta$ and
let $Y = Y_{(M^\circ,S)} \subset G$, where $S = \cC_v^M Z(M)^\circ$. From
Theorem \ref{thm:5.5} we see that $\cF = \cH^{2 d_C} (\IC (\overline Y, \pi_* 
\tilde{\cE})_\rho ) | \cC_u^G$. To $(M,\cC_v^M,q \cE)$ we can apply the same 
constructions as to $(L,\cC_v^L,\cE)$ in \eqref{eq:5.20}, in particular 
there is a version $q\pi_*$ of $\pi_*$ based on $G$ and $M$. Distinguishing the 
various operations $\pi_*$, we get usual isomorphisms of local systems on $Y$:
\begin{equation}\label{eq:6.7}
q \pi_* (\widetilde{q \cE}) =
q \pi_* \Big( \widetilde{(\pi_*^M \tilde{\cE_M})_{\rho_M}} \Big) \cong 
(\pi_* \tilde{\cE} )_{\rho_M} .
\end{equation}
Since $\pi_* \tilde{\cE}$ is semisimple, $\IC (\overline{Y},
(\pi_* \tilde{\cE} )_{\rho_M}) \cong \IC (\overline{Y},\pi_* \tilde{\cE})_{\rho_M}$.
Hence $\cF$ is also a direct summand of 
\begin{equation}\label{eq:6.2}
\cH^{2 d_C} \big( \IC \big( \overline Y, q\pi_* ( \widetilde{q \cE}) \big) \big)
| \cC_u^G \cong \cH^{2 d_C} \big( \IC (\overline Y, 
\pi_* \tilde{\cE})_{\rho_M} \big) | \cC_u^G .
\end{equation}
It follows that the characterization of $\Psi_G$ given in \cite[Theorem 6.5]{Lus1}
remains valid for $q \Psi_G$. In particular $\cF$ and $q \cE$ have the same
$Z(G)$-character and
\begin{equation}\label{eq:6.14}
q \Psi_G \text{ preserves } Z(G)\text{-characters.} 
\end{equation}
We abbreviate 
\[
q \ft = [M,\cC_v^M, q \cE]_G.
\]
Let $N_G (q \ft)$ be the stabilizer of $(M,\cC_v^M,q \cE)$ in $G$. It normalizes $M$ 
and contains $M$, because $(q \cE)_v \in \Irr (A_M (v))$. Every element of 
$N_G (q\ft)$ maps $\cE$ to a $M$-associate local system on $\cC_v^{M^\circ}$,
because $q \cE$ is a $M$-equivariant local system which, as $M^\circ$-equivariant
sheaf, contains $\cE$. Hence $N_G (q \ft) = N_G (\ft,q\cE) M$.

Analogous to $W_\ft = N_G (\ft) / M^\circ$, we define
\begin{equation}\label{eq:6.17}
W_{q \ft} = N_G (q \ft) / M . 
\end{equation}
There are natural isomorphisms 
\begin{equation}\label{eq:6.3}
W_{q \ft} \cong N_G (\ft , q \cE) M / M \cong N_G (\ft , q \cE) / M_\cE \cong 
\Stab_{W_\ft} (q \cE) / (M_\cE / M^\circ) .
\end{equation}
The group $W_{\ft} = N_{G^\circ}(M^\circ) / M^\circ$ is isomorphic to 
$N_{G^\circ}(\ft) M / M$, and there is a natural injection
\begin{equation}\label{eq:6.16}
W_{\ft} \cong N_{G^\circ}(\ft) M / M \to W_{q \ft} . 
\end{equation}

\begin{lem}\label{lem:6.1}
There exists a 2-cocycle $\kappa_{q \ft}$ of $W_{q \ft}$ such that:
\enuma{
\item there is a bijection 
\[
q \Psi_G^{-1} (q \ft) \to \Irr \big( \overline{\Q_{\ell}} [W_{q \ft},\kappa_{q \ft}] \big) ,
\]
\item $\kappa_{q \ft}$ factors through $W_{q \ft} / W_{\ft}$ and
$\overline{\Q_{\ell}}[W_{\ft}]$ is canonically embedded in 
$\overline{\Q_{\ell}}[W_{q \ft}, \kappa_{q \ft}]$.
}
\end{lem}
\begin{proof}
(a) Recall the bijection $\Psi_G^{-1}(\ft) \to \Irr \big( \overline{\Q_{\ell}} 
[W_\ft,\natural_\cE] \big)$ from Theorem \ref{thm:5.5}. With \cite[\S 53]{CuRe} we
can find a central extension $\widetilde{W_\ft}$ of $W_\ft$ and a minimal idempotent
$p_\cE \in \overline{\Q_{\ell}}[\ker (\widetilde{W_\ft} \to W_\ft)]$, such that
\begin{equation}\label{eq:6.4}
\overline{\Q_{\ell}}[W_\ft,\natural_\cE ] \cong p_\cE \overline{\Q_{\ell}}[\widetilde{W_\ft}] . 
\end{equation}
Let $N \subset \widetilde{W_\ft}$ be the inverse image of $M_\cE / M^\circ \subset
W_\ft$. It is a normal subgroup of $\widetilde{W_\ft}$ because $M_\cE = 
M \cap N_G (\ft)$ is normal in $N_G (\ft)$. We note that 
\begin{equation}\label{eq:6.5}
\widetilde{W_\ft} / N \cong W_\ft / (M_\cE / M^\circ) \cong N_G (\ft) / M_\cE .
\end{equation}
As a consequence of \eqref{eq:6.4}
\begin{equation}\label{eq:6.10}
\overline{\Q_{\ell}}[M_\cE / M^\circ,\natural_\cE ] \cong p_\cE \overline{\Q_{\ell}}[N] .  
\end{equation}
By Theorem \ref{thm:1.2} there is a bijection
\[
\begin{array}{ccc} 
\Irr_{\overline{\Q_{\ell}}}(\widetilde{W_\ft}) & \longleftrightarrow &
\big( \Irr_{\overline{\Q_{\ell}}}(N) \q \widetilde{W_\ft} / N \big)_\kappa \\
\pi \rtimes \sigma & \longleftrightarrow  & (\pi,\sigma) .
\end{array}
\]
With Proposition \ref{prop:1.1}.c we can restrict it to representations on which
$p_\cE$ acts as the identity. With \eqref{eq:6.4} that yields a bijection 
\begin{equation}\label{eq:6.6}
\Irr \big( \overline{\Q_{\ell}}[W_\ft,\natural_\cE] \big) \longleftrightarrow 
\Big( \Irr \big( \overline{\Q_{\ell}}[M_\cE / M^\circ,\natural_\cE] \big) \q 
W_\ft / (M_\cE / M^\circ) \Big)_\kappa .
\end{equation}
Under the bijections from Theorem \ref{thm:5.5} and \eqref{eq:6.6}, the set
$q \Psi_G^{-1}(q \ft) \subset \Psi_G^{-1} (\ft)$ is mapped to the fiber of
$W_\ft \cdot \rho_M$ (with respect to the map from the extended quotient on 
the right hand side of \eqref{eq:6.6} to the corresponding ordinary quotient).
By the definition of extended quotients, this fiber is in bijection with $\Irr 
\big( \overline{\Q_{\ell}}[W_{\ft,\rho_M} / (M_\cE / M^\circ), \kappa_{\rho_M}] \big)$.
By the equivariance of the Springer correspondence, the stabilizer of 
$\Sigma_{\ft_M}(\cC_v^M, q \cE) = \rho_M$ in 
$W_\ft / (M_\cE / M^\circ)$ is $\Stab_{W_\ft}(q \cE) / (M_\cE / M^\circ)$,
which by \eqref{eq:6.3} is isomorphic with $W_{q \ft}$. Thus the composition
of Theorem \ref{thm:5.5} and \eqref{eq:6.6} provides the required bijection,
with $\kappa_{\rho_M}$ as cocycle.\\
(b) Consider $w \in W_{\ft}$ with preimage $\tilde w \in \widetilde{W_\ft}$.
Since $M_\cE / M^\circ \cong M_\cE G^\circ / G^\circ$, $w$ commutes with
$M_\cE / M^\circ$. As $\natural_\cE$ is trivial on $W_{\ft}$, moreover
for all $m \in M_\cE / M^\circ$ 
\begin{equation}\label{eq:6.15}
T_w T_m (T_w )^{-1} = T_m \qquad \text{in } \overline{\Q_{\ell}}[W_\ft,\natural_\cE] . 
\end{equation}
Hence $W_{\ft}$ stabilizes $\rho_M$ and 
\[
W_{\ft} \cong W_{\ft} (M_\cE / M^\circ) / (M_\cE / M^\circ)
\quad \text{is contained in} \quad W_{\ft,\rho_M} / (M_\cE / M^\circ) .
\]
It also follows from \eqref{eq:6.15} that we can take $I^w_{\rho_M} = 
\mathrm{Id}_{V_{\rho_M}}$. In view of Proposition \ref{prop:1.1}.a, the 2-cocycle
$\kappa_{\rho_M}$ on $W_{\ft^\circ}$ agrees with $\natural_\cE |_{W_\ft} = 1$.
Via \eqref{eq:6.16} we consider $W_{\ft}$ as a subgroup of $W_{q \ft}$. 
Then the subalgebra of $\overline{\Q_{\ell}}[W_{q \ft}, \kappa_{q \ft}]$ spanned
by the $T_w$ with $w \in W_{\ft}$ is simply $\overline{\Q_{\ell}}[W_{\ft}]$.
\end{proof}

We will make the bijection of Lemma \ref{lem:6.1}.a canonical, by replacing
$\overline{\Q_{\ell}} [W_{q \ft},\kappa_{q \ft}]$ with the endomorphism algebra of the
equivariant local system \eqref{eq:6.7}.

\begin{lem}\label{lem:6.2}
Let $\kappa_{q \ft}$ be as in Lemma \ref{lem:6.1}. There is an isomorphism 
\[
\End_G \big( q\pi_* (\widetilde{q \cE}) \big) \cong 
\overline{\Q_{\ell}} [W_{q \ft},\kappa_{q \ft}] , 
\]
and it is canonical up to automorphisms of the right hand side which come from
characters of $W_{q \ft} / W_{\ft^\circ}$.
Under this isomorphism $\overline{\Q_{\ell}} [W_{\ft^\circ}]$ to corresponds to
$\End_{G^\circ}(\pi_* \tilde{\cE})$, which acts on $q\pi_* (\widetilde{q \cE})$
via \eqref{eq:6.7}.
\end{lem}
\begin{proof}
Like Proposition \ref{prop:5.3}, the larger part of this result follows from 
\cite[\S 3]{Lus1}. The constructions over there apply
equally well to quasi-Levi subgroups of the possibly disconnected group $G$.
These arguments show that, as a $\overline{\Q_{\ell}}$-vector space, $\End_G (q\pi_* 
(\widetilde{q \cE}))$ is in a canonical way a direct sum of one-dimensional 
subspaces $q \mc A_w$ indexed by $W_{q \ft}$. Moreover, as an algebra it is
a twisted group algebra of $W_{q \ft}$, with respect to some 2-cocycle. 
To analyse the 2-cocycle, we relate it to objects appearing in the proof of 
Lemma \ref{lem:6.1}.a.

By \eqref{eq:6.7} $V_{\rho_M} \otimes q\pi_* (\widetilde{q \cE})$, with $G$ 
acting trivially on $V_{\rho_M}$, is a direct summand of $\pi_* (\tilde{\cE})$.
By \cite[Proposition 3.5]{Lus1} the $G$-equivariant local system $\pi_* \tilde{\cE}$ 
is semisimple. Therefore
\begin{equation}\label{eq:6.9}
\End_G \big( V_{\rho_M} \otimes q\pi_* (\widetilde{q \cE}) \big) \cong
\End_G \big( q\pi_* (\widetilde{q \cE}) \big) \otimes \End_{\overline{\Q_{\ell}}}
(V_{\rho_M})
\end{equation}
is a subalgebra of $\End_G (\pi_* \tilde{\cE})$.
The basis elements $b_w \in \End_G (\pi_* \tilde{\cE}) ,\;
w \in W_\ft$, as constructed in Proposition \ref{prop:5.3}, permute the
subsystems of $\pi_* \tilde{\cE}$ corresponding to different $\rho'_M \in 
\Irr \big( \End_M (\pi_*^M \widetilde{\cE_M}) \big)$. More precisely,
$b_w$ stabilizes $V_{\rho_M} \otimes q\pi_* (\widetilde{q \cE})$ if and only if
$w$ stabilizes $\rho_M$. Together with Proposition \ref{prop:5.3} this shows
that \eqref{eq:6.9} is spanned (over $\overline{\Q_{\ell}}$) by the $b_w$ with 
$w \in \Stab_{W_\ft}(\rho_M) = \Stab_{W_\ft}(q \cE)$. \\
In view of the description of $(\pi_* \tilde{\cE})_y$ in \eqref{eq:5.21},
the stalk of $V_{\rho_M} \otimes q\pi_* (\widetilde{q \cE})$ at $y \in S_\reg$ is
\begin{equation}\label{eq:6.8}
\bigoplus\nolimits_{z \in Z_{N_G (M^\circ)}(v) / Z_M (v)} z \cdot \big( 
(\pi_*^M \widetilde{\cE_M})_{\rho_M} \big)_y \otimes V_{\rho_M} .
\end{equation}
As concerns the index set for the sum, by Theorem \ref{thm:4.1}.a the canonical map\\
$Z_{N_G (M^\circ)} (v) / Z_M (v) \to N_G (M^\circ) / M$ is bijective.

The $b_w$ with $w \in M_\cE / M^\circ$ act only on the second tensor
factor of \eqref{eq:6.8}, and by the irreducibility of the 
$\End_M (\pi_*^M \widetilde{\cE_M})$-module $\rho_M$ they span 
the algebra $\End_{\overline{\Q_{\ell}}} (V_{\rho_M})$. 
Let $[W_{q \ft}] \subset W_\ft$ be a set of representatives for 
$\Stab_{W_\ft}(q \cE) / (M_\cE / M^\circ)$. By \eqref{eq:6.16} we may 
assume that it contains $W_{\ft^\circ}$. 
From \eqref{eq:5.22} we see that the $b_w$ with $w \in [W_{q \ft}]$ permute 
the direct summands of \eqref{eq:6.8} according to the inclusion
\[
W_{q \ft} = N_G (q \ft) / M \to N_G (M^\circ) / M \cong Z_{N_G (M^\circ)} (v) / Z_M (v) . 
\]
In particular $\{ b_w : w \in [W_{q \ft}] \}$ is linearly independent over
$\End_{\overline{\Q_{\ell}}}(V_{\rho_M})$. Since \eqref{eq:6.9} is spanned
by the $b_w$ with $w \in \Stab_{W_\ft}(\rho_M)$, it follows that
in fact $\{ b_w : w \in [W_{q \ft}] \}$ is a basis of \eqref{eq:6.9} over
$\End_{\overline{\Q_{\ell}}}(V_{\rho_M})$.

We want to modify these $b_w$ to endomorphisms of $q\pi_* (\widetilde{q \cE})$,
say to $q b_w \in q \mc A_w$. For $w \in W_{\ft^\circ}$ there is an easy
canonical choice, as \eqref{eq:6.15} shows that $W_{\ft^\circ}$ commutes with
$\overline{\Q_{\ell}}[M / M^\circ,\natural_\cE]$. Hence $b_w$ fixes
$\rho_M \in \Irr (\overline{\Q_{\ell}}[M / M^\circ,\natural_\cE])$ pointwise.
Therefore we can take $q b_w = b_w$ for $w \in W_{\ft^\circ}$. By Theorem 
\ref{thm:2.2} these elements span the algebra $\End_{G^\circ}(\pi_* \tilde{\cE})
\cong \overline{\Q_{\ell}}[W_{\ft^\circ}]$.

For general $w \in [W_{q \ft}]$ the description given in \eqref{eq:5.22} shows
that the action of $b_w$ on \eqref{eq:6.8} consists of a permutation of the
direct factors combined with a linear action on $V_{\rho_M}$. 
Let $\widetilde{W_\ft}$ and $N$ be as in \eqref{eq:6.4} and \eqref{eq:6.10}.
Then \eqref{eq:6.8} can be embedded in a sum of copies of
$\ind_N^{\widetilde{W_\ft}}(V_{\rho_M})$. 

Now Proposition \ref{prop:1.1}.b
shows that there is a unique $q b_w \in q \mc A_w$ such that the action of
$b_w$ on \eqref{eq:6.8} can be identified with $q b_w \otimes I^w$, where
$I^w$ is as in \eqref{eq:1.3}. We may choose the same $I^w$ as we did
(implicitly) in the last part of the proof of Lemma \ref{lem:6.1}.a, where
we used them to determine the cocycle $\kappa_{\rho_M} = \kappa_{q \ft}$.
Then Proposition \ref{prop:1.1}.b shows also that these $q b_w$ multiply 
as in the algebra $\overline{\Q_{\ell}}[W_{q \ft},\kappa_{\rho_M}]$. 

Finally, the claim about the uniqueness follows in the same way as in the
last part of the proof of Propostion \ref{prop:5.3}.
\end{proof}

Some remarks about the 2-cocycle $\kappa_{q \ft}$ are in order.
If $W_{q \ft}$ is cyclic then $\kappa_{q \ft}$ is trivial because
$H^2 (W_{q \ft},\overline{\Q_{\ell}}^\times) = \{1\}$. Furthermore
\[
\text{if } M_\cE = M^\circ \text{, then }
\End_G (q\pi_* (\widetilde{q \cE})) = \End_G (\pi_* (\tilde{\cE}))
\cong \overline{\Q_{\ell}}[W_\ft,\natural_\cE] 
\]
by Proposition \ref{prop:5.3}.
However, in contrast with the cocycle $\natural_\cE$ appearing in
Sections \ref{sec:cusp} and \ref{sec:noncusp}, it is in general rather
difficult to obtain explicit information about $\kappa_{q \ft}$. One 
reason for this is that the classification of cuspidal local systems
on disconnected reductive groups, as achieved in Theorem \ref{thm:4.1} 
and in Proposition \ref{prop:4.4}, leaves many possibilities. 
In particular the groups $G_\cE / G^\circ$ can be very large.

\begin{thm}\label{thm:6.3}
\enuma{
\item There exists a canonical bijection
\[
\begin{array}{cccc}
q \Sigma_{q \ft} : & q \Psi_G^{-1}(q \ft) & \to & 
\Irr \big( \End_G (q\pi_* (\widetilde{q \cE})) \big) \\
& (\cC_u^G ,\cF) & \mapsto & \Hom_G \Big( \cF, \cH^{2 d_C} 
\big( \IC \big( \overline{Y}, q\pi_* (\widetilde{q \cE}) \big) \big) | \cC_u^G \Big) .
\end{array}
\]
It can be defined by the condition
\[
q \Sigma_{q \ft}^{-1} (\tau) = (\cC_u^G,\cF) 
\quad \Longleftrightarrow \quad \cF = \cH^{2 d_C} 
\Big( \IC \big( \overline{Y}, q\pi_* (\widetilde{q \cE}) \big)_\tau  \Big) | \cC_u^G .
\]
\item The restriction of $\cF$ to a $G^\circ$-equivariant local system on $\cC_u^{G^\circ}$
is $\bigoplus_i \Sigma_{\ft^\circ}^{-1}(\tau_i)$, where $\ft^\circ = [M^\circ,
\cC_v^{M^\circ},\cE]_{G^\circ}$ and $\tau = \bigoplus_i \tau_i$ is a decomposition into 
irreducible $\End_{G^\circ}(\pi_* \tilde{\cE})$-subrepresentations.
\item Upon choosing an isomorphism as in Lemma \ref{lem:6.2}, we obtain the bijection
\[
q \Psi_G^{-1} (q \ft) \to \Irr \big( \overline{\Q_{\ell}} [W_{q \ft},\kappa_{q \ft}] \big) 
\]
from Lemma \ref{lem:6.1}.
}
\end{thm}
\begin{proof}
(c) Write $\ft_M = \Psi_M (\cC_v^M,q \cE)$ and recall
that $\rho_M = \Sigma_{\ft_M}(\cC_v^M, q \cE)$. By Lemma \ref{lem:4.3} 
\[
\End_M \big( \pi_*^M (\widetilde{\cE_M}) \big) \cong 
\overline{\Q_{\ell}}[M_\cE / M^\circ ] \cong p_\cE \overline{\Q_{\ell}}[N] ,
\]
and by Lemma \ref{lem:6.2}
\[
\End_G (q\pi_* (\widetilde{q \cE})) \cong 
\overline{\Q_{\ell}}[W_{\ft,\rho_M} / (M_\cE / M^\circ),\kappa_{\rho_M}] .
\]
From a $\tau$ as in the theorem we obtain, using \eqref{eq:6.6}, 
\[
\rho_M \rtimes \tau \in \Irr (p_\cE \overline{\Q_{\ell}}[\widetilde{W_\ft}]) = 
\Irr (\overline{\Q_{\ell}}[W_\ft,\natural_\cE]) .
\]
By Theorem \ref{thm:5.5} the bijection from Lemma \ref{lem:6.1} maps 
$(\cC_u^G,\cF)$ to $\tau$ if and only if
\begin{equation}\label{eq:6.12}
\cF \quad \text{equals} \quad \cH^{2 d_C} \Big( \IC \big( \overline{Y}, 
\pi_* (\tilde{\cE}) \big)_{\rho_M \rtimes \tau}  \Big) | \cC_u^G .
\end{equation}
Recall from \eqref{eq:6.7} that $\Hom_N (\rho_M,q\pi_* \tilde{\cE}) \cong
\pi_* (\widetilde{q \cE})$. We apply Proposition \ref{prop:1.1}.d
to $\widetilde{W_\ft}, N$ and the representation $\pi_* \tilde{\cE}$, and
we find that the right hand side of \eqref{eq:6.12} is isomorphic with
$\cH^{2 d_C} \big( \IC \big( \overline{Y}, q\pi_* (\widetilde{q \cE})_\tau \big) 
| \cC_u^G \big)$. Since $q\pi_* (\widetilde{q \cE})$ is semisimple, taking the 
$\tau$-Hom-space commutes with forming an intersection cohomology complex. 
Hence the bijection from Lemma \ref{lem:6.1} satisfies exactly the condition 
given in the theorem.\\
(a) This condition clearly is canonical, so with part (a) it determines a 
canonical bijection $q \Sigma_{q \ft}^{-1}$. 

It remains to check that the given formula for $q\Sigma_{q \ft}$ agrees with
the above construction. Let $(\cC_u^G,\cF) \in q \Psi_G^{-1}(q \ft)$ and write
$\cF_u = \eta^\circ \rtimes \tau'$ as in the proof of Theorem \ref{thm:5.5}. 
Then, by \eqref{eq:6.7}
\begin{align}
\nonumber \Hom_G \Big( \cF, \cH^{2 d_C} \big( \IC \big( \overline{Y}, 
q\pi_* (\widetilde{q \cE}) \big) \big) | \cC_u^G \Big) =
\Hom_G \Big( \cF, \cH^{2 d_C} \big( \IC \big( \overline{Y}, \pi_* (\tilde \cE) 
\big)_{\rho_M} \big) | \cC_u^G \Big) \\
\label{eq:6.27} = \Hom_{A_G (u)} \Big( \eta^\circ \rtimes \tau', \cH^{2 d_C} 
\big( \IC \big( \overline{Y}, \pi_* (\tilde \cE) \big)_{\rho_M} \big)_u \Big) .
\end{align}
As the actions of $A_G (u)$ and 
\[
\overline{\Q_l}[M_\cE / M^\circ ,\natural_\cE] \cong 
\End_M (\pi_*^M (\widetilde{\cE_M})) \subset \End_G (\pi_* (\tilde \cE))
\]
commute, \eqref{eq:5.19} shows that \eqref{eq:6.27} is isomorphic to
\[
\Hom_{\overline{\Q_l}[M_\cE / M^\circ ,\natural_\cE]} \Big( \rho_M ,
\ind_{\overline{\Q_{\ell}}[W_{\ft,\rho^\circ},\natural_\cE]}^{\overline{\Q_{\ell}}
[W_\ft,\natural_\cE]} (V^*_{\tau'} \otimes V_{\rho^\circ}) \Big) =
\Hom_{\overline{\Q_l}[M_\cE / M^\circ ,\natural_\cE]} \big( \rho_M ,
\tau'^* \ltimes \rho^\circ \big) .
\]
From \eqref{eq:6.6} and the subsequent argument we see that 
$\tau'^* \ltimes \rho^\circ = \rho_M \rtimes \tau$
for a unique irreducible representation $\tau$ of (using Lemma \ref{lem:6.2})
\[
\overline{\Q_l}[W_{\ft,\rho_M} / (M_\cE / M^\circ), \natural_\cE] \big) \cong
\overline{\Q_l}[W_{q \ft},\kappa_{q\ft}] \cong \End_G \big(q\pi_* (\widetilde{q \cE}) \big) .
\]
In view of all this, \eqref{eq:6.27} becomes
\[
\Hom_{\overline{\Q_l}[M_\cE / M^\circ ,\natural_\cE]} \big( \rho_M ,\rho_M \rtimes \tau
\big) = \Hom_{\overline{\Q_l}[M_\cE / M^\circ ,\natural_\cE]} \big( \rho_M ,
\ind^{\overline{\Q_l}[W_\ft , \natural_\cE]}_{\overline{\Q_l}[W_{\ft,\rho_M} ,\natural_\cE]} 
(V_{\rho_M} \otimes V_\tau) \big) = V_\tau .
\]
This means that $q\Sigma_{q\ft}(\cC_u^G,\cF)$ as given in the statement is isomorphic
with the outcome of the bijection via part (c).\\
(b) The behaviour of the restriction of $q \Sigma_{q \ft}^{-1} (\tau)$ to $\cC_u^{G^\circ}$ 
follows from comparing the characterization with Theorem \ref{thmlus}.(3).
\end{proof}

By Theorem \ref{thm:5.5} and \eqref{eq:6.6}, $q\Sigma_{q \ft}(u,\eta)$ is also given by
\begin{equation}\label{eq:6.13}
\Sigma_\ft (\cC_u^G,\cF) = 
\Sigma_{\ft_M}(\cC_v^M,q \cE) \rtimes q\Sigma_{q \ft}(\cC_u^G,\cF) .
\end{equation}
However, it is hard to make sense of this $\rtimes$-sign in a completely canonical way,
without using the isomorphisms from Proposition \ref{prop:5.3} and Lemma \ref{lem:6.2}.

There is also an analogue of Theorem \ref{thm:5.6} with quasi-Levi subgroups.
Assume that $H \subset G$ is an algebraic subgroup such that $H \cap G^\circ$ is a
Levi subgroup of $G^\circ$ and $H$ contains the quasi-Levi subgroup 
$Z_G (Z(G^\circ \cap H)^\circ)$. Let $u \in H^\circ$ be unipotent. We saw in 
\eqref{eq:5.23} that $A_H (u)$ can be regarded as a subgroup of $A_G (u)$.

\begin{prop}\label{prop:6.4}
In the above setting, let $\eta \in \Irr_{\overline{\Q_{\ell}}}(A_G (u))$ and 
$\eta_H \in \Irr_{\overline{\Q_{\ell}}}(A_H (u))$. 
\enuma{
\item If $\eta_H$ appears in $\Res^{A_G (u)}_{A_H (u)}(\eta)$, then 
$q \Psi_G (u,\eta) = q \Psi_H (u,\eta_H) / G$-conjugacy.
\item There is a natural inclusion of algebras $\End_H (q\pi_*^H (\widetilde{q \cE_H})) \to
\End_G (q\pi_* (\widetilde{q \cE}))$.

Suppose that $q \Psi_G (u,\eta) = q \Psi_H (u,\eta_H) / G$-conjugacy.
Then $q \Sigma_{q \Psi_H (u,\eta_H) }(u,\eta_H)$ is a constituent of
$\Res_{\End_H (q\pi_*^H (\widetilde{q \cE_H}))}^{\End_G (q\pi_* (\widetilde{q \cE}))} 
q \Sigma_{q \ft}(u,\eta)$ if and only if $\eta_H$ is a constituent of 
$\Res^{A_G (u)}_{A_H (u)}(\eta)$.
}
\end{prop}
\begin{proof}
(a) From Theorem \ref{thm:5.6}.a we know that 
$[M^\circ,\cC_v^{M^\circ},\cE]_G = \Psi_G (u,\eta)$ equals $\Psi_H (u,\eta)$ up to 
$G$-conjugacy. In particular $M^\circ \subset H^\circ$ and, by the assumptions on $H$, 
$M \subset H$. It follows that $\End_M (\pi_*^M \widetilde{\cE_M})$ is also a
subalgebra of $\End_H (\pi_*^H \widetilde{\cE_H})$. By Theorem \ref{thm:5.6}.b we may choose
$\rho_M$ (used in \eqref{eq:6.11} to construct $q \cE$) to be a constituent of $\rho_H =
\Sigma_{\ft_H}(u,\eta_H)$. Then $\Psi_H (u,\eta_H) = [M,\cC_v^M,q \cE]_H$, which agrees with
\eqref{eq:6.1}. \\
(b) By Lemma \ref{lem:6.2} $\End_H \big( q\pi_*^H (\widetilde{q \cE_H}) \big) \cong 
\overline{\Q_{\ell}} [W_{q \ft_H},\kappa_{q \ft_H}]$. Here
\[
W_{q \ft_H} = W_{\ft_H,\rho_M} / (M_\cE / M^\circ) = 
N_H (M^\circ,\cC_v^{M^\circ},\cE) / M_\cE 
\]
is a subgroup of
\[
N_G (M^\circ,\cC_v^{M^\circ},\cE) / M_\cE = W_{\ft,\rho_M} / (M_\cE / M^\circ) = W_{q \ft}.
\]
The 2-cocycle $\kappa_{q \ft_H}$ is just the restriction of $\kappa_{q \ft}$, because both
are based on the same representation $\rho_M$ of $\End_M (\pi_*^M \widetilde{\cE_M})$.
This gives an injection 
\[
\overline{\Q_{\ell}} [W_{q \ft_H},\kappa_{q \ft_H}] \to
\overline{\Q_{\ell}} [W_{q \ft},\kappa_{q \ft}] .
\]
With Lemma \ref{lem:6.2} we get an injection
\[
\End_H (q\pi_*^H (\widetilde{q \cE_H})) \to \End_G (q\pi_* (\widetilde{q \cE})) .
\]
It is natural because every basis element $q b_w \; (w \in W_{q \ft_H})$ of $\End_G ( q\pi_* 
(\widetilde{q \cE}))$ constructed in the proof of Lemma \ref{lem:6.2} stabilizes the subset
$q\pi_*^H (\widetilde{q \cE_H})$ and hence naturally determines an automorphism of that sheaf.

For the group $H$ \eqref{eq:6.13} says 
\[
\Sigma_{\ft_H} (u,\eta_H) = \rho_M \rtimes q\Sigma_{q \ft_H}(u,\eta_H) .
\]
By Theorem \ref{thm:5.6}.b $\Sigma_{\ft_H} (u,\eta_H)$ appears in $\Sigma_{\ft} (u,\eta)$
if and only if $\eta_H$ appears in $\eta$. With Proposition \ref{prop:1.1}.c and 
\eqref{eq:6.13} we see that this is also equivalent to $q \Sigma_{q \ft_H}(u,\eta_H)$ 
appearing in $q \Sigma_{q \ft}(u,\eta)$.
\end{proof}

This concludes the part of the paper which deals exclusively with Springer correspondences.
We remark once more that all the results from Sections \ref{sec:genSpr}--\ref{sec:quasiLevi} 
also hold with $\C$ instead of $\overline{\Q_{\ell}}$.

\section{Cuspidal Langlands parameters}

We will introduce a notion of cuspidality for enhanced L-parameters. 
Before we come to that, we recall some generalities about Langlands parameters
and Levi subgroups. For more background we refer to \cite{Bor,Vog,ABPS7}.

Let $F$ be a local non-archimedean field with Weil group $\mb W_F$. Let
$\cH$ be a connected reductive algebraic group over $F$ and let $\cH^\vee$
be its complex dual group. The data for $\cH$ provide an action of $\mb W_F$
on $\cH^\vee$ which preserves a pinning, and that gives the Langlands dual group 
${}^L \cH = \cH^\vee \rtimes \mb W_F$. (All these objects are determined by
$F$ and $\cH$ up to isomorphism.) 

\begin{defn}\label{def:7.6}
Let $T \subset \cH^\vee$ be a torus such that the projection 
$Z_{\cH^\vee \rtimes \mb W_F}(T) \to \mb W_F$ is surjective. Then we call
${}^L L = Z_{\cH^\vee \rtimes \mb W_F}(T)$ a Levi L-subgroup of ${}^L \cH$. 
\end{defn}
We remark that in \cite{Bor} such groups are called Levi subgroups of 
${}^L \cH$. We prefer to stick to the connectedness of Levi subgroups.

Choose a $\mb W_F$-stable pinning for $\cH^\vee$. This defines the notion of
standard Levi subgroups of $\mc H^\vee$.
An alternative characterization of Levi L-subgroups of ${}^L \cH$ is as follows.

\begin{lem}\label{lem:7.4}
Let ${}^L L$ be a Levi L-subgroup of ${}^L \cH$.
There exists a $\mb W_F$-stable standard Levi subgroup $\cL^\vee$ of $\cH^\vee$ 
such that ${}^L L$ is $\cH^\vee$-conjugate to $\cL^\vee \rtimes \mb W_F$
and $L := {}^L L \cap \cH^\vee$ is conjugate to $\cL^\vee$. 

Conversely, every $\cH^\vee$-conjugate of this $\cL^\vee \rtimes \mb W_F$ 
is a Levi L-subgroup of ${}^L \cH$.
\end{lem}
\begin{proof}
By \cite[Lemma 3.5]{Bor} there exists a parabolic subgroup $P \subset \cH^\vee$
such that 
\begin{itemize}
\item $N_{\cH^\vee \rtimes \mb W_F}(P) \to \mb W_F$ is surjective;
\item $L$ is a Levi factor of $P$;
\item ${}^L L = N_{\cH^\vee \rtimes \mb W_F}(L) \cap N_{\cH^\vee \rtimes \mb W_F}(P)$.
\end{itemize}
To construct such a $P$, choose a $\Z$-linear function $X^* (T) \to \Z$ in generic 
position (i.e. not orthogonal to any coroot). Then we can define $P$ as the subgroup of 
$\cH^\vee$ generated by $L$ and by all root subgroups associated to positive (with respect
to this linear function) cocharacters of $T$.

Let $P_I = L_I \rtimes U_I$
be the unique standard parabolic subgroup of $\cH^\vee$ conjugate to $P$. Here
$U_I$ denotes the unipotent radical of $P_I$, and $L_I$ its standard Levi factor.
Then $N_{\cH^\vee \rtimes \mb W_F}(P_I) \to \mb W_F$ is still surjective, so
$P_I$ is $\mb W_F$-stable. Pick $h \in \cH^\vee$ with $P_I = h P h^{-1}$. Then
$h L h^{-1}$ is a Levi factor of $P_I$ and
\[
h \, {}^L L h^{-1} = N_{P_I \rtimes \mb W_F}(h \, {}^L L h^{-1}) 
\]
is a complement to $U_I$ in $P_I \rtimes \mb W_F$. All Levi factors of $P_I$ are
$U_I$-conjugate, so there exists a $u \in U_I$ with 
$u h \, {}^L L h^{-1} u^{-1} = L_I$. Then
\[
u h \, {}^L L h^{-1} u^{-1} = N_{P_I \rtimes \mb W_F}(L_I) = 
L_I \rtimes \mb W_F .
\]
For the converse, let $\cL^\vee$ be a $\mb W_F$-stable standard Levi subgroup of $\cH^\vee$.
Denote the standard maximal torus of $\mc H^\vee$ by $L_\emptyset$ and consider the root 
system $R := R (\cH^\vee, L_\emptyset)$. By assumption $\mb W_F$ acts on $R$ and stabilizes 
a basis $\Delta$. Let $T \subset L_\emptyset$ be the neutral component of 
$Z( \cL^\vee)^{\mb W_F}$. This is a $\mb W_F$-fixed torus which commutes with $\cL^\vee$ and
\[
\alpha (t) = (w \cdot \alpha)(t) \qquad \forall t \in T, \alpha \in R, w \in \mb W_F. 
\]
The Lie algebra $\mf l_\der$ of 
$L_\emptyset \cap \cH^\vee_\der$ is spanned by $\Delta^\vee$ and $\mb W_F$-stable. Let
$I$ be the set of simple roots in $R (\cL^\vee, L_\emptyset)$. Since $\cL^\vee$ is 
$\mb W_F$-stable, so are $I$ and $\Delta \setminus I$. Let $X \in \mf l_\der$ be an element 
which annihilites $I$ and takes the same value in $\R_{>0}$ on all simple roots not in $I$. 
Then $X \in \mr{Lie}(\mc L^\vee)$ is fixed by $\mb W_F$.
This gives an element $\exp (X) \in T$ with $\alpha (\exp X) = \exp (\alpha (X)) > 1$
for all positive roots in $R \setminus R (\cL^\vee, L_\emptyset)$.

In general, the $\cH^\vee$-centralizer of the torus $T \subset Z(\cL^\vee)^\circ$ is 
generated by $\cL^\vee$ and by the root subgroups $U_\alpha$ for which $\alpha$ 
becomes trivial on $T$. With the above elements $\exp (X)$ we deduce that
\[
Z_{\cH^\vee}(T) = \cL^\vee \quad \text{and} \quad 
Z_{\cH^\vee \rtimes \mb W_F}(T) = \cL^\vee \rtimes \mb W_F .
\]
This means that $\cL^\vee \rtimes \mb W_F$ is a Levi L-subgroup of ${}^L \cH$ in
the sense of Definition \ref{def:7.6}. For any $h \in \cH^\vee$:
\[
h (\cL^\vee \rtimes \mb W_F) h^{-1} = Z_{\cH^\vee \rtimes \mb W_F}(h T h^{-1}) .
\]
This group contains $h \mb W_F h^{-1}$, so it projects onto $\mb W_F$. Hence it
is again a Levi L-subgroup of ${}^L \cH$.
\end{proof}

\begin{rem} 
Most Levi L-subgroups of ${}^L \cH$ are not quasi-Levi, and conversely. For example, let 
$\mc U = \rU(n,E/F)$ be a $p$-adic unitary group ($E$ is a quadratic extension of $F$) and 
let ${}^L \mc U$ be its dual L-group. The group $\bW_F$ acts on $\mc {U}^\vee= \GL(n,\C)$ 
via an outer automorphism which preserves the diagonal torus $T$ and the standard Borel 
subgroup $B$. Then $T \rtimes \bW_F$ is a Levi L-subgroup of ${}^L \mc U$: it is the 
centralizer of $T^{\bW_F}$ in ${}^L \mc U$.
However, it is not quasi-Levi. Namely $Z_{{}^L \mc U}(T) = T \rtimes \bW_E$, which is an 
index two subgroup of $T \rtimes \bW_F$.
\end{rem}

The following definitions are well-known, we repeat them here to facilitate
comparison with later generalizations.

\begin{defn}\label{def:7.7}
A L-parameter for ${}^L \cH$ is a continuous group homomorphism \\
$\phi \colon \mb W_F \times \SL_2 (\C) \to {}^L \cH$ such that: 
\begin{itemize}
\item $\phi (w) \in \cH^\vee w$ for all $w \in \mb W_F$;
\item $\phi (w)$ is semisimple for all $w \in \mb W_F$;
\item $\phi |_{\SL_2 (\C)} : \SL_2 (\C) \to \cH^\vee$ is a homomorphism of
complex algebraic groups.
\end{itemize}
\end{defn}
Recall that all inner forms of $\cH$ share the same Langlands dual group
${}^L \cH$, so the group $\cH$ is not determined by the target ${}^L \cH$
of a L-parameter. Let us specify which L-parameters are relevant for $\cH$, 
and which are bounded or discrete.

\begin{defn}\label{def:7.3}
Let $\phi \colon \mb W_F \times \SL_2 (\C) \to {}^L \cH$ be a L-parameter. 
We say that $\phi$ is bounded if $\phi (\Fr) = (h,\Fr)$ with $h$ in some
compact subgroup of $\cH^\vee$.

Suppose that ${}^L L$ is a Levi L-subgroup of ${}^L \cH$ and that
\begin{itemize}
\item ${}^L L$ contains the image of $\phi$;
\item there is no smaller Levi L-subgroup of ${}^L \cH$ with this property.
\end{itemize}
Then we call $\phi$ relevant for $\cH$ if and only if the conjugacy class of 
${}^L L$ is relevant for $\cH$, that is, it corresponds
to a conjugacy class of Levi subgroups of $\cH$. 

In this case we also say that $\phi$ is a discrete L-parameter for ${}^L L$, 
and for any Levi subgroup $\cL \subset \cH$ in the associated class. In particular 
$\phi$ is discrete for $\cH$ if and only
if there is no proper Levi L-subgroup of ${}^L \cH$ containing the image of $\phi$.
\end{defn}

The group $\cH^\vee$ acts on the set of relevant L-parameters for $\cH$.
We denote the set of relevant $L$-parameters modulo $\cH^\vee$-conjugation
by $\Phi (\cH)$. The subset of bounded L-parameters (up to conjugacy) is denoted by
$\Phi_{\bdd}(\cH)$.
The local Langlands correspondence predicts that $\Irr (\cH)$
is partitioned into finite L-packets $\Pi_\phi (\cH)$, parametrized by $\Phi (\cH)$.
Under this correspondence $\Phi_{\bdd}(\cH)$ should give rise to L-packets
consisting entirely of tempered representations, and that should account
for the entire tempered dual of $\cH$.

It is expected (and established in many cases)
that the following conditions are equivalent for $\phi \in \Phi (\cH)$:
\begin{itemize}
\item $\phi$ is discrete;
\item $\Pi_\phi (\cH)$ contains an essentially square-integrable representation;
\item all elements of $\Pi_\phi (\cH)$ are essentially square-integrable.
\end{itemize}
In other words, ``discrete'' (respectively ``bounded") is the correct translation 
of ``essentially square-integrable'' (respectively ``tempered") under the local 
Langlands correspondence.

However, it is more difficult to characterize when $\Pi_\phi (\cH)$ contains
supercuspidal $\cH$-representations. Of course $\phi$ has to be discrete, but
even then. Sometimes $\Pi_\phi (\cH)$ consists entirely of supercuspidal
representations, for example when $\cH = \SL_2 (F)$ and $\phi$ comes from an
irreducible representation $\mb W_F \to \GL_2 (\C)$. In other cases $\Pi_\phi (\cH)$
contains only non-supercuspidal essentially square-integrable representations,
for example when $\cH = \SL_2 (F) ,\; \phi|_{\mb W_F} = \mathrm{id}_{\mb W_F}$
and $\phi |_{\SL_2 (\C)}$ is an irreducible two-dimensional representation of
$\SL_2 (\C)$. 

Moreover there are mixed cases, where $\Pi_\phi (\cH)$ contains both supercuspidal
and non-supercuspidal representations. An example is formed by a Langlands 
parameter for a group of type $G_2$, with $\phi \big( 1,\matje{1}{1}{0}{1} \big)$
a subregular unipotent element of $G_2 (\C)$. Then $\Pi_\phi (G_2 (F))$ has a 
unique supercuspidal element and contains two representations from the principal
series of $G_2 (F)$, see \cite{Lus2}.

To parametrize the representations in a given L-packet, we need more information 
then just the Langlands parameter itself. Let $Z_{\cH^\vee}(\phi)$ be the 
centralizer of $\phi (\mb W_F \times \SL_2 (\C))$ in $\cH^\vee$. This is a complex
reductive group, in general disconnected. We write 
\begin{equation} \label{eqn:R-group}
\cR_\phi := \pi_0 \big( Z_{\cH^\vee}(\phi) / Z(\cH^\vee)^{\mb W_F} \big).
\end{equation}
It is expected that $\Pi_\phi (\cH)$ is in bijection with $\Irr (\cR_\phi)$ if 
$\cH$ is quasi-split. However, for general $\cH$ this is not good enough, and we 
follow Arthur's setup \cite{Art2}. 

Let $\cH^\vee_{\sc}$ be the simply connected cover of the derived 
group $\cH^\vee_\der$ of $\cH^\vee$. The conjugation action of
$\cH^\vee_\der$ lifts to an action of $\cH^\vee_\sc$ on $\cH$ by conjugation. 
The action of $\mb W_F$ on $\cH^\vee_\der$ lifts to an action on $\cH^\vee_\sc$,
because the latter group is simply connected. Thus we can form the group
$\cH^\vee_\sc \rtimes \mb W_F$. In this semidirect product we can compute $h w h^{-1}$ 
for $h \in \cH^\vee_\sc$ and $w \in \mb W_F$. Dividing out the normal subgroup
$\ker (\cH^\vee_\sc \to \cH^\vee_\der)$, we can interpret $h w h^{-1}$ as an element
of $\cH^\vee_\der \rtimes \mb W_F$. 

Together with the above this provides a conjugation action of $\cH^\vee_\sc$ on 
$\cH^\vee \rtimes \mb W_F$. Hence $\cH^\vee_\sc$ also acts on the set of Langlands 
parameters for $\cH$ and we can form $Z_{\cH^\vee_\sc}(\phi)$.

Since $Z_{\cH^\vee}(\phi) \cap Z(\cH^\vee) = Z(\cH^\vee)^{\mb W_F}$,
\begin{equation}\label{eq:6.19}
Z_{\cH^\vee}(\phi) / Z(\cH^\vee)^{\mb W_F} \cong 
Z_{\cH^\vee}(\phi) Z(\cH^\vee) / Z(\cH^\vee) .
\end{equation} 
The right hand side can be considered as a subgroup of the adjoint group
$\cH^\vee_\ad$. Let $Z^1_{\cH^\vee_\sc}(\phi)$ be its inverse image under the
quotient map $\cH^\vee_\sc \to \cH^\vee_\ad$. We can also characterize it as 
\begin{align}
Z^1_{\cH^\vee_\sc}(\phi) \; = \; \big\{ h \in \cH^\vee_\sc : h \phi h^{-1} = \phi \, a_h
\text{ for some } a_h \in B^1 (\mb W_F ,Z(\cH^\vee)) \big\} \\
\nonumber = \big\{ h \in Z_{\cH^\vee_\sc} \big( \phi (\SL_2 (\C)) \big) : h \phi |_{\mb W_F} h^{-1}
= \phi |_{\mb W_F} \, a_h \text{ for some } a_h \in B^1 (\mb W_F ,Z(\cH^\vee)) \big\} \\
\nonumber = Z^1_{\cH^\vee_\sc}(\phi |_{\mb W_F}) \cap 
Z_{\cH^\vee_\sc} \big( \phi (\SL_2 (\C)) \big) .
\end{align}
Here $B^1 (\mb W_F ,Z(\cH^\vee))$ is the set of 1-coboundaries for group cohomology, 
that is, maps $\mb W_F \to Z(\cH^\vee)$ of the form $w \mapsto z w z^{-1} w^{-1}$ 
with $z \in Z(\cH^\vee)$. The neutral component of $Z^1_{\cH^\vee_\sc}(\phi)$
is $Z_{\cH^\vee_\sc}(\phi)^\circ$, so it is a complex reductive group.

The difference between $Z_{\cH^\vee_\sc}(\phi)$ and 
$Z^1_{\cH^\vee_\sc}(\phi)$ is caused by the identification \eqref{eq:6.19}, which as 
it were includes $Z(\cH^\vee)$ in $Z_{\cH^\vee}(\phi)$. We note that 
$Z^1_{\cH^\vee_\sc}(\phi) = Z_{\cH^\vee_\sc}(\phi)$ whenever $Z(\cH^\vee_\sc)^{\mb W_F} = 
Z(\cH^\vee_\sc)$, in particular if $\cH$ is an inner twist of a split group.

Given $\phi$, we form the finite group
\begin{equation} \label{eqn: S group}
\cS_\phi := \pi_0 \big( Z^1_{\cH^\vee_{\sc}}(\phi) \big) .
\end{equation}
Via \eqref{eq:6.19}, the map $\cH^\vee_\sc \to \cH^\vee_\ad$ induces a homomorphism
$\cS_\phi \to \cR_\phi$. In fact, $\cS_\phi$ is a central extension of 
$\cR_\phi$ by $\cZ_\phi := Z(\cH^\vee_\sc) / Z(\cH^\vee_\sc) \cap 
Z_{\cH^\vee_\sc}(\phi)^\circ$ \cite[Lemma 1.7]{ABPS7}:
\begin{equation}\label{eq:6.24}
1 \to \cZ_\phi \to \cS_\phi \to \cR_\phi \to 1 . 
\end{equation}
Since $\cH^\vee_\sc$ is a central extension of $\cH^\vee_\ad = \cH^\vee / Z(\cH^\vee)$, 
the conjugation action of $\cH^\vee_\sc$ on itself and on $\cS_\phi$ descends to an 
action of $\cH^\vee_\ad$. Via the canonical quotient map, also $\cH^\vee$ acts on 
$\cS_\phi$ by conjugation. 

An enhancement of $\phi$ is defined to be an irreducible complex representation
$\rho$ of $\cS_\phi$. We refer to \cite{Art2,ABPS7} for a motivation
of this particular kind of enhancements. 
We let $\cH^\vee$ and $\cH^\vee_\sc$ act on the set of enhanced L-parameters by
\begin{equation}\label{eq:7.7}
h \cdot (\phi,\rho) = (h \phi h^{-1},h \cdot \rho) \quad \text{where} \quad
(h\cdot \rho)(g) = \rho (h^{-1} g h) .
\end{equation}
We note that both groups acting in \eqref{eq:7.7} yield the same orbit space.

The notion of relevance for enhanced L-parameters is more subtle. Firstly, we
must specify $\cH$ not only as an inner form of a quasi-split group $\cH^*$, 
but even as an inner twist. That is, we must fix an isomorphism $\cH \to \cH^*$
of algebraic groups, defined over a separable closure of $F$. The inner twists
of $\cH$ are parametrized by the Galois cohomology group $H^1 (F,\cH_{\ad})$,
where $\cH_{\ad}$ denotes the adjoint group of $\cH$ (considered as an algebraic
group defined over $F$). The parametrization is canonically determined by 
requiring that $\cH^*$ corresponds to the trivial element of $H^1 (F,\cH_{\ad})$.
Kottwitz \cite[Theorem 6.4]{Kot} found a natural group isomorphism
\begin{equation}\label{eq:7.8}
H^1 (F,\cH_{\ad}) \cong \Irr_\C \big( Z (\cH^\vee_{\sc})^{\mb W_F} \big) .
\end{equation}
(When $F$ has positive characteristic, see \cite[Theorem 2.1]{Tha}.)
In this way every inner twist of $\cH$ is associated to a unique character of
$Z (\cH^\vee_{\sc})^{\mb W_F} = Z( \cH^\vee_{\sc} \rtimes \mb W_F )$. 
The functoriality of the Kottwitz homomorphism implies that this parametrization
behaves well with respect to Levi subgroups. To make this statement precise,
let $\cL$ be a Levi $F$-subgroup of $\cH$. Via $\cH \to \cH^*$, we regard
$\cL$ as an inner twist of a quasi-split Levi subgroup $\cL^*$ of $\cH^*$. 
Let $\cL_c^\vee$ be the inverse image of $\cL^\vee$ under $\cH_\sc^\vee \to
\cH$. It contains $\cL_\sc^\vee$ as the derived subgroup of $\cL_c^\vee$.
The next lemma is a variation on \cite[Lemma 0.4.9]{KMSW}, tailored for our purposes.

\begin{lem}\label{lem:Kottwitz}
\enuma{
\item The centers of $\cH_\sc^\vee, \cL_c^\vee$ and $\cL_\sc^\vee$ are related by
\[
Z (\cH_\sc^\vee)^{\mb W_F} Z(\cL_c^\vee)^{\mb W_F,\circ} = Z(\mc L_c^\vee)^{\mb W_F}
\supset Z(\mc L_\sc^\vee)^{\mb W_F} .
\]
\item The character of $Z(\cH_\sc^\vee)^{\mb W_F}$ determined by \eqref{eq:7.8} is
trivial on \\ $Z(\cH_\sc^\vee)^{\mb W_F} \cap Z(\cL_c^\vee)^{\mb W_F,\circ}$. Using part
(a) we extend it to $Z(\mc L_c^\vee)^{\mb W_F}$, trivially on $Z(\cL_c^\vee)^{\mb W_F,\circ}$.
Then the character of $Z(\mc L_\sc^\vee)^{\mb W_F}$ obtained by restriction equals
the character of $Z(\mc L_\sc^\vee)^{\mb W_F}$ associated to $\cL$ by \eqref{eq:7.8}.
}  
\end{lem}
\begin{proof}
(a) See \cite[Lemma 1.1]{Art1}.\\
(b) The morphisms of reductive $F$-groups $\cH_\ad \leftarrow \cL / Z(\cH) \to \cL_\ad$
induce the following commutative diagram:
\[
\xymatrix{ 
H^1 (F,\cH_\ad) \ar[d] & H^1 (F,\mc L / Z(\mc H)) \ar[l] \ar[r] \ar[d] & H^1 (F,\cL_\ad) \ar[d] \\
\Irr \big( Z (\cH_\sc^\vee)^{\mb W_F} \big) & \Irr \big( \pi_0 (Z(\mc L_c^\vee)^{\mb W_F}) \big)
\ar[r] \ar[l] & \Irr \big( Z(\mc L_\sc^\vee)^{\mb W_F} \big) .
} 
\]
All the vertical arrows are isomorphisms, and according to \cite[p. 217]{Art1} the left 
horizontal arrows are injective. Since $\cL$ is a Levi $F$-subgroup of $\cH$, the element of
$H^1 (F,\cH_\ad)$ which parametrizes $\cH$ can be represented by a Galois cocycle with values in
the Levi subgroup $\cL / Z(\cH)$ of $\cH_\ad$. This cocycle maps naturally to an element of
$H^1 (F,\cL_\ad)$, which then parametrizes the inner twist $\cL$ of $\cL^*$.

On the bottom line of the diagram, the associated character of $Z(\cL_\sc^\vee)^{\mb W_F}$
must come from a character of $\pi_0 (Z(\mc L_c^\vee)^{\mb W_F})$, via the map 
$Z(\cL_\sc^\vee)^{\mb W_F} \to \pi_0 (Z(\mc L_c^\vee)^{\mb W_F})$ from part (a). Hence this
character is trivial on $Z(\cH_\sc^\vee)^{\mb W_F} \cap Z(\cL_c^\vee)^{\mb W_F,\circ}$. 
The character of $Z(\cL_\sc^\vee)^{\mb W_F}$ associated to $\cL$ is then obtained by applying
the lower right map in the diagram. This works out as restriction to $Z(\cL_\sc^\vee)^{\mb W_F}$,
in the indicated way.
\end{proof}

Given any Langlands parameter $\phi$ for ${}^L \cH$, there is a natural group 
homomorphism $Z (\cH^\vee_{\sc})^{\mb W_F} \to Z(\cS_\phi)$. The centre of 
$\cS_\phi$ acts by a character on any $\rho \in \Irr_\C (\cS_\phi)$, so
any enhancement $\rho$ of $\phi$ determines a character $\zeta_\rho$ of
$Z (\cH^\vee_{\sc})^{\mb W_F}$.

\begin{defn}\label{def:7.9}
Let $(\phi,\rho)$ be an enhanced L-parameter for ${}^L \cH$. We say that
$(\phi,\rho)$ or $\rho$ is $\cH$-relevant if $\zeta_\rho$ parametrizes the
inner twist $\cH$ via \eqref{eq:7.8}.
\end{defn}

By the next result, Definition \ref{def:7.8} fits well with the earlier
notion of relevance of $\phi$, as in Definition \ref{def:7.3}.

\begin{prop}\label{prop:7.10}
Let $\cH$ be an inner twist of a quasi-split group and let $\zeta \in \Irr_\C 
\big( Z (\cH^\vee_{\sc})^{\mb W_F} \big)$ be the associated character.
Let $\phi$ be a Langlands parameter for ${}^L \cH$. The following are
equivalent:
\begin{enumerate}
\item $\phi$ is relevant for $\cH$;
\item $Z (\cH^\vee_{\sc})^{\mb W_F} \cap Z_{\cH^\vee_{\sc}}(\phi)^\circ 
\subset \ker \zeta$;
\item there exists a $\rho \in \Irr_\C (\cS_\phi)$ with $\zeta_\rho = 
\zeta$, that is, such that $(\phi,\rho)$ is $\cH$-relevant.
\end{enumerate}
\end{prop}
\begin{proof}
For the equivalence of (1) and (2) see \cite[Lemma 9.1]{HiSa} and 
\cite[Corollary 2.2]{Art1}. The equivalence of (2) and (3) is easy, it
was already noted in \cite[Proposition 1.6]{ABPS7}. 
\end{proof}

Let us remark here that the usage of $\cH^\vee_{\sc}$ and the above relevance 
circumvents some of the problems in \cite[\S 2]{Vog}. In particular it removes 
the need to consider variations such as "pure inner forms" or "pure inner twists".

We denote the set of $\cH^\vee$-equivalence classes of enhanced relevant
L-parameters for $\cH$ by $\Phi_e (\cH)$. Following \cite{Art2} we choose an
extension $\zeta_\cH$ of $\zeta$ to a character of $Z(\cH^\vee_\sc)$. We define
\begin{equation}\label{eq:6.25}
\Phi_{e,\zeta_\cH}(\cH) = \{ (\phi,\rho) \in \Phi_e (\cH) : 
\zeta_\cH \, \mr{id}_{V_\rho} = \rho \circ (Z(\cH^\vee_\sc) \to \cS_\phi) \} ,
\end{equation}
where $V_\rho$ is the vector space underlying $\rho$.
According to \cite[\S 4]{Art2}
\[
Z(\cH^\vee_\sc) \cap Z_{\cH^\vee_\sc}(\phi)^\circ = 
Z(\cH^\vee_\sc)^{\mb W_F} \cap Z_{\cH^\vee_\sc}(\phi)^\circ .
\]
Hence every extension of $\zeta$ to a character of $Z(\cH^\vee_\sc)$ is eligible
if $\phi$ is $\cH$-relevant. Of course we take $\zeta_\cH =$ triv if $\cH$ is 
quasi-split. Since $\cS_\phi / \cZ_\phi \cong \cR_\phi$, we obtain
\[
\Phi_{e,\mathrm{triv}}(\cH) = \{ (\phi,\rho) : \phi \in \Phi (\cH),
\rho \in \Irr (\cR_\phi) \} \qquad \text{if } \cH \text{ is quasi-split.} 
\]
It is conjectured \cite{Art2,ABPS7} that the local
Langlands correspondence for $\cH$ can be enhanced to a bijection
\[
\Irr (\cH) \longleftrightarrow \Phi_{e,\zeta_\cH} (\cH) . 
\]
Recall that by the Jacobson--Morosov theorem any unipotent element $u$ of \\
$Z_{\cH^\vee_{\sc}}(\phi (\mb W_F) )^\circ$ can be extended to a homomorphism of
algebraic groups $\SL_2 (\C) \to  Z_{\cH^\vee_{\sc}}(\phi (\mb W_F) )^\circ$ taking 
the value $u$ at $\matje{1}{1}{0}{1}$. Moreover, by \cite[Theorem 3.6]{Kos} this 
extension is unique up to conjugation. 
Hence any element $(\phi,\rho) \in \Phi_e (\cH)$ is already determined by 
$\phi |_{\mb W_F}, u_\phi = \phi \big( 1, \matje{1}{1}{0}{1} \big)$ and $\rho$. 
More precisely, the map 
\begin{equation}\label{eq:7.5}
\phi \mapsto \Big( \phi |_{\mb W_F},u_\phi = 
\phi \big( 1, \matje{1}{1}{0}{1} \big) \Big)
\end{equation}
provides a bijection between $\Phi (\cH)$ and the $\cH^\vee$-conjugacy classes of
pairs $(\phi |_{\mb W_F},u_\phi)$. The inclusion $Z^1_{\cH^\vee_{\sc}}(\phi) \to 
Z^1_{\cH^\vee_{\sc}}(\phi |_{\mb W_F}) \cap Z_{\cH^\vee_\sc}(u_\phi)$ 
induces a group isomorphism
\begin{equation}\label{eq:7.1}
\cS_\phi \to \pi_0 \big( Z^1_{\cH^\vee_{\sc}}(\phi |_{\mb W_F}) \cap 
Z_{\cH^\vee_\sc}( u_\phi) \big) .
\end{equation}
We will often identify $\Phi_e (\cH)$ with the set of $\cH^\vee$-equivalence
classes of such triples $(\phi |_{\mb W_F}, u_\phi, \rho)$. Another way to formulate
\eqref{eq:7.1} is
\begin{equation}\label{eq:7.2}
\cS_\phi \cong \pi_0 (Z_G (u_\phi)) \quad \text{where} \quad
G = Z^1_{\cH^\vee_{\sc}}(\phi |_{\mb W_F} ) \text{ and } 
u_\phi = \phi \big( 1, \matje{1}{1}{0}{1} \big) .
\end{equation}
We note also that there is a natural bijection between the set of unipotent elements 
in $\cH^\vee$ and those in $\cH^\vee_{\sc}$, so we may take $u_\phi$ in either of 
these groups.

Based on many examples we believe that the following kind of enhanced L-parameters 
should parametrize supercuspidal representations.

\begin{defn}\label{def:7.1}
An enhanced L-parameter $(\phi,\rho)$ for ${}^L \cH$ is cuspidal if $\phi$ is
discrete and $(u_\phi,\rho)$ is a cuspidal pair for $G = Z^1_{\cH^\vee_{\sc}}(\phi 
|_{\mb W_F})$. Here $\rho$ is considered as a representation of $\pi_0 (Z_G (u_\phi))$
via \eqref{eq:7.2}.

We denote the set of $\cH^\vee$-equivalence classes of $\cH$-relevant cuspidal
L-parameters by $\Phi_\cusp (\cH)$. When $\zeta_\cH$ is as in \eqref{eq:6.25}, we
put $\Phi_{\cusp,\zeta_\cH} (\cH) = \Phi_\cusp (\cH) \cap \Phi_{e,\zeta_\cH}$.
\end{defn}

It is easy to see that every group $\cH$ has cuspidal L-parameters. 
Let $\phi \in \Phi (\cH)$ be a discrete parameter which is trivial on $\SL_2 (\C)$. 
Then $u_\phi = 1$ and $Z_{\cH^\vee}(\phi)^\circ = Z(\cH^\vee)^{\mb W_F,\circ}$. Hence  
$G = Z^1_{\cH^\vee_\sc}(\phi)$ is finite and every enhancement $\rho$ of $\phi$ is 
cuspidal. By Proposition \ref{prop:7.10} we can choose a $\cH$-relevant $\rho$.

In the case of quasi-split groups we can also use enhanced L-parameters of the form
$(\phi,\rho)$ with $\rho \in \Irr (\cR_\phi)$, where $\cR_\phi$ is as in 
\eqref{eqn:R-group}. Such a parameter is cuspidal if and only if $(u_\phi,\rho)$ is
a cuspidal pair for $Z_{\cH^\vee}(\phi)$.

\begin{conj} \label{conj:cusp}
Let $\mc H$ be any reductive $p$-adic group, and choose a character $\zeta_{\mc H}$
of $Z(\mc H_\sc^\vee)$ whose restriction to $Z(\mc H_\sc^\vee)^{\mb W_F}$ parametrizes
$\mc H$ via the Kottwitz homomorphism. Under the local Langlands correspondence,
$\Phi_{\cusp,\zeta_\cH} (\cH)$ is in bijection with the set of supercuspidal irreducible 
smooth $\mc H$-representations (up to isomorphism).
\end{conj}

Now we check that, in many cases where a local Langlands correspondence is known,
Conjecture \ref{conj:cusp} holds.
 
\begin{ex} \label{ex: cusp 3}
Let $F$ be a $p$-adic field, $D$ a division algebra over $F$ such that $\dim_F D=d^2$ 
and $\cH=\GL_m(D)$. Then $\cH$ is an inner form of $\GL_n(F)$ with $n=md$.
Let $(\phi,\rho) \in \Phi_\cusp (\cH)$. We have $\cH^\vee_{\sc}=\SL_n(\C)$ and 
${}^{L} \cH=\GL_n(\C) \times \mb W_F$. Since $\phi$ is discrete, it is an 
irreducible representation of $\mb W_F \times \SL_2 (\C)$ and
\[
\cS_\phi=\pi_0(Z_{\SL_{n}(\C)}(\phi)) = Z (\SL_n (\C)) \cong \Z / n \Z .
\]
Because $(\phi,\rho)$ is relevant for $\cH$, $\rho$ is a character of $\cS_\phi$
of order $d$. Furthermore $\phi$ decomposes as
\[
\phi = \pi \boxtimes S_{\pi} \text{ with } 
\pi \in \Irr (\mb W_F), S_\pi \in \Irr (\SL_2 (\C)) .
\]
Let $d'$ denote the dimension of $S_{\pi}$. We will use same argument as in 
\cite[p. 247]{Lus1}. Choose an isomorphism $M_n (\C) \cong M_{n/d'}(\C) \otimes M_{d'}(\C)$
and let $1_{n/d'}$ be the multiplicative unit of the matrix algebra $M_{n/d'}(\C)$. Then
\[
G = Z_{\SL_{n}(\C)}(\phi({\mb W_F})) \simeq 
\big( 1_{n/d'} \otimes \GL_{d'}(\C) \big) \cap \SL_n(\C).
\]
Since we assume that $(\phi,\rho)$ is cuspidal, this implies that $u_\phi$ is in 
the regular unipotent class of $\GL_{d'}(\C)$, and $Z(\SL_{d'}(\C))$ acts on $\rho$ 
by a character of order $d'$. The kernel of the $Z(\SL_n (\C))$-character $\rho$ 
consists precisely of the $d$-th powers in $Z (\SL_n (\C))$. This is possible if
and only if no such $d$-th power is a nontrivial element of $1_{n/d'} \otimes
Z(\SL_{d'}(\C))$. Thus the only additional condition on $d'$ becomes:
lcm$(d,n/d') = n$. 

By the local Langlands correspondence for $\GL_m (D)$ (see \cite[\S 11]{HiSa} and
\cite[\S 2]{ABPS3}) $\phi$ is associated to a unique essentially square integrable
representation $\pi_\phi$ of $\GL_m (D)$. According to \cite[Th\'eor\`eme B.2.b]{DKV}
$\pi_\rho$ is supercuspidal if and only if lcm$(d,n/d') = n$. Consequently the
LLC for $\GL_m (D)$ restricts to a bijection between $\Phi_\cusp (\GL_m (D))$
and $\Irr_\cusp (\GL_m (D))$.
 
We recover the case $\GL_n (F)$ when $D=F$ and $\phi=\pi$ is an irreducible 
representation of $\mb W_F$. An other case is when $\cH=\GL_1(D)$ with $d=2$. 
We find that the cuspidal L-parameters of $\GL_1(D)$ come in two forms: 
\begin{itemize}
\item $(\pi ,\mr{id}_{Z(\SL_2 (\C))})$ with $\pi$ an irreducible two-dimensional
representation of $\mb W_F$;
\item $(\chi \boxtimes S_2 ,\mr{id}_{Z(\SL_2 (\C))})$, with $\chi$ a character of 
$\mb W_F$ and $S_2$ the irreducible two-dimensional representation of $\SL_2 (\C)$.
\end{itemize}
The Langlands parameter in the latter case corresponds to the character 
$\hat \chi \circ \mr{Nrd}$ of $GL_1 (D)$ and to the $GL_2 (F)$-representation 
$\hat \chi \circ \det \otimes \St_{\GL_2(F)}$. These two representations are 
connected by the Jacquet--Langlands correspondence.
\end{ex} 
 
\begin{ex} \label{ex: cusp 4}
Let $F$ be a $p$-adic field, and let $\cH$ be a symplectic group $\Sp_{2n}(F)$ or a split 
special orthogonal group $\SO_{m}(F)$.
We have ${}^{L} \cH=\cH^\vee \times \mb W_F$. Then \cite[Proposition~4.14]{Mou} shows, using 
results of Arthur and M\oe glin, that the supercuspidal irreducible representations of $\cH$ 
correspond, via the local Langlands correspondence, to cuspidal enhanced L-parameters.   
\end{ex}

\begin{ex} \label{ex: cusp unitary}
Let $F$ be a $p$-adic field and $E$ a quadratic extension of $F$. Let $\mathcal{H}=U_n(F)$ 
be the quasi-split unitary group defined over $F$ and split over $E$. We have 
${}^L \mathcal{H}=\GL_n(\C)\rtimes\Gal(E/F)$.
Let $\phi \colon\mb W_F \times \SL_2 (\C) \longrightarrow {}^L \mathcal{H}$ be a discrete 
Langlands parameter and fix $\sigma \in \mb W_F$ such that $\mb W_F / \mb W_E \simeq \langle 
\sigma \rangle$. We use the notions of conjugate-dual, conjugate-orthogonal and 
conjugate-symplectic defined in \cite[\textsection 3]{GGP}. We can decompose the 
restriction of $\phi$ to $\mb W_E$ as an $n$-dimensional representation: 
\begin{equation}\label{eq:6.29}
\restriction{\phi}{\mb W_E} = \bigoplus_{\pi \in I_{\O}^E} m_{\pi} \pi \;
\oplus \bigoplus_{\pi \in I_{\rS}^E} m_{\pi} \pi \; \oplus \bigoplus_{\pi \in I_{\GL}^{E}} m_{\pi} 
\left(\pi \oplus {}^{\sigma} \pi^{\vee} \right),
\end{equation}
where \begin{itemize}
\item $I_\O^E$ is a set of irreducible conjugate-orthogonal representations of $\mb W_E$ ;
\item $I_\rS^E$ is a set of irreducible  conjugate-symplectic representations of $\mb W_E$ ;
\item $I_{\GL}^E$ is a set of irreducible representations of $\mb W_E$ 
which are not conjugate-dual.
\end{itemize} 
Then, by \cite[p.15]{GGP} 
\[
Z_{\mathcal{H}^{\vee}}(\phi(\mb W_F)) \simeq \prod_{\pi \in I_\O^E} 
\O_{m_{\pi}}(\C) \times \prod_{\pi \in I_\rS^E} \Sp_{m_{\pi}}(\C) \times 
\prod_{\pi \in I_{\GL}^E} \GL_{m_{\pi}}(\C).
\]
Every term $m_\pi \pi$ in \eqref{eq:6.29} can be decomposed as $\oplus_a \pi \boxtimes S_a$,
where $S_a$ denotes the $a$-dimensional irreducible representation of $\SL_2 (\C)$. Here 
$a$ runs through some subset of $\N$ -- every $a$ appears at most once because $\phi$ is
discrete. For every such $(\pi,a)$ we choose an element $z_{\pi,a} \in A_{\GL_n (\C)}(\phi)$
which acts as $-1$ on $\pi \boxtimes S_a$ and as the identity on all other factors 
$\pi' \boxtimes S_{a'}$.

From now on we assume that $\phi$ can be enhanced to a cuspidal L-parameter. 
The above and the classification of cuspidal pairs in \cite{Lus1} show that 
$u_{\phi}=(u_{\phi,\pi})$ satisfies: 
\begin{itemize}
\item if $\pi \in I_\O^E$, then the partition associated to $u_{\phi,\pi}$ is 
$(1,3,\ldots,2d_{\pi}-1)$,\\ 
$A_{\O_{m_{\pi}}(\C)}(u_{\phi,\pi})=\prod_{a=1}^{d_{\pi}} \langle z_{\pi,2a-1} \rangle 
\simeq (\Z/2\Z)^{d_{\pi}}$ and $\varepsilon \in \Irr(A_{\O_{m_{\pi}}(\C)}(u_{\phi,\pi}))$ 
is given by $\varepsilon(z_{\pi,2a-1})=(-1)^a$ or $\varepsilon(z_{\pi,2a-1})=(-1)^{a+1}$;
\item if $\pi \in I_\rS^E$, then the partition associated to $u_{\phi,\pi}$ is 
$(2,4,\ldots,2d_{\pi})$,\\
$A_{\Sp_{m_{\pi}}(\C)}(u_{\phi,\pi})=\prod_{a=1}^{d_{\pi}} \langle z_{\pi,2a} \rangle 
\simeq (\Z/2\Z)^{d_{\pi}}$ and $\varepsilon \in \Irr(A_{\Sp_{m_{\pi}}(\C)}(u_{\phi,\pi}))$ 
is given by $\varepsilon(z_{\pi,2a})=(-1)^a$;
\item if $\pi \in I_{\GL}^E$, then $m_{\pi}=1$ and $u_{\phi,\pi}=1$.
\end{itemize} 
Because $\phi$ is discrete, $I_{\GL}^E$ is empty. Hence
\begin{equation}\label{eq:6.18}
\restriction{\phi}{\mb W_E \times \SL_2 (\C)} = \bigoplus_{\pi \in I_\O^E} 
\bigoplus_{a=1}^{d_{\pi}} \pi \boxtimes S_{2a-1} \; \oplus\bigoplus_{\pi \in I_\rS^E} 
\bigoplus_{a=1}^{d_{\pi}} \pi \boxtimes S_{2a}.
\end{equation}
Moreover, in \cite[Th\'{e}or\`eme 8.4.4]{Moe}, M\oe glin classified the supercuspidal 
representations in an Arthur packet. In particular, for tempered Langlands parameters 
(which are Arthur parameters trivial on the second copy of $\SL_2(\C)$), the description is 
given in term of a Jordan block and a character defined by this Jordan block. Here the Jordan 
block $\Jord(\phi)$ of the Langlands parameter $\phi$ of a supercuspidal representation of 
$\mathcal{H}$ is the set of pairs $(\pi,a)$, where $\pi$ is an irreducible representation of 
$\mb W_E$ stable under the action of the composition of inverse-transpose and $\sigma$, and 
$a$ is an integer such that $\pi \boxtimes S_a$ is a subrepresentation of 
$\restriction{\phi}{\mb W_E}$. 

The condition 
on the Jordan block says that it has no holes (or is without jumps). More explicitly, for all 
$a>2$, if $(\pi,a) \in \Jord(\phi)$ then $(\pi,a-2) \in \Jord(\phi)$. The shape of $\phi$ is 
then as \eqref{eq:6.18}. Moreover, the alternated characters are exactly the cuspidal ones. 
More precisely, \cite[p.194]{Moe} gives the definition $z_{\pi,a}$ as our $z_{\pi,a}z_{\pi,a-2}$ 
(or $z_{\pi,2}$ in the case of $a=2$). But the cuspidal characters are exactly the characters 
which are alternated, i.e. such that $\varepsilon(z_{\pi,a}z_{\pi,a-2})=-1$.
\end{ex}

\begin{ex} \label{ex: cusp 1}
Let $\phi$ be a relevant discrete L-parameter which is trivial on the wild inertia subgroup 
$\mf P_F$ of the inertia group $\bI_F$, and such that the centralizer of $\phi (\bI_F)$ 
in $\cH^\vee$ is a torus. The latter condition forces $\phi$ to be trivial on $\SL_2(\C)$. 
Hence $u_\phi=1$, and any enhancement of $\phi$ gives a cuspidal L-parameter. Let
\[
C_\phi = \pi_0 \big( Z_{\cH^\vee}(\phi) / Z({}^L \cH)^\circ \big)
\]
and let $\rho \in \Irr (C_\phi)$. It is known from \cite{DeRe} that
these enhanced L-parameters $(\phi,\rho)$ correspond to the depth-zero generic 
supercuspidal irreducible representations of $\cH$, in the case where $\cH$ is a pure inner 
form of an unramified reductive $p$-adic group. We note that the component group 
$C_\phi$ is a quotient of our $\cS_\phi$, namely by the kernel of $\cH^\vee_{\sc} \to \cH^\vee$.
A priori in these references only a subset of our enhancements of $\phi$ is considered. However, 
it boils down to the same, because the $p$-adic group $\cH$ is chosen such that $\rho$ 
is relevant for it \cite[\S 2]{DeRe}.
\end{ex}

\begin{ex} \label{ex: very cusp} 
Let $(\phi,\rho)$ be a relevant enhanced L-parameter such that 
$\phi$ is discrete and trivial on $\mf P_F^{(r+1)}$ and nontrivial on $\mf P_F^{(r)}$ for some 
integer $r>0$, and such that the centralizer in $\cH^\vee$ of $\phi(\mf P_F^{(r)})$ is a 
maximal torus of $\cH^\vee$. Again any such $(\phi,\rho)$ is cuspidal. The same argument as in 
Example~\ref{ex: cusp 1} shows that the result of Reeder in \cite[\S 6]{Ree3} implies that 
these enhanced L-parameters correspond to the depth $r$ generic 
supercuspidal irreducible representations of $\cH$, when $\cH$ is a pure inner 
form of an unramified reductive $p$-adic group.
\end{ex}

\section{The cuspidal support of enhanced L-parameters}

In the representation theory of $p$-adic groups Bernstein's cuspidal support map
(see \cite[\S 2]{BeDe} or \cite[VI.7.1]{Ren}) plays an important role. It assigns to 
every irreducible smooth $\cH$-representation $\pi$ a Levi subgroup $\cL$ of $\cH$ 
and a supercuspidal $\cL$-representation $\sigma$, such that $\pi$ is contained in 
the normalized parabolic induction of $\sigma$. This condition 
determines $(\cL,\sigma)$ uniquely up to $\cH$-conjugacy. 
It is common to call $(\cL,\sigma)$ a cuspidal pair for $\cH$. The cuspidal 
support of $\pi \in \Irr (\cH)$ is a $\cH$-conjugacy class of cuspidal pairs,
often denoted by $\mathbf{Sc}(\pi)$.

It is expected that {\bf Sc} relates very well to the LLC. In fact this is a 
special case of a conjecture about the relation with parabolic induction, see
\cite[Conjecture 5.22]{Hai} and \cite[\S 1.5]{ABPS7}. Suppose that $\mc P = 
\cL \mc{U}_{\mc P}$ is a parabolic subgroup of $\cH$, that $\phi \in \Phi (\cL)$
and $\sigma \in \Pi_\phi (\cL)$. Then the L-packet $\Pi_\phi (\cH)$ should consist
of constituents of the normalized parabolic induction $I_{\mc P}^\cH (\sigma)$.

We will define an analogue of {\bf Sc} for enhanced L-parameters.
In this setting a cuspidal pair for ${}^L \cH$ should become a triple
$(\cL^\vee \rtimes \mb W_F , \phi, \rho)$, where $\cL^\vee \rtimes \mb W_F$ is
the L-group of a Levi subgroup $\cL \subset \cH$ and $(\phi,\rho)$ is a cuspidal
L-parameter for $\cL$. However, the collection of such objects is not stable
under $\cH^\vee$-conjugation, because $h \cL^\vee h^{-1}$ need not be 
$\mb W_F$-stable. To allow $\cH^\vee$ to act on these triples,
we must generalize Definition \ref{def:7.1} in a less restrictive way.

\begin{defn}\label{def:7.5}
Let ${}^L L$ be a Levi L-subgroup of ${}^L \cH$. A Langlands parameter for 
${}^L L$ is a group homomorphism $\phi : \mb W_F \times \SL_2 (\C) \to {}^L L$
satisfying the requirements of Definition \ref{def:7.7}. An enhancement of
$\phi$ is an irreducible representation $\rho$ of $\pi_0 (Z^1_{L_{\sc}}(\phi))$, 
where $L_{\sc}$ is the simply connected cover of the derived group of 
$L = {}^L L \cap \cH^\vee$. The group $L$ acts on the collection of enhanced 
L-parameters for ${}^L L$ by \eqref{eq:7.7}. 

We say that $(\phi,\rho)$ is cuspidal for ${}^L L$ if $\phi$ is discrete for 
${}^L L$ and $\big( u_\phi = \phi \big( 1, \matje{1}{1}{0}{1} \big), \rho \big)$ 
is a cuspidal pair for $Z^1_{L_{\sc}}(\phi | _{\mb W_F})$.
We denote the set of $L$-orbits by $\Phi_e ({}^L L)$ and the subset
of cuspidal $L$-orbits by $\Phi_\cusp ({}^L L)$. 
\end{defn}

We remark that in this definition it is not specified for which $p$-adic group
an enhanced L-parameter for ${}^L L$ is relevant. Hence $\Phi_e ({}^L \cL)$ is in
general strictly larger than $\Phi_e (\cL)$, it also contains enhanced L-parameters
for inner forms of $\cL$.

Let $L_c$ be the pre-image of $L$ under under $\cH^\vee_\sc \to \cH^\vee$.
Since $L$ is a Levi subgroup of $\cH^\vee$, the derived group of $L_c$ is the simply
connected cover of $L_\der$. Thus we identify $L_{\sc}$ with the inverse image of $L_\der$ 
under $\cH^\vee_{\sc} \to \cH^\vee$. 

\begin{defn}\label{defn:cuspdat}
A cuspidal datum for ${}^L \cH$ is a triple $({}^L L,\phi,\rho)$ as in Definition 
\ref{def:7.5}, such that $(\phi,\rho)$ is cuspidal for ${}^L L$.
It is relevant for $\cH$ if 
\begin{itemize}
\item $\rho = \zeta$ on $L_\sc \cap Z(\cH^\vee_\sc)^{\mb W_F}$, where 
$\zeta \in \Irr (Z(\cH^\vee_\sc)^{\mb W_F})$ parametrizes the inner twist $\cH$ via 
the Kottwitz isomorphism \eqref{eq:7.8}.
\item $\rho = 1$ on $L_\sc \cap Z(L_c)^\circ$.
\end{itemize}
\end{defn}

For $h \in \cH^\vee_{\sc}$ the conjugation action 
\[
L \to h L h^{-1} : l \mapsto h l h^{-1} 
\]
stabilizes the derived group of $L$ and lifts to $L_{\sc} \to (h L h^{-1})_{\sc}$.
Using this, $\cH^\vee_{\sc}$ and $\cH^\vee$ act naturally on cuspidal data for 
${}^L \cH$ by
\[
h \cdot ({}^L L, \phi, \rho) = (h \, {}^L L h^{-1}, h \phi h^{-1}, h \cdot \rho) . 
\]
By Lemma \ref{lem:7.4} every cuspidal datum for ${}^L \cH$ is $\cH^\vee$-conjugate 
to one of the form $(\cL^\vee \rtimes \mb W_F , \phi, \rho)$, where $\cL^\vee$ is
a $\mb W_F$-stable standard Levi subgroup of $\cH^\vee$.
For $\zeta_{\cH} \in \Irr (Z(\cH^\vee_\sc))$ we write
\begin{align*}
& \Phi_{e,\zeta_\cH}({}^L L) = \{ (\phi,\rho) \in \Phi_e ({}^L L) :
\rho = \zeta_{\mc H} \text{ on } L_\sc \cap Z(\cH^\vee_\sc), 
\rho = 1 \text{ on } L_\sc \cap Z(L_c)^\circ \} , \\
& \Phi_{\cusp,\zeta_\cH}({}^L L) = \Phi_\cusp ({}^L L) \cap \Phi_{e,\zeta_\cH}({}^L L) .
\end{align*}
This depends only on the restriction of $\zeta_\cH$ to the subgroup
$Z(L_\sc) \subset Z(\cH^\vee_\sc)$.

Often we will be interested in cuspidal data up to $\cH^\vee$-conjugacy.
Upon fixing the first ingredient of $({}^L L, \phi, \rho)$, we can consider 
$(\phi,\rho)$ as a cuspidal L-parameter for ${}^L L$, modulo $L$-conjugacy.
Recall from \eqref{eq:7.5} that $\phi$ is determined up to $L$-conjugacy
by $\phi |_{\mb W_F}$ and $u_\phi$. Hence the quadruple
\begin{equation}\label{eq:7.6}
({}^L L, \phi |_{\mb W_F}, u_\phi, \rho)
\end{equation}
determines a unique $\cH^\vee$-conjugacy class of cuspidal data. 
Therefore we will also regard quadruples of the form \eqref{eq:7.6} as cuspidal
data for ${}^L \cH$.

Let $\Irr_\cusp (\cL)$ be the set of supercuspidal $\cL$-representations and let
$\sigma_1, \sigma_2 \in \Irr_\cusp (\cL)$. We note that the cuspidal pairs
$(\cL,\sigma_1)$ and $(\cL,\sigma_2)$ are $\cH$-conjugate if and only if $\sigma_1$
and $\sigma_2$ are in the same orbit under
\begin{equation}\label{eq:C1}
W(\cH,\cL) = N_\cH (\cL) / \cL .
\end{equation}
Recall from \cite[Proposition 3.1]{ABPS7} that there is a canonical isomorphism
\begin{equation}\label{eq:8.10}
W( \cH,\cL) \cong N_{\cH^\vee} (\cL^\vee \rtimes \mb W_F) / \cL^\vee . 
\end{equation}
Motivated by \eqref{eq:8.10} we write, for any Levi L-subgroup ${}^L L$
of ${}^L \cH$:
\[
W({}^L \cH, {}^L L) := N_{\cH^\vee}({}^L L) / L . 
\]
This group acts naturally on the collection of cuspidal data for ${}^L \cH$ 
with first ingredient ${}^L L$. Two cuspidal data 
\begin{multline}\label{eq:C2}
({}^L L,\phi_1,\rho_1) \text{ and } ({}^L L,\phi_2,\rho_2) 
\text{ are } \cH \text{-conjugate} \qquad \Longleftrightarrow \\
(\phi_1,\rho_1), (\phi_2,\rho_2) \in \Phi_\cusp ({}^L L) 
\text{ are in the same orbit for the action of } W({}^L \cH, {}^L L).
\end{multline}
In the notation of \eqref{eq:7.2}, we use Section \ref{sec:quasiLevi} 
(with complex representations and sheafs) to write 
\[
q \Psi_G (u_\phi,\rho) = [M,v,q \epsilon]_G , \text{ where } 
G = Z^1_{\cH^\vee_{\sc}}(\phi |_{\mb W_F}) .
\]

\begin{prop}\label{prop:7.2}
Let $(\phi,\rho) \in \Phi_e (\cH)$.
\enuma{
\item $(Z_{\cH^\vee \rtimes \mb W_F}(Z(M)^\circ),\phi |_{\mb W_F},v,q \epsilon)$
is a $\cH$-relevant cuspidal datum for ${}^L \cH$.
\item Upon replacing $(\phi,\rho)$ by a $\cH^\vee$-conjugate representative
L-parameter, there exists a Levi subgroup $\cL$ of $\cH$ such that:
\begin{itemize} 
\item $Z_{\cH^\vee \rtimes \mb W_F}(Z(M)^\circ) = \cL^\vee \rtimes \mb W_F$,
\item $q \epsilon$ and $\rho$ yield the same character of 
$Z(\cH^\vee_\sc) Z(\cL_c^\vee)^\circ$. It is trivial on $Z(\cL_c^\vee)^\circ$ 
and determined by its restriction to $Z(\cL^\vee_\sc)$.
\item $(\phi |_{\mb W_F},v,q \epsilon)$ is a cuspidal L-parameter for $\cL$. 
\end{itemize}
\item The $\cH^\vee$-conjugacy class of $\cL^\vee \rtimes \mb W_F$ is
uniquely determined by $(\phi,\rho)$.
}
\end{prop}
\begin{proof}
(a) and (b) The torus $Z(M)^\circ$ commutes with $M$, so $Z_{\cH^\vee}(Z(M)^\circ)$
is a Levi subgroup of $\cH^\vee$ which contains the image of $M$ in $\cH^\vee$. 
As $Z(M)^\circ \subset Z^1_{\cH^\vee_{\sc}}(\phi |_{\mb W_F})$, 
$Z_{\cH^\vee \rtimes \mb W_F}(Z(M)^\circ)$ is a Levi L-subgroup of 
$\cH^\vee \rtimes \mb W_F$.

In view of Lemma \ref{lem:7.4} this implies that, upon conjugating $(\phi,\rho)$
with a suitable element of $\cH^\vee$, we may assume that the above 
construction yields a $\mb W_F$-stable standard Levi subgroup 
$\cL^\vee := Z_{\cH^\vee}(Z(M)^\circ)$ with 
\[
\phi (\mb W_F) \subset Z_{\cH^\vee \rtimes \mb W_F}(Z(M)^\circ) = 
\cL^\vee \rtimes \mb W_F .
\]
Its pre-image $\cL^\vee_c$ in $\cH^\vee_{\sc}$ satisfies 
\begin{equation}\label{eq:7.4}
G \cap \cL^\vee_c = Z^1_{\cH^\vee_{\sc}}(\phi |_{\mb W_F}) \cap  
Z_{\cH^\vee_\sc}(Z(M)^\circ) = M .
\end{equation}
Moreover $\cL^\vee$ contains $v$ (or rather its image in $\cH^\vee$, which we
also denote by $v$). Suppose that ${}^L L$ is another  
Levi L-subgroup of ${}^L \cH$ which contains $\phi (\mb W_F) \cup \{v\}$. Let $L_c$ 
be the inverse image of $L = {}^L L \cap \cH^\vee$ in $\cH^\vee_{\sc}$.
Since $(v,q \epsilon)$ is a cuspidal pair for $M$, $M^\circ$ is a Levi subgroup of
$G^\circ$ minimally containing $v$ (see \cite[Proposition 2.8]{Lus1} or 
Theorem \ref{thm:4.1}.a). Hence $L_c \cap G$ contains a $Z_G (v)$-conjugate of 
$M^\circ$, say $z M^\circ z^{-1}$. 
Then $Z(L_c)^\circ \subset z Z(M)^\circ z^{-1}$, so
\begin{equation}\label{eq:7.3}
L_c = Z_{\cH^\vee_{\sc}}(Z(L_c)^\circ) \supset 
Z_{\cH^\vee_{\sc}}(z Z(M)^\circ) z^{-1}) = z \cL^\vee_c z^{-1} .
\end{equation}
Thus $L$ contains a conjugate of $\cL^\vee$. Equivalently $\cL^\vee \rtimes
\mb W_F$ minimally contains $\phi (\mb W_F) \cup \{v\}$. Hence $(\phi |_{\mb W_F},v)$
is a discrete L-parameter for $\cL^\vee \rtimes \mb W_F$ and for some $F$-group 
$\cL$ with complex dual $\cL^\vee$. 

By \eqref{eq:7.1} $\rho \in \Irr \big(\pi_0 (Z^1_{\cH^\vee_{\sc}}(\phi)) \big)$ can
be regarded as a representation of $\pi_0 (Z_G (u_\phi))$, and by \eqref{eq:6.14}
it has the same $Z (\cH^\vee_{\sc})$-character, say $\zeta$, 
as $q \epsilon \in \Irr (\pi_0 (Z_M (v)))$.

Because $Z (M)^\circ$ becomes the trivial element in $\pi_0 (Z_M (v))$, $\zeta$ is 
trivial on $Z(M)^\circ \cap Z(\cH^\vee_\sc)$,  We note that
\begin{equation}\label{eq:6.23}
G \cap Z(\cL^\vee_c)^\circ = G \cap Z \big( Z_{\cH^\vee_\sc}(Z(M)^\circ) \big)^\circ
\subset G \cap Z(M)^\circ = Z(M)^\circ .
\end{equation}
But by construction $Z(M)^\circ \subset Z(\cL^\vee_c)^\circ$, so \eqref{eq:6.23} is 
actually an equality. As $Z(\cH_\sc^\vee) \subset G$, it follows that $\zeta$ is also
trivial on $Z(\cL_c^\vee)^\circ \cap Z(\cH^\vee_\sc)$. In particular it extends uniquely 
to a character (still denoted $\zeta$) of $Z(\cH^\vee_\sc) Z(\cL_c^\vee)^\circ$, which 
is trivial on $Z(\cL_c^\vee)^\circ$.

Furthermore $\cL^\vee_c$ is a connected Lie group, so $\cL^\vee_c = 
\cL^\vee_\sc Z(\cL^\vee_c)^\circ$. From this we see that $\zeta$ is determined by its
restriction to $\cL^\vee_\sc \cap Z(\cH^\vee_\sc) Z(\cL^\vee_c)^\circ$. 
By \cite[Lemma 1.1]{Art1} (see also Lemma \ref{lem:Kottwitz}.a), that group 
can be simplified to
\[
\cL^\vee_\sc \cap Z(\cH^\vee_\sc) Z(\cL^\vee_c)^\circ = 
\cL_\sc^\vee \cap Z(\cL_c^\vee) = Z(\cL_\sc^\vee) .
\]
Although $Z_{\cL^\vee_\sc}(\phi) = Z_{\cH^\vee_\sc}(\phi) \cap \cL^\vee_\sc$, the inclusion
$Z^1_{\cL^\vee_\sc}(\phi) \supset Z^1_{\cH^\vee_\sc}(\phi) \cap \cL^\vee_\sc$ can be strict,
as the definitions of the two $Z^1$'s are different. Nevertheless, always
\begin{equation}\label{eq:7.26}
Z^1_{\cL^\vee_\sc}(\phi) \subset (Z^1_{\cH^\vee_\sc}(\phi) \cap \cL^\vee_\sc ) 
Z(\cL^\vee_c)^\circ .
\end{equation}
Hence the relevant centralizers for $(\cL,\phi |_{\mb W_F},v)$ are
\[
Z^1_{\cL^\vee_{\sc}}(\phi |_{\mb W_F}) \cap Z_{\cL^\vee_\sc}(v) \subset
(G \cap Z_{\cL^\vee_{\sc}}(v)) Z(\cL^\vee_c)^\circ = Z_{M_\der}(v) Z(\cL^\vee_c)^\circ .
\]
Since $q \epsilon \in \Irr (A_M (v))$ is trivial on $Z(M)^\circ = Z(\cL^\vee_c)^\circ \cap M$,
it can be considered as a representation of
$\pi_0 \big( Z^1_{\cL^\vee_{\sc}}(\phi |_{\mb W_F}) \cap Z_{\cL^\vee_\sc}(v) \big)$
which is trivial on $Z(\cL^\vee_\sc) \cap Z(\cL^\vee_c)^\circ$. We conclude that 
$(\phi |_{\mb W_F},v,q \epsilon)$ is a cuspidal Langlands parameter for some inner form of $\cL$. 
The $Z(\cL_\sc^\vee)$-character of $q\epsilon$ is obtained from that of $\rho$ via extension
to $Z(\cH^\vee_\sc) Z(\cL_c^\vee)^\circ$ and then restriction. Comparing with Lemma
\ref{lem:Kottwitz}.b, and recalling that $(\phi,\rho)$ is relevant for $\cH$,
we see that $(\phi |_{\mb W_F},v,q \epsilon)$ is relevant for a Levi
subgroup $\cL$ of $\cH$.

By Definition \ref{defn:cuspdat} relevance of cuspidal data can be read off from their 
$Z (\cH^\vee_{\sc})^{\mb W_F}$-characters. The same comparison involving $\zeta$ 
says that $(\cL^\vee \rtimes \mb W_F , \phi |_{\mb W_F},v,q \epsilon)$ 
is also $\cH$-relevant.\\ 
(c) Suppose that ${}^L L$ is as above and that it minimally contains $\phi (\mb W_F) 
\cup \{v\}$. From \eqref{eq:7.3} or \cite[Proposition 8.6]{Bor} we see that 
${}^L L$ is $\cH^\vee$-conjugate to $\cL^\vee \rtimes \mb W_F$. Hence the 
L-Levi subgroup $\cL^\vee \rtimes \mb W_F$ is uniquely determined up to conjugation. 
\end{proof}

Before we continue with the cuspidal support map, we work out some consequences of
the above proof.

\begin{lem}\label{lem:7.14}
\enuma{
\item The exists a character $\zeta_\cH \in \Irr (Z(\cH^\vee_\sc))$ such that:
\begin{itemize}
\item $\zeta_\cH |_{Z(\cH^\vee_\sc)^{\mb W_F}}$ parametrizes the inner twist $\cH$ via
the Kottwitz isomorphism \eqref{eq:7.8},
\item $\zeta_\cH = 1$ on $Z(\cH^\vee_\sc) \cap Z(\cL^\vee_c)^\circ$, for every Levi
subgroup $\cL$ of $\cH$.
\end{itemize}
\item Let $\cL \subset \cH$ be a Levi subgroup and let $\phi : \mb W_F \times \SL_2 
(\C) \to {}^L \cL$ be a Langlands parameter for $\cL$. There exists a natural injection
$\cR_\phi^\cL \to \cR_\phi$.
\item In the setting of parts (a) and (b), extend $\zeta_\cH$ to a character of 
$Z(\cH^\vee_\sc) Z(\cL_c^\vee)^\circ$ which is trivial on $Z(\cL_c^\vee)^\circ$.
Let $\zeta_\cH^\cL$ be the restriction of the latter character to $Z(\cL^\vee_\sc)$. 
Let $p_{\zeta_\cH} \in \C[\cZ_\phi]$ and $p_{\zeta_\cH^\cL} \in \C[\cZ_\phi^\cL]$ 
be the central idempotents associated to these
characters. Then there is a canonical injection
\[
p_{\zeta_\cH^\cL} \C [\cS_\phi^\cL] \to p_{\zeta_\cH} \C[\cS_\phi].
\]
}
\end{lem}
\begin{proof}
(a) Let $\cL$ be a minimal Levi subgroup of $\cH$ and let $\phi \in \Phi (\cL)$ be
a discrete Langlands parameter which is trivial on $\SL_2 (\C)$. Then $\phi$ is 
$\cH$-relevant, so by Proposition \ref{prop:7.10} there exists an enhancement 
$\rho \in \Irr (\cS_\phi)$ such that the character $\zeta_\rho$ of $Z(\cH^\vee_\sc)$  
determined by $\rho$ parametrizes $\cH$ via the Kottwitz isomorphism. Then 
\[
G^\circ = Z_{\cH^\vee_\sc}(\phi)^\circ = \big( Z(\cL^\vee_c)^{\mb W_F} \big)^\circ
\]
is a torus, so every element of $\cN_G^+$ is cuspidal. It follows that 
\[
q \Psi_G (u_\phi = 1,\rho) = [G,v = 1,q \epsilon]_G .
\] 
Now Proposition \ref{prop:7.2}.b yields the desired condition for $\cL$.

Then the same condition holds for any Levi subgroup $\mathcal{M}$ of $\cH$ containing
$\cL$, for $Z(\mathcal{M}^\vee_c)^\circ \subset Z(\cL^\vee_c)^\circ$. Moreover 
$\zeta_\cH$ is invariant under conjugation, because it lives only on the centre. 
So the condition even holds for all $\cH^\vee_\sc$-conjugates of $\mathcal{M}^\vee_c$, 
which means that it is satisfied for all Levi subgroups of $\cH$.\\
(b) There is an obvious map
\begin{equation}\label{eq:7.20}
Z_{\cL^\vee}(\phi) \to \cR_\phi = 
Z_{\cH^\vee}(\phi) \big/ Z_{\cH^\vee}(\phi)^\circ \, Z(\cH^\vee)^{\mb W_F}.
\end{equation}
Its kernel equals
\begin{multline}\label{eq:7.21}
Z_{\cL^\vee}(\phi) \cap Z_{\cH^\vee}(\phi)^\circ Z(\cH^\vee)^{\mb W_F} =
Z_{\cH^\vee}(Z(\cL^\vee)^\circ) \cap Z_{\cH^\vee}(\phi)^\circ Z(\cH^\vee)^{\mb W_F} \\
= \big( Z_{\cH^\vee}(Z(\cL^\vee)^\circ) \cap Z_{\cH^\vee}(\phi)^\circ \big) 
Z(\cH^\vee)^{\mb W_F} 
= Z_{\cL^\vee}(\phi)^\circ Z(\cH^\vee)^{\mb W_F} .
\end{multline}
For the last equality we used that taking centralizers with tori preserves connectedness.
We note that $Z_{\cL^\vee}(\phi)^\circ \subset Z_{\cH^\vee}(\phi)^\circ$.
By \cite[Lemma 1.1]{Art1}
\[
Z(\cL^\vee)^{\mb W_F} = (Z(\cL^\vee)^{\mb W_F})^\circ \, Z(\cH^\vee)^{\mb W_F} ,
\]
which is contained in $Z_{\cH^\vee}(\phi)^\circ Z(\cH^\vee)^{\mb W_F}$. Hence
\eqref{eq:7.20} factors through 
\[
\cR_\phi^\cL = Z_{\cL^\vee}(\phi) \big/ Z_{\cL^\vee}(\phi)^\circ \, Z(\cL^\vee)^{\mb W_F}.
\]
By \eqref{eq:7.21} the kernel of the just constructed map $\cR_\phi^\cL \to \cR_\phi$
is the image of \\ $Z_{\cL^\vee}(\phi)^\circ Z(\cH^\vee)^{\mb W_F}$ in $\cR_\phi^\cL$,
which is only the neutral element.\\
(c) Lemma \ref{lem:Kottwitz}.a (for the trivial $\mb W_F$-action) shows that 
$\zeta_\cH^\cL$ is well-defined. By \eqref{eq:6.24} every system of 
representatives for $\cR_\phi \cong \cS_\phi / \cZ_\phi$ in $\cS_\phi$ provides a
basis of $p_{\zeta_\cH} \C [\cS_\phi]$. Similarly
\begin{equation}\label{eq:7.25}
p_{\zeta_\cH^\cL} \C [\cS_\phi^\cL] \cong \C [\cR_\phi^\cL]
\text{ as vector spaces.}
\end{equation}
We have to find an appropriate variation on $\C [\cR_\phi^\cL] \to \C[\cR_\phi]$.
Recall from \eqref{eq:7.26} that 
\begin{equation}\label{eq:7.28}
Z^1_{\cL^\vee_\sc}(\phi) = (Z^1_{\cH^\vee_\sc}(\phi) \cap \cL^\vee_\sc) 
(Z(\cL^\vee_c )^\circ \cap \cL^\vee_\sc ) .
\end{equation}
This gives a group homomorphism
\begin{equation}\label{eq:7.27}
\lambda : Z^1_{\cL^\vee_\sc}(\phi) \to Z^1_{\cH^\vee_\sc}(\phi) \big/ 
\big( Z^1_{\cH^\vee_\sc}(\phi) \cap Z(\cL^\vee_c )^\circ \cap \cL^\vee_\sc \big) 
\end{equation}
which lifts $\cR_\phi^\cL \to \cR_\phi$. Consider the diagram
\[
\xymatrix{
p_{\zeta_\cH^\cL} \C [\cS_\phi^\cL] \ar@{^{(}->}[r] \ar@{.>}[d] & 
\C[\cS_\phi^\cL] \ar[d]^{\lambda} \\
p_{\zeta_\cH} \C[\cS_\phi] \ar@{^{(}->}[r] & \C [Z^1_{\cH^\vee_\sc}(\phi) \big/ 
\big( Z^1_{\cH^\vee_\sc}(\phi) \cap Z(\cL^\vee_c )^\circ \cap \cL^\vee_\sc \big)] .
}
\]
The lower arrow exists because $\zeta_\cH = 1$ on $Z(\cH^\vee_\sc) \cap 
Z(\cL^\vee_c)^\circ$. The image $\lambda (p_{\zeta_\cH^\cL} \C [\cS_\phi^\cL])$
is contained in $p_{\zeta_\cH} \C[\cS_\phi]$ by the relation between $\zeta_\cH$
and $\zeta_\cH^\cL$, which gives the left vertical arrow. Since \eqref{eq:7.27} is a
lift of $\cR_\phi^\cL \to \cR_\phi$ and by \eqref{eq:7.25}, this arrow is injective.
\end{proof}

It turns out that the cuspidal datum constructed in Proposition \ref{prop:7.2}.a need 
not have the same infinitesimal character as $\phi$ (in the sense of \cite{Hai,Vog}).
Since this would be desirable for a cuspidal support map, we now work out some
constructions which compensate for this. See \eqref{eq:infChar} for their effect.

Recall from \eqref{eq:6.1} that the unipotent element $v$ in $q \Psi_G (u_\phi,\rho)$
also appears as $\Psi_{G^\circ}(u_\phi,\rho^\circ) = (M^\circ,v,\epsilon)$,
where $\rho^\circ$ is an irreducible $A_{G^\circ}(u_\phi)$-constituent of $\rho$.
The construction of $\Psi_{G^\circ}$, which already started in \eqref{eq:2.2},
entails that there exists a parabolic subgroup $P$ of $G^\circ$ such that
\begin{itemize}
\item $M^\circ$ is a Levi factor of $P$,
\item $u_\phi = v u_P$ with $u_P$ in the unipotent radical $U_P$ of $P$.
\end{itemize}
Upon conjugating $\phi$ with a suitable element of $Z_{G^\circ}(u_\phi)$, we may assume 
that $M^\circ$ contains $\phi \big( 1, \matje{z}{0}{0}{z^{-1}} \big)$ for all 
$z \in \C^\times$. (Alternatively, one could conjugate $M^\circ$ inside $G^\circ$.) 
Since the $G^\circ$-conjugacy class of $(M^\circ,v)$ matters most, this conjugation 
is harmless.

\begin{lem}\label{lem:7.11}
Suppose that $\phi \big( 1, \matje{z}{0}{0}{z^{-1}} \big) \in M^\circ$ for all 
$z \in \C^\times$. Then \\ $\phi \big( 1, \matje{z}{0}{0}{z^{-1}} \big) v 
\phi \big( 1, \matje{z^{-1}}{0}{0}{z} \big) = v^{z^2}$ for all $z \in \C^\times$. 
\end{lem}
\begin{proof}
The condition on $M^\circ$ entails that 
\[
\mathrm{Ad} \circ \phi \big( 1, \matje{z}{0}{0}{z^{-1}} \big) (v)
\in M^\circ \quad \text{and} \quad \mathrm{Ad} \circ \phi 
\big( 1, \matje{z}{0}{0}{z^{-1}} \big) (u_P) \in U_P .
\]
Hence $\mathrm{Ad} \circ \phi \big( 1, \matje{z}{0}{0}{z^{-1}} \big) (v)$ is the image of 
\[
\mathrm{Ad} \circ \phi \big( 1, \matje{z}{0}{0}{z^{-1}} \big) (v u_P) = 
\mathrm{Ad} \circ \phi \big( 1, \matje{z}{0}{0}{z^{-1}} \big) (u_\phi)
\] 
under $P / U_P \isom M^\circ$. Since $\phi |_{\SL_2 (\C)} : \SL_2 (\C) \to G^\circ$ is an
algebraic group homomorphism, 
\begin{multline*}
\mathrm{Ad} \circ \phi \big( 1, \matje{z}{0}{0}{z^{-1}} \big) (u_\phi) = \phi \big( 1, 
\matje{z}{0}{0}{z^{-1}} \matje{1}{1}{0}{1} \matje{z^{-1}}{0}{0}{z} \big) \\
= \phi \big( 1, \matje{1}{z^2}{0}{1} \big) = u_\phi^{z^2} = (v u_P)^{z^2} .
\end{multline*}
By the unipotency of $v u_P$ there are unique $X \in \text{Lie} (M^\circ), Y \in 
\text{Lie} (U_P)$ such that $v u_P = \exp_P (X + Y)$. As Lie $(U_P)$ is an ideal of 
Lie $(P)$, $\exp_P (X+Y) \in \exp_{M^\circ} (X) U_P$, and hence $X = \log_{M^\circ} (v)$. 
Similarly we compute
\[
(v u_P)^{z^2} = \exp_P (\log_P (v u_P)^{z^2}) = \exp_P (z^2 (X + Y)) \in 
\exp_{M^\circ} (z^2 X) U_P .
\]
Consequently the image of $(v u_P)^{z^2}$ under $P / U_P \isom M^\circ$ is 
$\exp_{M^\circ} (z^2 X) = v^{z^2}$.
\end{proof}

In the setting of Lemma \ref{lem:7.11}, \cite[\S 2.4]{KaLu} shows that there exists
an algebraic group homomorphism $\gamma_v : \SL_2 (\C) \to M^\circ$ such that
\begin{itemize}
\item $\gamma_v \matje{1}{1}{0}{1} = v$,
\item $\gamma_v (\SL_2 (\C))$ commutes with $\phi \big( 1, \matje{z}{0}{0}{z^{-1}} \big)
\gamma_v \matje{z^{-1}}{0}{0}{z}$ for all $z \in \C^\times$.
\end{itemize}
Moreover $\gamma_v$ is unique up to conjugation by $Z_{M^\circ}\big( v, 
\phi \big( 1, \matje{z}{0}{0}{z^{-1}} \big) \big)$, for any $z \in \C^\times$ of infinite order. We will
say that a homomorphism $\gamma_v$ satisfying these conditions is adapted to $\phi$.

\begin{lem}\label{lem:7.12}
Let $(\phi,\rho)$ be an enhanced L-parameter for $\cH$ and write
$q \Psi_G (u_\phi,\rho) = [M,v,q\epsilon]_G$, using \eqref{eq:7.2}. Up to $G$-conjugacy
there exists a unique $\gamma_v : \SL_2 (\C) \to M^\circ$ adapted to $\phi$. 
Moreover the cocharacter
\[
\chi_{\phi,v} \colon z \mapsto \phi \big( 1, \matje{z}{0}{0}{z^{-1}} \big)
\gamma_v \matje{z^{-1}}{0}{0}{z}
\]
has image in $Z(M)^\circ$.
\end{lem}
\begin{proof}
Everything except the last claim was already checked above. Since $(v,q \epsilon)$ is
cuspidal, Theorem \ref{thm:4.1}.a says that $v$ is distinguished. This means that it
does not lie in any proper Levi subgroup of $M^\circ$. In other words,
every torus of $M^\circ$ which centralizes $v$ is contained in $Z(M^\circ)^\circ$. 
Finally we note that, as $M$ is a quasi-Levi subgroup, $Z(M)^\circ = Z(M^\circ)^\circ$.
\end{proof}

Notice that the image of the cocharacter $\chi_{\phi,v} : \C^\times \to Z(M)^\circ$ 
commutes not only with $\gamma_v (\SL_2 (\C))$ but also with $\phi (\mb W_F)$, because
$M^\circ \subset G^\circ \subset Z_{\cH_\sc^\vee}(\phi |_{\mb W_F})$.

\begin{defn}\label{def:7.8}
In the setting of Lemma \ref{lem:7.12} we put 
\[
{}^L \Psi (\phi,\rho) = (Z_{\cH^\vee \rtimes \mb W_F}(Z(M)^\circ),
\phi |_{\mb W_F},v,q \epsilon),
\]
a $\cH$-relevant cuspidal datum for ${}^L \cH$.

Let $\norm{\cdot} : \mb W_F \to \R_{>0}$ be the group homomorphism with $\norm{w} = q$ if
$w(f) = f^q$ for all $f$ in the algebraic closure of the residue field of $F$.

We define a L-parameter $\varphi_v \colon \mb W_F \times \SL_2 (\C) \to 
Z_{\cH^\vee \rtimes \mb W_F}(Z (M)^\circ)$ by
\[
\varphi_v (w,x) = \phi (w) \chi_{\phi,v}(\norm{w}^{1/2}) \gamma_v (x) .
\]
The cuspidal support of $(\phi,\rho)$ is 
\[
\mathbf{Sc} (\phi,\rho) = (Z_{\cH^\vee \rtimes \mb W_F}(Z(M)^\circ),\varphi_v,q \epsilon),
\]
another $\cH$-relevant cuspidal datum for ${}^L \cH$.
\end{defn}

By parts (a) and (c) of Proposition \ref{prop:7.2} the map ${}^L \Psi$ is canonical in the 
sense that its image is unique up to conjugation. By Lemma \ref{lem:7.12} {\bf Sc} is
also canonical. Furthermore, the images of ${}^L \Psi$ and {\bf Sc} are $\cH$-relevant by 
Proposition \ref{prop:7.2}.b. In view of Proposition \ref{prop:7.2}.c, we can always
represent ${}^L \Psi (\phi,\rho)$ and $\mathbf{Sc}(\phi,\rho)$ by a cuspidal L-parameter
for a Levi subgroup of $\cH$.

An advantage of $\varphi_v$ over $(\phi|_{\mb W_F},v)$ is that
\begin{equation}\label{eq:infChar}
\begin{split}
\varphi_v \big( w, \matje{\norm{w}^{1/2}}{0}{0}{\norm{w}^{-1/2}} \big) = 
\phi (w) \chi_{\phi,v}(\norm{w}^{1/2}) \gamma_v \matje{\norm{w}^{1/2}}{0}{0}{\norm{w}^{-1/2}} \\
= \phi \big( w, \matje{\norm{w}^{1/2}}{0}{0}{\norm{w}^{-1/2}} \big) .
\end{split}
\end{equation}
In the terminology from \cite{Hai,Vog}, this says that the cuspidal support map for 
enhanced L-parameters preserves infinitesimal characters. It is interesting to
compare the fibres of {\bf Sc} with the variety constructed in \cite[Corollary 4.6]{Vog}.
Vogan considers the set of all L-parameters for ${}^L \cH$ with a fixed infinitesimal
character (up to conjugation). In \cite[Proposition 4.5]{Vog} he proves that this set has 
the structure of a complex affine variety, on which $\cH^\vee$ acts naturally, with only 
finitely many orbits. The same picture can be obtained from a fibre of {\bf Sc}, upon
neglecting all enhancements of L-parameters.

More or less by definition Bernstein's cuspidal support map for $\Irr (\cH)$ preserves
infinitesimal characters. That property is slightly less strong for our {\bf Sc} on the 
Galois side, for enhanced L-parameters with different cuspidal support can have the same 
infinitesimal character. The map 
\[
{}^L \Psi : \Phi_e (\cH) \to \{\text{cuspidal data for } \cH \} / \cH^\vee \text{-conjugacy} 
\]
is an analogue of a modified version, say $\widetilde{\mb{Sc}}$, of Bernstein's
cuspidal support map for $\Irr (\cH)$. Neither ${}^L \Psi$ nor 
$\widetilde{\mathbf{Sc}}$ preserve infinitesimal characters, but they
have other advantages that the cuspidal support maps lack. For ${}^L \Psi$
this will become clear in the Section \ref{sec:extquot}, while the importance 
of $\widetilde{\mathbf{Sc}}$ stems from its role in the ABPS conjecture. 

To enable a comparison, we recall its definition from \cite[\S 2.5]{ABPS7}.
Let $\mc P = \mc M \mc{U}_{\mc P}$ be a parabolic subgroup of $\cH$ and let 
$\omega \in \Irr (\mc M)$ be square-integrable modulo centre.
Suppose that $\pi_t \in \Irr (\cH)$ is tempered and that it is a direct
summand of the normalized parabolic induction $I_{\mc P}^\cH (\omega)$.
Let $(\cL,\sigma)$ be the cuspidal support of $\omega$. Then $\sigma$ can
be written uniquely as $\sigma = \sigma_u \otimes \nu$, with $\nu : \cL
\to \R_{>0}$ an unramified character and $\sigma_u \in \Irr_\cusp (\cL)$
unitary (and hence tempered). One defines
\[
\widetilde{\mb{Sc}} (\pi_t) = (\cL,\sigma_u) / \cH \text{-conjugacy} .
\]
Notice that $\widetilde{\mb{Sc}}$ preserves temperedness of representations,
in contrast with $\mb{Sc}$.

More generally, by \cite[Theorem 2.15]{SolPCH} every $\pi \in \Irr (\cH)$
can be written (in an essentially unique way) as a Langlands quotient of
$I_{\mc P}^\cH (\omega \otimes \chi)$, where $\mc P = \mc M \mc U_{\mc P}$
and $\omega$ are as above and $\chi \in X_\nr (\mc M)$. Then 
$\chi$ restricts to an unramified character of $\cL$ and the cuspidal 
support of  $\omega \otimes \chi$ is $(\cL,\sigma \otimes \chi)$.
In this case one defines
\[
\widetilde{\mb{Sc}} (\pi) = (\cL,\sigma_u \otimes \chi) / \cH \text{-conjugacy}. 
\]
We note that the only difference with $\mb{Sc}(\pi)$ is $\nu |_\cL$, an 
unramified character $\cL \to \R_{>0}$ which represents the absolute value 
of the infinitesimal central character of $\sigma$.

It has been believed for a long time that the (enhanced) L-parameters of $\pi \in
\Irr (\cH)$ and $\mathbf{Sc}(\pi)$ are always related, but it was not clear how.
With our new notions we can make this precise. Let $\mf{Lev}(\cH)$ be a set of
representatives for the conjugacy classes of Levi subgroups of $\cH$, and
recall \eqref{eq:C1} and \eqref{eq:C2}. 

\begin{conj}\label{conj:7.13}
Assume that a local Langlands correspondence exists for $\cH$ and for supercuspidal
representations of its Levi subgroups. The following diagram should commute:
\[
\xymatrix{
\Irr (\cH) \ar@{<-}[rr]^{\LLC} \ar[d]^{\mathbf{Sc}} & & \Phi_e (\cH) \ar[d]^{\mathbf{Sc}} \\
\bigsqcup_{\cL \in \mf{Lev}(\cH)} \Irr_\cusp (\cL) / W(\cH,\cL) \ar@{<-}[rr]_{\LLC} & & 
\bigsqcup_{\cL \in \mf{Lev}(\cH)} \Phi_\cusp (\cL) / W(\cH,\cL) .
}
\]
\end{conj}

Conjecture \ref{conj:7.13} is known to hold for many of the groups for which a LLC
has been established.
\begin{itemize}
\item For $\GL_n (F)$ it is a consequence of the Bernstein--Zelevinsky
classification of $\Irr (\GL_n (F))$ \cite{Zel} and the way it is used in the local
Langlands correspondence for $\GL_n (F)$, see \cite[\S 2]{Hen}.
\item Irreducible representations of inner forms $\GL_m (D)$ of $\GL_n (F)$ can 
also be classified via a Zelevinsky-like scheme, see \cite{Tad}. This is used in
the LLC in the same way as for $\GL_n (F)$ \cite[\S 2]{ABPS3}, so the conjecture
also holds for these groups.
\item The local Langlands correspondence for an inner form $\SL_m (D)$ of $\SL_n (F)$
is derived directly from that for $\GL_m (D)$: on the Galois side one lifts 
L-parameters $\mb W_F \times \SL_2 (\C) \to \PGL_n (\C)$ to $\GL_n (\C)$,
whereas on the $p$-adic side one restricts irreducible representations of
$\GL_m (D)$ to $\SL_m (D)$ to construct L-packets.
These two operations do not really change the infinitesimal central characters
of L-parameters or smooth representations, only on $Z(\GL_n (\C)) \cong \C^\times$ 
or $Z(\GL_m (D)) \cong F^\times$, respectively. Therefore Conjecture 
\ref{conj:7.13} for $\GL_m (D)$ implies it for $\SL_m (D)$.
\item For the split classical groups $\Sp_{2n}(F)$ and $\SO_{m}(F)$ when $F$ is a 
$p$-adic field. The support cuspidal map specializes to the map defined in 
\cite[Th\'eor\`eme~4.27]{Mou}, and the commutativity of the diagram follows from 
\cite[Th\'eor\`eme~5.9]{Mou}. 
\item For principal series representations of split groups see 
\cite[Theorem 15.1]{ABPS5}.
\item For unipotent representations of simple $p$-adic groups $\cH$ of adjoint 
type we refer to \cite{Lus3}. Although it is not so easy to see, the essence is 
that Lusztig uses the element 
$f = \phi \Big( \Fr, \matje{\norm{\Fr}^{1/2}}{0}{0}{\norm{\Fr}^{-1/2}} \Big)$
to parametrize the central character of a representation of a suitable
affine Hecke algebra \cite[\S 9.3]{Lus3}. By construction this also 
parametrizes the infinitesimal central character of the associated 
representation of $\cH$.
\end{itemize}

To support Conjecture \ref{conj:7.13}, we check that the cuspidal support map
is compatible with the Langlands classification for L-parameters. The latter
is a version of the Langlands classification for $\Irr (\cH)$ on the Galois
side of the LLC, it stems from \cite{SiZi}.

We will describe first a Galois side  analogue for unramified characters. 
Let $\mb I_F \subset \mb W_F$ be as above the inertia subgroup and let
$\Fr \in \mb W_F$ be a Frobenius element. Recall from \cite[\S 3.3.1]{Hai}
that there is a canonical isomorphism of complex tori
\begin{equation}\label{eq:8.1}
X_\nr (\cL) \cong \big( Z (\cL^\vee)^{\mb I_F} \big)^\circ_\Fr =
Z( \cL^\vee \rtimes \mb I_F )^\circ_{\mb W_F / \mb I_F} =
Z( \cL^\vee \rtimes \mb I_F )^\circ_{\cL^\vee \rtimes \mb W_F} .
\end{equation}
The group $X_\nr (\cL)$ acts on $\Irr (\cL)$ by tensoring. This corresponds
to an action of $\big( Z (\cL^\vee)^{\mb I_F} \big)^\circ_\Fr$ on $\Phi (\cL)$
and on $\Phi_e (\cL)$. Namely, let $\phi : \mb W_F \times \SL_2 (\C) \to
{}^L \cL$ be a relevant L-parameter and let $z \in Z (\cL^\vee)^{\mb I_F}$.
We define $z \phi \in \Phi (\cL)$ by
\begin{equation}\label{eq:8.2}
(z \phi) \big|_{\mb I_F \times \SL_2 (\C)} = \phi \big|_{\mb I_F \times \SL_2 (\C)} 
\quad \text{ and } \quad (z \phi) (\Fr) = z \, \phi (\Fr). 
\end{equation}
Notice that $z \phi \in \Phi (\cL)$ because $z \in Z( \cL^\vee \rtimes \mb I_F)$. 
Suppose that $z' \in Z (\cL^\vee)^{\mb I_F}$ represents the same element of 
$\big( Z (\cL^\vee)^{\mb I_F} \big)_\Fr$. Then $z^{-1} z' = x^{-1} \Fr (x)$ for
some $x \in Z (\cL^\vee)^{\mb I_F}$, and
\[
z' \phi = x^{-1} \Fr (x) z \phi = x^{-1} z \phi x . 
\]
Hence $z' \phi = z \phi$ in $\Phi (\cL)$ and we obtain an action of 
$\big( Z (\cL^\vee)^{\mb I_F} \big)_\Fr$ on $\Phi (\cL)$. As $z$ commutes with
$\cL^\vee$, $\cS_{z \phi} = \cS_\phi$. This enables us to lift the action to
$\Phi_e (\cL)$ by
\begin{equation}\label{eq:8.3}
z (\phi,\rho) = (z \phi,\rho) . 
\end{equation}
To allow $\cH^\vee$ to act on the above objects, we also have to define them 
for Levi L-subgroups ${}^L L$ of ${}^L \cH$. Generalizing \eqref{eq:8.1}, we put
\begin{equation}\label{eq:8.11}
X_\nr \big( {}^L L \big) = 
Z \big( \cH^\vee \rtimes \mb I_F \cap {}^L L \big)^\circ_\Fr . 
\end{equation}
This group plays the role of unramified characters for ${}^L L$, we will sometimes
refer to it as the unramified twists of ${}^L L$.
By the formula \eqref{eq:8.2}, $X_\nr ({}^L L)$ acts on Langlands parameters 
with image in ${}^L L$. As in \eqref{eq:8.3}, that extends to an action on
enhanced L-parameters for ${}^L L$. 

The following notion replaces the data in the Langlands classification for $\cH$.

\begin{defn}\label{def:8.1}
Fix a pinning of $\cH$ and a $\mb W_F$-stable pinning of $\cH^\vee$. 
A standard triple for $\cH$ consists of:
\begin{itemize}
\item a standard Levi subgroup $\cL$ of $\cH$;
\item a bounded L-parameter $\phi_t \in \Phi_{\bdd}(\cL)$;
\item an unramified twist $z \in X_\nr ({}^L \cL)$, which is strictly positive
with respect to the standard parabolic subgroup $\mc P$ with Levi factor $\cL$.
\end{itemize}
The last condition means that $\alpha^\vee (z) > 1$ for every root $\alpha$
of $(\mc U_{\mc P}, Z(\cL)^\circ)$, where $\mc U_{\mc P}$ denotes the unipotent 
radical of $\mc P$. 

An enhancement of a standard triple $(\cL,\phi_t,z)$ is an $\cL$-relevant
irreducible representation $\rho_t$ of $\cS^\cL_{z \phi_t}$. Let $\zeta_\cH$ and 
$\zeta_\cH^\cL$ be as in Lemma \ref{lem:7.14}. We say that $(\cL,\phi_t,z,\rho_t)$
is an enhanced standard triple for $(\cH,\zeta_\cH)$ if $\rho_t |_{Z(\cL^\vee_\sc)}
= \zeta_\cH^\cL$.
\end{defn}

\begin{thm}\label{thm:8.2}
\enuma{
\item There exists a canonical bijection from the set of standard triples of $\cH$ to 
$\Phi (\cH)$. It sends $(\cL,\phi_t,z)$ to $z \phi_t$ (up to $\cH^\vee$-conjugacy).
\item The natural map 
\[
p_{\zeta_\cH^\cL} \C[\cS^\cL_{\phi_t}] = p_{\zeta_\cH^\cL} \C[\cS^{\cL}_{z \phi_t}] 
\to p_{\zeta_\cH} \C[\cS^\cH_{z \phi_t}]
\]
from Lemma \ref{lem:7.14} is an isomorphism. Hence part (a) can be enhanced to a 
canonical bijection 
\[
\begin{array}{ccc}
\{\text{enhanced standard triples for } (\cH,\zeta_\cH) \} & \longleftrightarrow & 
\Phi_{e,\zeta_\cH} (\cH) \\
(\cL,\phi_t,z,\rho_t) & \mapsto & (z \phi_t,\rho_t) .
\end{array}
\]
} 
\end{thm}
\begin{proof}
(a) See \cite[Theorem 4.6]{SiZi}. The differences are only notational: we 
replaced a standard parabolic subgroup $\mc P$ of $\cH$ by its standard Levi factor 
$\cL$ and we used $z \in X_\nr ({}^L \cL)$ instead of the presentation of $X_\nr (\cL)$
by elements of $\mf a^*_\cL$. The regularity of $\nu \in \mf a^*_\cL$ in \cite{SiZi}
means that it lies in the open Weyl chamber of $\mf a^*_\cL$ determined by
$\mc P$. This translates to $z$ being strictly positive with respect to $\mc P$.

(b) Since $z \in Z(\cL^\vee)$, $\phi_t$ and $z \phi_t$ have the same $\cS$-groups 
for $\cL$. In \cite[Proposition 7.1]{SiZi} it is shown that the natural map 
$\cR_{z \phi_t}^\cL \to \cR_{z \phi_t}$ is a bijection. In Lemma \ref{lem:7.14} 
we constructed a natural injection
\[ 
p_{\zeta_\cH^\cL} \C[\cS^{\cL}_{z \phi_t}] \to p_{\zeta_\cH} \C[\cS^\cH_{z \phi_t}] .
\]
The dimensions of these spaces are, respectively, $|\cR_\phi^\cL|$ and $|\cR_\phi|$.
These are equal by \cite[Proposition 7.1]{SiZi}, so the above map is an algebra
isomorphism.
\end{proof}

The maps ${}^L \Psi$ and {\bf Sc} from Definition \ref{def:7.8} are compatible with
Theorem \ref{thm:8.2} in the sense that they factor through this Langlands
classification.

\begin{lem}\label{lem:8.3}
Let $(\phi,\rho) \in \Phi_{e,\zeta_\cH} (\cH)$ and let $(\cL,\phi_t,z,\rho_t)$ 
be the enhanced standard triple associated to it by Theorem \ref{thm:8.2}. Then
\[
\begin{array}{ccccc}
{}^L \Psi^\cH (\phi,\rho) & = & {}^L \Psi^\cL (z \phi_t,\rho_t) & = & 
z \cdot{}^L \Psi^\cL (\phi_t,\rho_t) , \\
\mb{Sc}^\cH (\phi,\rho) & = & \mb{Sc}^\cL (z \phi_t,\rho_t) & = & 
z \cdot \mb{Sc}^\cL (\phi_t,\rho_t) .
\end{array}
\]
\end{lem}
\begin{proof}
Because all the maps are well-defined on conjugacy classes of enhanced L-parameters,
we may assume that $\phi = z \phi_t$ and $\rho = \rho_t$. By definition 
${}^L \Psi^\cH (z \phi_t,\rho_t)$ is given in terms of $q \Psi_G (u_\phi,\rho) =
[M,v,q \epsilon]_G$, as $(Z_{{}^L \cH}(Z(M)^\circ),\phi |_{\mb W_F},v,q \epsilon)$.
Consider
\[
G_1 := Z_G (Z(\cL^\vee_c)^\circ ) = 
Z^1_{\cH^\vee_\sc}(\phi |_{\mb W_F}) \cap \cL^\vee_c
\]
Since $\cL^\vee_c = Z_{\cH^\vee_{\sc}}(Z(\cL^\vee_c)^\circ )$ is a Levi subgroup
of $\cH^\vee_{\sc}$, $G_1$ is a quasi-Levi subgroup of 
$G = Z^1_{\cH^\vee_{\sc}}(\phi |_{\mb W_F})$. Furthermore
\[
G_1^\circ = \big( Z_{\cH^\vee_\sc}(\phi |_{\mb W_F}) \cap \cL^\vee_c
\big)^\circ = Z_{\cL_c^\vee}(\phi |_{\mb W_F})^\circ .
\]
By Proposition \ref{prop:6.4} 
\[
q \Psi_{G_1}(u_\phi,\rho) = q \Psi_G (u_\phi,\rho).
\]
Write $G_2 = Z^1_{\cL^\vee_\sc}(\phi |_{\mb W_F}) Z(\cL_c^\vee)^\circ$ and abbreviate
$G_3 = Z^1_{\cL^\vee_\sc}(\phi |_{\mb W_F})$. From
\[
Z(\cH_\sc^\vee) \subset Z(\cL_c^\vee) = Z(\cL_\sc^\vee) Z(\cL_c^\vee)^\circ 
\]
and the description of $G_1^\circ$ we see that $G_1$ is a finite index subgroup of $G_2$.
Since the quasi-cuspidal support of $(u_\phi,\rho)$ (for $G_1$) is derived from the cuspidal 
support (for $G_1^\circ = G_2^\circ$), $Z(M)^\circ$ is the same for $G_2$ and $G_1$. Hence
\begin{equation}\label{eq:7.29}
q \Psi_{G_2}(u_\phi,\rho) = [M_2,v,q \epsilon_2]_{G_2}, \qquad M_2 = Z_{G_2}(Z(M)^\circ),
\end{equation}
where $q \epsilon_2$ is an extension of $q\epsilon \in \Irr (A_M (v))$ to $A_{M_2}(v)$.
By \eqref{eq:7.28} and Proposition \ref{prop:7.2}.b $q \epsilon_2$ is obtained from 
$q \epsilon$ by setting it equal to 1 on a suitable central subgroup. 
When we replace $G_2$ by $G_3$ we only omit a part of its connected centre, which does not
make a real difference for quasi-cuspidal supports. Concretely, \eqref{eq:7.29} entails
\[
q \Psi_{G_3}(u_\phi,\rho) = [M_3,v,q \epsilon_3]_{G_3}, \qquad M_3 = Z_{G_3}(Z(M)^\circ),
\]
where the inflation of $q \epsilon_3$ to a function on $Z_{M_3}(v)$ agrees with the
inflation of $q \epsilon_2$ to $Z_{M_3}(v)$. As explained in the proof of Proposition 
\ref{prop:7.2} after \eqref{eq:7.26}, this means that as cuspidal data
\[
{}^L \Psi^\cL (\phi,\rho) =   (Z_{{}^L \cL}(Z(M)^\circ),\phi |_{\mb W_F},v,q \epsilon_3) =
(Z_{{}^L \cL}(Z(M)^\circ),\phi |_{\mb W_F},v,q \epsilon) . 
\]
Now ${}^L \cL$ is a Levi L-subgroup of ${}^L \cH$ containing $\phi (\mb W_F) \cup \{v\}$. 
In the proof of Proposition \ref{prop:7.2}.a we checked that $\cL^\vee \rtimes \mb W_F \supset 
Z_{{}^L \cH}(Z(M)^\circ)$. Hence 
\[
Z_{{}^L \cH}(Z(M)^\circ) = Z_{{}^L \cL}(Z(M)^\circ) \quad \text{and} \quad 
{}^L \Psi^\cL (\phi,\rho) = {}^L \Psi^\cH (\phi,\rho).
\]
As $Z(\cL^\vee_c) \subset Z(M)$, the element 
\[
z \in X_\nr ({}^L \cL) = (Z(\cL^\vee)^{\mb I_F})^\circ_{\Fr}
\]
also lies in $X_\nr (Z_{{}^L \cH}(Z(M)^\circ)$. Since $z$ commutes with $\cL^\vee_c$, 
$G_t = Z^1_{\cH^\vee_\sc}(z \phi_t |_{\mb W_F}) \cap \cL^\vee_c$ equals 
$Z^1_{\cH^\vee_\sc}(\phi_t |_{\mb W_F}) \cap \cL^\vee_c$.
Now Definition \ref{def:7.8} shows that
\begin{multline*}
{}^L \Psi^\cL (z \phi_t,\rho_t) = (Z_{{}^L \cH}(Z(M)^\circ),z \phi_t |_{\mb W_F},v,q \epsilon) \\ 
= z \cdot (Z_{{}^L \cH}(Z(M)^\circ),\phi_t |_{\mb W_F},v,q \epsilon) =
z \cdot{}^L \Psi^\cL (\phi_t,\rho_t) .
\end{multline*}
The construction of $\chi_{\phi,v}$ in Lemma \ref{lem:7.12} depends only on 
$q \Psi_{G_t} (u_\phi,\rho)$, so $\mb{Sc}^\cH (\phi,\rho) = \mb{Sc}^\cL (z \phi_t,\rho_t) =
z \cdot \mb{Sc}^\cL (\phi_t,\rho_t)$ as well.
\end{proof}

\section{Inertial equivalence classes of L-parameters}

In important ingredient in Bernstein's theory of representations of $p$-adic 
groups are inertial equivalence classes. Let $\cL \subset \cH$ be a Levi
subgroup and let $X_\nr (\cL)$ be the group of unramified characters $\cL \to 
\C^\times$. Two cuspidal pairs $(\cL_1,\sigma_1)$ and $(\cL_2,\sigma_2)$ are
said to be inertially equivalent if there exist an unramified character 
$\chi_1$ of $\cL_1$ and an element $h \in \cH$ such that
\[
h \cL_2 h^{-1} = \cL_1 \quad \text{and} \quad 
h \cdot \sigma_2 = \sigma_1 \otimes \chi_1 .
\]
We denote a typical inertial equivalence of cuspidal pairs by $\fs = 
[\cL,\sigma]_\cH$, we let $\mf B (\cH)$ be the set of such classes. With
every $\fs \in \mf B (\cH)$ one can associate a set of irreducible smooth
$\cH$-representations:
\[
\Irr (\cH)^\fs = \{ \pi \in \Irr (\cH) : \text{ the cuspidal support of }
\pi \text{ lies in } \fs \} .
\]
A (weak) version of the Bernstein decomposition says that 
\begin{equation}\label{eq:8.5}
\Irr (\cH) = \bigsqcup\nolimits_{\fs \in \mf B (\cH)} \Irr (\cH)^\fs . 
\end{equation}
We will establish a similar decomposition for enhanced Langlands
parameters. 

Our notion of inertial equivalence generalizes \cite[Definition 5.33]{Hai} from 
homomorphisms $\mb W_F \to {}^L \cH$ to enhanced L-parameters.
\begin{defn}\label{def:8.4}
Let $({}^L L,\phi_v,q \epsilon)$ and $({}^L L',\phi'_v,q \epsilon')$ be two cuspidal 
data for ${}^L \cH$. They are inertially equivalent if there exist 
$z \in X_\nr ({}^L L)$ and $h \in \cH^\vee$ such that 
\[
h \, {}^L L' h^{-1} = {}^L L \quad \text{and} \quad 
(z \phi_v,q \epsilon) = (h \phi'_v h^{-1}, h \cdot q \epsilon').
\]
The class of $({}^L L,\phi_v,q \epsilon)$ modulo $X_\nr ({}^L L)$ is denoted
$[{}^L L,\phi_v ,q \epsilon ]_{{}^L L}$, and its inertial equivalence class is denoted
$[{}^L L,\phi_v ,q \epsilon ]_{{}^L \cH}$.
We say that $[{}^L L,\phi_v ,q \epsilon]_{{}^L \cH}$ is $\cH$-relevant if any of its 
elements is so. We write $\mf B^\vee ({}^L \cH)$ for the set of inertial 
equivalence classes of cuspidal pairs for ${}^L \cH$, and $\mf B^\vee (\cH)$ 
for its subset of $\cH$-relevant classes.

Given an inertial equivalence class $\fs^\vee$ for ${}^L \cH$, we write,
using Definition \ref{def:7.8},
\[
\Phi_e ({}^L \cH)^{\fs^\vee} = \{ (\phi,\rho) \in \Phi_e ({}^L \cH) :
\text{ the cuspidal support of } (\phi,\rho) \text{ lies in } \fs^\vee \} .
\]
When $\fs^\vee$ is $\cH$-relevant, we put
\[
\Phi_e (\cH)^{\fs^\vee} = \{ (\phi,\rho) \in \Phi_e (\cH) :
\text{ the cuspidal support of } (\phi,\rho) \text{ lies in } \fs^\vee \} .
\]
\end{defn}
We note that $\fs^\vee$ as above determines a character of $Z(L_\sc)$. 
In view of Proposition \ref{prop:7.2}.b, it extends in a unique way 
to a character of $Z(\cH^\vee_\sc)$ trivial on $Z(L_c)^\circ$.

The above construction yields partitions analogous to \eqref{eq:8.5}:
\begin{equation}\label{eq:8.4}
\Phi_e ({}^L \cH) = \bigsqcup_{\fs^\vee \in \mf B^\vee ({}^L \cH)} 
\Phi_e ({}^L \cH)^{\fs^\vee} \quad \text{ and } \quad \Phi_e (\cH) = 
\bigsqcup_{\fs^\vee \in \mf B^\vee (\cH)} \Phi_e (\cH)^{\fs^\vee} .
\end{equation}
In this sense we consider $\Phi_e ({}^L \cH)^{\fs^\vee}$ as Bernstein
components in the space of enhanced L-parameters for ${}^L \cH$.
We note that by Lemma \ref{lem:7.12} the difference between $\mathbf{Sc}
(\phi,\rho)$ and ${}^L \Psi (\phi,\rho)$, namely the homomorphism 
\[
\mb W_F \to Z(M)^\circ : w \mapsto \chi_{\phi,v} \big( \norm{w}^{1/2} \big),
\]
can be considered as an unramified twist of ${}^L L$. Hence $\mathbf{Sc}
(\phi,\rho)$ and ${}^L \Psi (\phi,\rho)$ belong to the same inertial
equivalence class, and we could equally well have used ${}^L \Psi$ to
define $\Phi_e ({}^L \cH)^{\fs^\vee}$ and $\Phi_e (\cH)^{\fs^\vee}$.

We return to $p$-adic groups, to consider other aspects of Bernstein's work.
Bernstein associated to each inertial equivalence class $\fs = 
[\cL,\sigma]_\cH \in \mf B (\cH)$ a finite group $W_\fs$. Let
$W(\cH,\cL) = N_\cH (\cL) / \cL$, the ``Weyl'' group of $\cL$. It acts on
$\Irr_\cusp (\cL)$, which induces an action on the collection of inertial
equivalence classes $[\cL,\omega]_\cL$ with $\omega \in \Irr_\cusp (\cL)$.
Notice that
\[
(\cL,\omega_1), (\cL,\omega_2) \text{ are } \cH\text{-conjugate }
\Longleftrightarrow \text{ there is a } w \in W(\cH,\cL) \text{ with } 
w \cdot \omega_1 \cong \omega_2 .
\]
The group $W_\fs$ is defined to be the stabilizer of $[\cL,\sigma]_\cL$
in $W(\cH,\cL)$. It keeps track of which elements of $[\cL,\sigma]_\cL$
are $\cH$-conjugate. This group plays an important role in the Bernstein centre.

Let $\Rep (\cH)$ be the category of smooth complex $\cH$-representations,
and let $\Rep (\cH)^\fs$ be its subcategory generated by $\Irr (\cH)^\fs$.
The strong form of the Bernstein decomposition says that 
\[
\Rep (\cH) = \bigsqcup\nolimits_{\fs \in \mf B (\cH)} \Rep (\cH)^\fs . 
\]
By \cite[Proposition 3.14]{BeDe} the centre of the category $\Rep (\cH)^\fs$
is canonically isomorphic to $\cO ([\cL,\sigma]_\cL / W_\fs )$. Here
$[\cL,\sigma]_\cL$ is regarded as a complex affine variety via the transitive 
action of $X_\nr (\cL)$. The centre of Rep$(\cH)$ is isomorphic to
\begin{multline*}
\bigoplus_{\cL \in \mf{Lev}(\cH)} \bigoplus_{\fs = [\cL,\sigma]_\cH \in \mf B (\cH)}
Z( \Rep (\cH)^\fs) \cong \bigoplus_{\cL \in \mf{Lev}(\cH)} 
\bigoplus_{\fs = [\cL,\sigma]_\cH \in \mf B (\cH)} \cO ([\cL,\sigma]_\cL / W_\fs ) \\
= \bigoplus\nolimits_{\cL \in \mf{Lev}(\cH)} \mc O ( \Irr_\cusp (\cL) / W(\cH,\cL) ) .
\end{multline*}
In other words, there are canonical bijections
\begin{equation}\label{eq:8.6}
\begin{array}{ccc}
\Irr \big( Z (\Rep (\cH)^\fs) \big) & \longleftrightarrow & 
[\cL,\sigma]_\cL / W_\fs , \\
\Irr \big( Z (\Rep (\cH)) \big) & \longleftrightarrow &
\bigsqcup_{\cL \in \mf{Lev}(\cH)} \Irr_\cusp (\cL) / W(\cH,\cL) .
\end{array}
\end{equation}
We want identify the correct analogue of $W_\fs$ on the Galois side. 
From \eqref{eq:8.11} we see that 
$N_{\cH^\vee}({}^L L)$ stabilizes $X_\nr ({}^L L)$ and that $L$ fixes
$X_\nr ({}^L L)$ pointwise. Therefore $W({}^L \cH, {}^L L)$ also acts on classes 
$[{}^L L,\phi,\rho]_{{}^L L}$ of cuspidal data modulo unramified twists. 
We note that, like \eqref{eq:C1} and \eqref{eq:C2},
\begin{equation}\label{eq:9.1}
\begin{split}
& [{}^L L,\phi_v, q \epsilon]_{{}^L \cH} = [{}^L L,\phi'_v, q \epsilon']_{{}^L \cH} 
\qquad \Longleftrightarrow \\ 
& \text{there is a } w \in W({}^L \cH, {}^L L) \text{ such that }
w \cdot [{}^L L,\phi_v, q \epsilon]_{{}^L L} = [{}^L L,\phi'_v, q \epsilon']_{{}^L L} .
\end{split}
\end{equation}
Given any inertial equivalence class $\fs^\vee = [{}^L L,\phi_v,q \epsilon
]_{{}^L \cH}$ with underlying class $\fs^\vee_L = [{}^L L,\phi_v,q \epsilon
]_{{}^L L}$, we define
\[
W_{\fs^\vee} := \text{ the stabilizer of } \fs^\vee_L \text{ in }
W({}^L \cH,{}^L L).
\]
Now we approach this group from the Galois side. 
From $({}^L L,\phi_v, q \epsilon)$ we can build
\begin{equation}\label{eq:9.10}
v = \phi_v \big( 1,\matje{1}{1}{0}{1} \big) \quad \text{and} \quad
G = Z^1_{\cH^\vee_{\sc}}(\phi_v |_{\mb W_F}).
\end{equation}
Let $L_c$ be the inverse image of $L$ in $\cH^\vee_{\sc}$ and 
consider the cuspidal quasi-support
\begin{equation}\label{eq:9.11}
q \ft = [G \cap L_c,v,q \epsilon]_G = 
\big[ G \cap L_c, \cC_v^{G \cap L_c},q \cE \big]_G .
\end{equation}
From \eqref{eq:6.17} we get
\begin{equation}\label{eq:8.14}
W_{q \ft} = N_G \big( G \cap L_c, \cC_v^{G \cap L_c}, q \cE \big) \big/ (G \cap L_c) .
\end{equation}

\begin{lem}\label{lem:9.1}
$W_{q \ft}$ is canonically isomorphic to the isotropy group of 
$({}^L L,\phi_v, q \epsilon)$ in $W({}^L \cH,{}^L L)$ and in $W_{\fs^\vee}$.
\end{lem}
\begin{proof}
Since $X_\nr ({}^L L)$ is stable under $W({}^L \cH,{}^L L)$, any element of
the latter group which fixes $({}^L L,\phi_v,\rho)$ automatically stabilizes 
$\fs^\vee_L$. Therefore it does not matter whether we
determine the isotropy group in $W({}^L \cH,{}^L L)$ or in $W_{\fs^\vee}$.

The proof of Proposition \ref{prop:7.2}, with $G \cap L_c$ in the role of
$M$, shows that\\ 
$Z_{\cH^\vee \rtimes \mb W_F}(Z (G \cap L_c)^\circ)$ is a
Levi L-subgroup of ${}^L \cH$ minimally containing the image of $\phi$. As
$Z (G \cap L_c)^\circ \supset Z(L_c)^\circ$,
\[
Z_{\cH^\vee}(Z (G \cap L_c)^\circ) \subset Z_{\cH^\vee}(Z(L_c^\circ)) = L .
\]
But ${}^L L$ also contains the image of $\phi$ minimally, so 
\begin{equation}\label{eq:8.12}
{}^L L = Z_{\cH^\vee \cap \mb W_F}(Z (G \cap L_c)^\circ) .
\end{equation}
Suppose that $n \in N_{\cH^\vee}({}^L L)$ fixes $[\phi_v,q \epsilon]_{L}$.
Then it lies in $N_{\cH^\vee}\big(G \cap L_c,\cC_{u_\phi}^{G \cap L_c},
q \cE \big)$. The kernel of $\cH^\vee_{\sc} \to \cH^\vee_\der$ is contained
in $L_c$, so in view of \eqref{eq:8.14} $n$ lifts to a unique element of 
$W_{q \ft}$. This induces an injection
\begin{align}
\nonumber \Stab_{W({}^L \cH,{}^L L)} ([\phi_v, q \epsilon]_{L}) & \cong
\Stab_{N_{\cH^\vee}({}^L L) \cap Z_{\cH^\vee}(\phi (\mb W_F))} \big( 
\cC_v^{G \cap L_c}, q \cE \big) / Z_L (\phi (\mb W_F)) \\
\label{eq:8.13} & \cong \Stab_{N_{\cH^\vee_{\sc}}({}^L L) \cap G}\big( 
\cC_v^{G \cap L_c}, q \cE \big) / (G \cap L_c) \; \to \; W_{q \ft} .
\end{align}
The only difference between the last two terms is that on the left hand side
elements of $G$ have to normalize ${}^L L$, whereas on the right hand side
they only have to normalize $G \cap L_{\sc}$. Consider any $g \in 
N_G (G \cap L_c)$. It normalizes $Z(G \cap L_c)$, so it also normalizes 
$Z_{\cH^\vee \rtimes \mb W_F}(Z (G \cap L_c)^\circ)$, which by \eqref{eq:8.12}
equals ${}^L L$. Therefore \eqref{eq:8.13} is also surjective.
\end{proof}

Assume for the remainder of this section that $Z(\cL^\vee_\sc)$ is fixed by $\mb W_F$, 
for every Levi subgroup $\cL \subset \cH$. (The general case is similar and can be 
obtained by  including characters $\zeta_\cH^\cL$ as in Lemma \ref{lem:7.14}.) In view 
of Lemma \ref{lem:9.1}, the analogue of the Bernstein centre \eqref{eq:8.6} becomes
\begin{equation}\label{eq:C3}
\bigsqcup_{\cL \in \mf{Lev}(\cH)} \Phi_\cusp (\cL) \big/ W({}^L \cH,{}^L \cL) =
\bigsqcup_{\fs^\vee = [{}^L \cL,\phi_v,q\epsilon]_{{}^L \cH} \in \mf B^\vee (\cH)}
\fs_\cL^\vee / W_{\fs^\vee} .
\end{equation}
Thus we interpret the "Bernstein centre of $\Phi_e (\cH)$" as the quotient 
along the map
\[
\mathbf{Sc} : \Phi_e (\cH) \to \bigsqcup\nolimits_{\cL \in \mf{Lev}(\cH)} 
\Phi_\cusp (\cL) \big/ W({}^L \cH,{}^L \cL) .
\]
Let us agree that two enhanced L-parameters in the same Bernstein component are 
inseparable if they have the same infinitesimal character. Then \eqref{eq:C3}
can be regarded as a maximal separable quotient of $\Phi_e (\cH)$. This fits
nicely with the Dauns--Hofmann theorem, which says that for many noncommutative
algebras $A$ the operation of taking the maximal separable quotient of $\Irr (A)$
is dual to restriction from $A$ to its centre $Z(A)$.

\section{Extended quotients and L-parameters}
\label{sec:extquot}

The ABPS-conjecture from \cite[\S 15]{ABPS2} and \cite[Conjecture 2]{ABPS7}
refines \eqref{eq:8.6}. In its roughest form it asserts that it can be 
lifted to a bijection
\begin{equation}\label{eq:8.9}
\Irr (\cH)^\fs \; \longleftrightarrow \; ( [\cL,\sigma]_\cL \q W_\fs )_\natural , 
\end{equation}
for a suitable family of 2-cocycles $\natural$. Equivalently, this can be
formulated as a bijection
\begin{equation}\label{eq:9.2}
\Irr (\cH) \; \longleftrightarrow \; \bigsqcup\nolimits_{\cL \in \mf{Lev}(\cH)} 
\big( \Irr_\cusp (\cL) \q W (\cH ,\cL) \big)_\natural . 
\end{equation}
The main goal of this section is to prove an analogue of \eqref{eq:8.9} and
\eqref{eq:9.2} for enhanced Langlands parameters, which refines \eqref{eq:C3}.

Fix a $\cH$-relevant cuspidal datum $({}^L L, \phi_v, q \epsilon)$ for ${}^L \cH$,
and write, in addition to the notations \eqref{eq:9.10} and \eqref{eq:9.11},
\begin{equation}\label{eq:9.6}
q \ft = [G \cap L_c,v,q\epsilon]_G ,\quad
\ft^\circ = [G^\circ \cap L_c, \cC_v^{G^\circ \cap L_c},\cE]_{G^\circ} .
\end{equation}
The next result is a version of the generalized Springer correspondence 
with enhanced L-parameters.

\begin{prop}\label{prop:10.1}  
\enuma{
\item There is a bijection
\[
\begin{array}{cccc}
{}^L \Sigma_{q \ft} : & {}^L \Psi^{-1} ({}^L L, \phi_v, q \epsilon) & 
\longleftrightarrow & \Irr (\C [W_{q \ft}, \kappa_{q \ft}]) \\
 & (\phi,\rho) & \mapsto & q \Sigma_{q \ft}(u_\phi,\rho) \\
 & (\phi |_{\mb W_F},q \Sigma_{q \ft}^{-1}(\tau) ) & \reflectbox{$\mapsto$} & \tau 
\end{array}
\]
It is canonical up to the choice of an isomorphism as in Lemma \ref{lem:6.2}.
\item Recall that Theorems \ref{thmlus}.(3) and \ref{thm:2.2}.c give a canonical 
bijection $\Sigma_{\ft^\circ}$ between $\Irr_\C (W_{\ft^\circ}) = 
\Irr (\End_{G^\circ} (\pi_* \tilde{\cE}))$ and $\Psi_{G^\circ}^{-1}(\ft^\circ) 
\subset \cN_{G^\circ}^+$. It relates to part (a) by
\[
{}^L \Sigma_{q \ft} (\phi,\rho) |_{W_{\ft^\circ}} = 
\bigoplus\nolimits_i \Sigma_{\ft^\circ} (u_\phi,\rho_i) ,
\]
where $\rho = \bigoplus_i \rho_i$ is a decomposition into irreducible
$A_{G^\circ}(u_\phi)$-subrepresentations.
\item The $\cH^\vee$-conjugacy class of $(\phi|_{\mb W_F},u_\phi,\rho_i)$
is determined by any irreducible $\C [W_{\ft^\circ}]$-subrepresentation of
${}^L \Sigma_{q \ft}(\phi,\rho)$.
} 
\end{prop}
\begin{proof}
(a) By Theorem \ref{thm:6.3} (with $\C$-coefficients) every $(\phi,\rho) \in
{}^L \Psi^{-1} ({}^L L, \phi_v, q \epsilon)$ determines a unique irreducible
representation $q \Sigma_{q \ft}(u_\phi,\rho)$ of $\C [W_{q \ft}, \kappa_{q \ft}]$.
Conversely, every $\tau \in \Irr (\C [W_{q \ft}, \kappa_{q \ft}])$ gives rise to
a unique $q \Sigma_{q \ft}^{-1}(\tau) = (u_\phi,\rho) \in \cN_G^+$, and that
determines an enhanced L-parameter $(\phi|_{\mb W_F},u_\phi,\rho)$ for ${}^L \cH$.
It remains to see that $(\phi|_{\mb W_F},u_\phi,\rho)$ is $\cH$-relevant.
By \eqref{eq:6.14} $\rho$ has the same $Z(\cH^\vee_{\sc})^{\mb W_F}$-character
as $q \epsilon$. By the assumed $\cH$-relevance of $q \epsilon ,\; \rho$ is
$\cH$-relevant. By Definition \ref{def:7.9}, $(\phi|_{\mb W_F},u_\phi,\rho)$
is also $\cH$-relevant.\\
(b) This is a direct consequence of Theorem \ref{thm:6.3}.b.\\
(c) By the irreducibility of $\rho$, all the $\rho_i$ are $Z_G (\phi)$-conjugate.
Similarly the irreducibility of the $\C [W_{q \ft}, \kappa_{q \ft}]$-representation
$\tau = {}^L \Sigma_{q \ft}(\phi,\rho)$ implies (with Theorem \ref{thm:1.2}) that
all the irreducible $\C [W_{\ft^\circ}]$-constituents $\tau_i$ of $\tau$ are
$W_{q \ft}$-conjugate. By part (b) $\tau_i$ determines a pair $(u_\phi,\rho_i)$
up to $G^\circ$-conjugacy. Hence it determines $(\phi|_{\mb W_F},u_\phi,\rho_i)$
up to $\cH^\vee$-conjugacy.
\end{proof}

We will promote Proposition \ref{prop:10.1} to a statement involving extended
quotients. By Lemma \ref{lem:9.1} $W_{\fs^\vee,\phi_v,q \epsilon} = W_{q \ft}$, so we 
can regard $\kappa_{q \ft}$ as a 2-cocycle $\kappa_{\phi_v,q \epsilon}$ of 
$W_{\fs^\vee,\phi_v,q \epsilon}$. Then we can build 
\begin{multline*}
\widetilde{\fs^\vee_L} = \big( [{}^L L, \phi_v, q \epsilon]_{{}^L L} \big)^\sim_\kappa \\
= \big\{ \big( ({}^L L, z \phi_v, q \epsilon), \rho \big) : z \in X_\nr ({}^L L),
\rho \in \Irr (\C [W_{\fs^\vee,\phi_v,q \epsilon}, \kappa_{\phi_v,q \epsilon}]) \big\} .
\end{multline*}
Comparing with \eqref{eq:extquot}, we see that we still need an action on
$W_{\fs^\vee}$ on this set. 

\begin{lem}\label{lem:10.2}
Let $w \in W_{\fs^\vee}$ and $z \in X_\nr ({}^L L)$ with $w (\phi_v,q \epsilon) \cong
(z \phi, q \epsilon)$. There exists a family of algebra isomorphisms (for various
such $w,z$)
\[
\psi_{w,\phi_v,q \epsilon} : \C [W_{\fs^\vee,\phi_v,q \epsilon}, \kappa_{\phi_v,q \epsilon}]
\to \C [W_{\fs^\vee,z \phi_v,q \epsilon}, \kappa_{z \phi_v,q \epsilon}]
\]
such that:
\enuma{
\item The family is canonical up to the choice of isomorphisms
$\C [W_{\fs^\vee,\phi_v,q \epsilon}, \kappa_{\phi_v,q \epsilon}] \cong
\End_G (\pi_* (\widetilde{q \cE}))$ as in Lemma \ref{lem:6.2}.
\item $\psi_{w,\phi_v,q \epsilon}$ is conjugation with $T_w$ if 
$w \in W_{\fs^\vee,\phi_v,q \epsilon}$.
\item $\psi_{w',z \phi_v,q \epsilon} \circ \psi_{w,\phi_v,q \epsilon} =
\psi_{w' w,\phi_v,q \epsilon}$ for all $w' \in W_{\fs^\vee}$.
\item ${}^L \Sigma_{q \ft}^{-1}(\rho) \cong {}^L \Sigma_{w (q \ft)}^{-1}
(\rho \circ \psi_{w,\phi_v,q \epsilon}^{-1} )$ for all $\rho \in
\Irr (\C [W_{\fs^\vee,\phi_v,q \epsilon}, \kappa_{\phi_v,q \epsilon}])$.
}
\end{lem}
\begin{proof}
(a) Recall from Lemma \ref{lem:6.2} that
\begin{equation}\label{eq:9.3}
\C [W_{\fs^\vee,\phi_v,q \epsilon}, \kappa_{\phi_v,q \epsilon}] \cong
\C [W_{q \ft}, \kappa_{q \ft}] \cong \End_G (\pi_* (\widetilde{q \cE})) .
\end{equation}
We fix such isomorphisms. For any $n \in N_{\cH^\vee}({}^L L)$ representing $w$:
\begin{equation}\label{eq:9.4}
\begin{split}
& \C [W_{\fs^\vee,w(\phi_v),w(q \epsilon)}, \kappa_{w(\phi_v),w(q \epsilon)}] \cong 
\C [W_{n(q \ft)}, \kappa_{n (q \ft)}] \cong \End_G (\pi_* (\widetilde{n \cdot q \cE})), \\
& \C [W_{\fs^\vee, z \phi_v,q \epsilon}, \kappa_{z \phi_v ,q \epsilon}] \cong
\End_{n G n^{-1}} (\pi_* (\widetilde{q \cE})) \cong 
\End_G (\pi_* (\widetilde{\mathrm{Ad}(n)^* q \cE})) ,
\end{split}
\end{equation}
where $\pi_* (\widetilde{q \cE})$ now denotes a sheaf on $n Y n^{-1}$.
By assumption there exists a $L_{\sc}$-intertwining map 
\begin{equation}\label{eq:9.5}
q \cE \to \mathrm{Ad}(n)^* q \cE , 
\end{equation}
and by the irreducibility of $q \cE$ it is unique up to scalars. In the same way as in
\cite[\S 3]{Lus1} and in the proof of Lemma \ref{lem:6.2}, it gives rise to an 
isomorphism of $G$-equivariant local systems
\[
q b_w : \pi_* (\widetilde{q \cE}) \to \pi_* (\widetilde{\mathrm{Ad}(n)^* q \cE}) .
\]
In view of \eqref{eq:9.4} and the essential uniqueness of \eqref{eq:9.5}, conjugation
by $q b_w$ gives a canonical algebra isomorphism
\[
\tilde \psi_{w,\phi_v,q \epsilon} : \End_{G} (\pi_* (\widetilde{q \cE})) \to
\End_{n G n^{-1}} (\pi_* (\widetilde{q \cE})) .
\]
We define $\psi_{w,\phi_v,q \epsilon}$ as the composition of \eqref{eq:9.3},
$\tilde{\psi}_{w,\phi_v,q \epsilon}$ and \eqref{eq:9.4}. \\
(b) For $w \in W_{\fs^\vee,\phi_v,q \epsilon}$ we thus obtain conjugation by the image 
of $q b_w$ which by construction (see the proof of Lemma \ref{lem:6.2}) is $T_w$. \\
(c) The canonicity ensures that 
\[
\tilde{\psi}_{w',z \phi_v,q \epsilon} \circ \tilde{\psi}_{w,\phi_v,q \epsilon} =
\tilde{\psi}_{w' w,\phi_v,q \epsilon} 
\]
which automatically leads to (c).\\
(d) By Theorem \ref{thm:6.3}.a 
\[
n \cdot q \Sigma_{q \ft}^{-1}(\tilde \rho) \cong q \Sigma_{w (q \ft)}^{-1}
(\tilde \rho \circ \tilde{\psi}_{w,\phi_v,q \epsilon}^{-1} ) \quad \text{for} 
\quad \tilde \rho \in \Irr \big( \End_G (\pi_* (\widetilde{q \cE})) \big).
\]
Since ${}^L \Sigma_{q \ft}$ was defined using \eqref{eq:9.3}, we obtain property (d).
\end{proof}
 
We could have characterized $\psi_{w,\phi_v,q \epsilon}^{-1}$ also with property
(d) of Lemma \ref{lem:10.2} only, that would suffice for our purposes. 
However, then one would not see so readily that the map is exactly as canonical as 
our earlier constructions.

\begin{thm}\label{thm:10.3}
\enuma{
\item Let $\fs^\vee_L = [{}^L L,\phi_v, q  \epsilon]_{{}^L L}$ be an $\cH$-relevant 
inertial equivalence class for the Levi L-subgroup ${}^L L$ of ${}^L \cH$ and recall
the notations \eqref{eq:9.6}. The maps ${}^L \Sigma_{q \ft}$ from Proposition 
\ref{prop:10.1}.a combine to a bijection
\[
\begin{array}{ccc}
\Phi_e ({}^L \cH )^{\fs^\vee} & \longleftrightarrow & 
\big( \Phi_e ({}^L L)^{\fs^\vee_L} \q W_{\fs^\vee} \big)_\kappa \\
(\phi,\rho) & \mapsto & \big( {}^L \Psi (\phi,\rho), q\Sigma_{q \ft}(u_\phi,\rho) \big) \\
\big( \phi_v |_{\mb W_F}, q \Sigma_{q \ft}^{-1}(\tau) \big) & \reflectbox{$\mapsto$} & 
({}^L L, \phi_v, q \epsilon, \tau)
\end{array}
\]
\item The bijection from part (a) has the following properties:
\begin{itemize}
\item It preserves boundedness of (enhanced) L-parameters.
\item It is canonical up to the choice of isomorphisms as in \eqref{eq:9.3}.
\item The restriction of $\tau$ to $W_{\ft^\circ}$ canonically determines the 
(non-enhanced) L-parameter in ${}^L \Sigma_{q \ft}(\tau)$.
\item Let $z,z' \in X_\nr ({}^L L)$ and let $\Gamma \subset W_{\fs^\vee,z \phi_v,
q \epsilon}$ be a subgroup. Suppose that $\Gamma = \overline{\Gamma} / L \cong 
\overline{\Gamma}_c / L_c$, where 
\[
\overline{\Gamma} \subset N_{\cH^\vee}({}^L L) \cap Z^1_{\cH^\vee}(z' \phi |_{\mb W_F}) 
\quad \text{with preimage} \quad \overline{\Gamma}_c \subset
Z_{\cH^\vee_{\sc}}(z' \phi (\mb W_F))^\circ.
\]
Then the 2-cocycle $\kappa_{\fs^\vee,z \phi_v,q \epsilon}$ is trivial on $\Gamma$.
\end{itemize}
\item Let $\zeta_\cH \in \Irr (Z(\cH^\vee_\sc))$ and $\zeta_\cH^\cL$ be as in
Lemma \ref{lem:7.14}. We write
\[
\Phi_{e,\zeta_\cH} (\cH,\cL) = \{ (\phi,\rho) \in \Phi_{e,\zeta_\cH} (\cH) : 
\mathbf{Sc}(\phi,\rho) \in \Phi_\cusp (\cL) \} .
\]
The bijections from part (a) give a bijection
\[
\Phi_{e,\zeta_\cH} (\cH,\cL) \; \longleftrightarrow \; 
\big( \Phi_{\cusp,\zeta_\cH^\cL} (\cL) \q W(\cH,\cL) \big)_\kappa . 
\]
\item Let $\mf{Lev}(\cH)$ be a set of representatives for the conjugacy classes of Levi
subgroups of $\cH$. The maps from part (c) combine to a bijection
\[
\Phi_{e,\zeta_\cH} (\cH) \; \longleftrightarrow \; \bigsqcup\nolimits_{\cL \in \mf{Lev}(\cH)}
\big( \Phi_{\cusp,\zeta_\cH^\cL} (\cL) \q W(\cH,\cL) \big)_\kappa . 
\]
\item Assume that $Z(\cL^\vee_\sc)$ is fixed by $\mb W_F$ for every Levi subgroup 
$\cL \subset \cH$. (E.g. $\cH$ is an inner twist of a split group.)
Let $\cH^u$ be the inner twist of $\cH$ determined by $u \in H^1(F,\cH_\ad) \cong 
\Irr_\C \big( Z (\cH^\vee_{\sc} )^{\mb W_F} \big)$.
The union of part (d) for all such $u$ is a bijection
\[
\Phi_e ({}^L \cH) \; \longleftrightarrow \; \bigsqcup_{u \in H^1 (F,\cH_{\ad})} 
\bigsqcup_{\cL^u \in \mf{Lev}(\cH^u)} 
\big( \Phi_\cusp (\cL^u) \q W(\cH^u,\cL^u) \big)_\kappa . 
\]
}
\end{thm}
\begin{proof}
(a) Proposition \ref{prop:10.1}.a gives a bijection
\begin{align}\label{eq:9.7}
& {}^L \Psi^{-1}({}^L L,\phi_v,q \epsilon) \; \longleftrightarrow \; \Irr (\C [W_{\fs^\vee,
\phi_v,q \epsilon}, \kappa_{\phi_v,q \epsilon}]) \; \longleftrightarrow \\
\nonumber \bigsqcup_{W_{\fs^\vee} / W_{\fs^\vee,\phi_v, q \epsilon}} & 
\Irr (\C [W_{\fs^\vee, w(\phi_v),w(q \epsilon)}, \kappa_{w(\phi_v),w(q \epsilon)}]) 
\big/ W_{\fs^\vee}
= \big( W_{\fs^\vee} \cdot ({}^L L,\phi_v,q \epsilon) \q W_{\fs^\vee} \big)_\kappa .
\end{align}
For $z \in X_\nr ({}^L L)$ the pre-images ${}^L \Psi^{-1}({}^L L,\phi_v,q \epsilon)$
and ${}^L \Psi^{-1}({}^L L,z \phi_v,q \epsilon)$ intersect in $\Phi_e ({}^L \cH)$ if
and only if their $L$-conjugacy classes differ by an element of $W_{\fs^\vee}$. Hence
the maps \eqref{eq:9.7} combine to the desired bijection.\\
(b) It preserves boundedness because it does not change $\phi |_{\mb W_F}$. The second
and third properties follow from Proposition \ref{prop:10.1}.c.

Write $G_{z'} = Z^1_{\cH^\vee_{\sc}}(z' \phi |_{\mb W_F})$ and consider $\Gamma$ as a
subgroup of $W_{z' \ft^\circ} = N_{G_{z'}^\circ}(z' \ft^\circ) / L_{\sc}$. 
Let $\pi_* (\widetilde{\cE}_{z'})$ be the $G_{z'}$-equivariant sheaf constructed 
like $\pi_* (\widetilde{\cE})$, but with $G_{z'}$ instead of $G$.
For $\gamma \in \Gamma$ the proof of Lemma \ref{lem:6.2} provides a canonical element 
$qb'_\gamma \in \End_{G_{z'}}(\pi_* (\widetilde{\cE}_{z'}))$, such that
\[
\Gamma \to \Aut_{G_{z'}} (\pi_* (\widetilde{\cE}_{z'})) : \gamma \mapsto qb'_\gamma
\]
is a group homomorphism. Let $n \in G_{z'}^\circ \cap G_z$ be a lift of $\gamma$. 
Then $qb'_\gamma$ restricts to an isomorphism
\begin{equation}\label{eq:9.9}
q\cE = z' q \cE \to \mathrm{Ad}(n)^* (z' q \cE) = \mathrm{Ad}(n)^* (q \cE) .
\end{equation}
As in the proof of Lemma \ref{lem:10.2}, \eqref{eq:9.9} gives rise to an element
$q b_\gamma \in \End_{G_{z}}(\pi_* (\widetilde{\cE}_{z}))$. 
We can choose the basis element
\[
T_\gamma \quad \text{ of } \quad \C [W_{z q \ft}, \kappa_{z q \ft}] = 
\C[W_{\fs^\vee,z \phi_v,q \epsilon}, \kappa_{z \phi_v,q \epsilon}]
\]
to be the image of $qb_\gamma$ under \eqref{eq:9.3}. Then the $\C$-span of
$\{ T_\gamma : \gamma \in \Gamma \}$ is isomorphic to $\C[\Gamma]$, which shows
that $\kappa_{z \phi_v,q \epsilon} \big|_{\Gamma \times \Gamma} = 1$.\\
(c) The union of the instances of part (a) with ${}^L L = {}^L \cL$ 
and $q \epsilon |_{Z(\cL^\vee_\sc)} = \zeta_\cH^\cL$ yields a surjection
\begin{equation}\label{eq:9.8}
\bigsqcup\nolimits_{(\phi_v,q \epsilon) \in \Phi_{\cusp,\zeta_\cH^\cL} (\cL) / 
X_\nr ({}^L \cL)} \big( \Phi_e (\cL)^{\fs^\vee_{\cL}} \q W_{\fs^\vee} 
\big)_\kappa \to \Phi_{e,\zeta_\cH} (\cH,\cL) .
\end{equation}
Two elements $(\phi_v,q \epsilon,\tau)$ and $(\phi'_v,q \epsilon',\tau')$ on the
left hand side can only have the same image in $\Phi_{e,\zeta_\cH} (\cH,\cL)$ if they 
have the same cuspidal support modulo unramified twists, for the map in Proposition 
\ref{prop:10.1}.a preserves that. By \eqref{eq:9.1} the inertial equivalence classes 
of $(\phi_v,\tau)$ and $(\phi'_v,\tau')$ differ only by an element of
$W({}^L \cH,{}^L \cL) \cong W(\cH,\cL)$. We already know that the restriction of
\eqref{eq:9.8} to one inertial equivalence class is injective. Hence every fiber
of \eqref{eq:9.8} is in bijection with $W({}^L \cH,{}^L \cL) / W_{\fs^\vee}$ for
some $\fs^\vee$. 

By Lemma \ref{lem:9.1} $\widetilde{\Phi_\cusp (\cL)}$ (with respect to 
$W({}^L \cH,{}^L \cL)$) equals the disjoint union $\bigsqcup_{\fs^\vee_\cL} 
\widetilde{(\fs^\vee_\cL )}_\kappa$. 
In view of part (a), there is a unique way to extend the action of $W_{\fs^\vee}$
on $\widetilde{(\fs^\vee_\cL )}_\kappa$ (for various $\fs^\vee_\cL = 
[{}^L \cL ,\phi'_v,q \epsilon']_{{}^L \cL}$) to an action of $W ({}^L \cH,{}^L \cL)$ 
on $\widetilde{\Phi_\cusp (\cL)}$ such that maps from part (a) become constant on
$W ({}^L \cH,{}^L \cL)$-orbits. Then
\[
\big( \Phi_{\cusp,\zeta_\cH^\cL} (\cL) \q W ({}^L \cH,{}^L \cL) \big)_\kappa = 
\widetilde{\Phi_{\cusp,\zeta_\cH^\cL} (\cL)} / W ({}^L \cH,{}^L \cL)  
\to \Phi_{e,\zeta_\cH} (\cH,\cL)
\]
is the desired bijection.\\
(d) This is a direct consequence of part (c).\\
(e) By \eqref{eq:7.8} and Definition \ref{def:7.9} 
\[
\Phi_e ({}^L \cH) = \bigsqcup\nolimits_{u \in H^1 (F,\cH_{\ad})} \Phi_e (\cH^u) .
\]
By the assumption $\Phi_\cusp (\cL^u) = \Phi_{\cusp,\zeta_u}(\cL^u)$ for every
extension $\zeta_u$ of the Kottwitz parameter of $\cL^u$ to a character of 
$Z(\cL^\vee_\sc)$, for there is nothing to extend to. Now apply part (d).
\end{proof}

The canonicity in part (b) can be expressed as follows. Given $({} \cL,\phi_v,
q\epsilon, \tau^\circ)$ with $\tau^\circ \in \Irr (W_{\ft^\circ})$, the set
\[
\big\{ (\phi_v |_{\mb W_F}, q\Sigma_{q \ft}^{-1}(\tau)) \in \Phi_e ({}^L \cH)^{\fs^\vee} : 
\tau \in \Irr (\C [W_{\fs^\vee,\phi_v ,q\epsilon},\kappa_{\phi_v ,q\epsilon}])
\text{ contains } \tau^\circ \big\}
\]
is canonically determined.

It would be interesting to know when the above 2-cocycles $\kappa$ are
trivial on $W_{\fs^\vee}$. Theorem \ref{thm:10.3}.b shows that this happens quite 
often, in particular whenever $W_{\fs^\vee}$ fixes a point 
$({}^L L,z' \phi_v,q \epsilon) \in \fs^\vee$ and at the same time $W_{\fs^\vee}$ equals the
Weyl group $W(G_{z'}^\circ,L)$, where $G_{z'} = Z^1_{\cH^\vee_{\sc}}(z' \phi |_{\mb W_F})$.

\begin{ex}\label{ex:B}
Yet there are also cases in which $\kappa$ is definitely not trivial. Take 
$\cH = \SL_5 (D)$, where $D$ is a quaternion division algebra over $F$. This is an 
inner form of $\SL_{10}(F)$ and ${}^L \cH = \PGL_{10}(\C) \times \mb W_F$. 

We will rephrase Example \ref{ex:A} with L-parameters. We can ignore the factor 
$\mb W_F$ of ${}^L \cH$, because it acts trivially on $\cH^\vee$.
Let $\overline{\phi} : \mb W_F \to \SL_2 (\C)^5$ be a group homomorphism whose image
is the group $Q$ from Example \ref{ex:A}. It projects to a homomorphism $\phi |_{\mb W_F}:
\mb W_F \to \PGL_{10}(\C)$. Let $u$ and $\epsilon$ be as in the same example.
These data determine an enhanced L-parameter $(\phi,\epsilon)$ for $\cH$.
The group 
\[
G = Z_{\SL_{10}(\C)}(Q) = Z_{\cH^\vee_{\sc}}(\phi (\mb W_F))
\] 
was considered in Example \ref{ex:A}. We checked over there that $W_{\fs^\vee} \cong 
(\Z / 2 \Z)^2$ and that its 2-cocycle $\kappa_{\fs^\vee,\phi,\epsilon}$ is nontrivial. 
We remark that this fits with the non-triviality of the 2-cocycle in 
\cite[Example 5.5]{ABPS4}, which is essentially the same example, but on the $p$-adic
side of the LLC.
\end{ex}

Just like ${}^L \Psi$ in Lemma \ref{lem:8.3}, the maps from Theorem \ref{thm:10.3}
are compatible with the Langlands classification for L-parameters from Theorem 
\ref{thm:8.2}.

\begin{lem}\label{lem:10.4}
Let $(\phi,\rho) \in \Phi_e (\cH)$ and let $(\cL,\phi_t,z,\rho_t)$ be the enhanced
standard triple associated to it by Theorem \ref{thm:8.2}. 
\enuma{
\item Write $q \Psi_G (u_\phi,\rho) = q \ft = [M,\cC_v^M,q\epsilon]_G ,\; G_2 = 
Z^1_{\cL^\vee_{\sc}}(\phi |_{\mb W_F}) ,\; M_2 = Z_{G_2}(Z(M)^\circ)$ and 
$q \ft_\cL = [M_2,\cC_v^{M_2},q\epsilon_2]_{G_t}$ as in the proof of Lemma \ref{lem:8.3}.
Then $W_{q \ft} \cong W_{q \ft_\cL}$.
\item The image of $(\phi,\rho) \in \Phi_e (\cH)$ under Theorem \ref{thm:10.3}.a
equals the image of $(z \phi_t,\rho_t) \in \Phi_e (\cL)$. The latter can be expressed
as $z \cdot ({}^L \Psi^\cL (\phi_t,\rho_t), q \Sigma_{q \ft_\cL}(u_{\phi_t},\rho_t))$.
}
\end{lem}
\begin{proof}
Because all the maps are well-defined on conjugacy classes of enhanced L-parameters,
we may assume that $\phi = z \phi_t$ and $\rho = \rho_t$.\\
(a) Recall from Lemma \ref{lem:9.1} that
\begin{equation}\label{eq:10.1}
\begin{split}
& W_{q \ft} \cong W({}^L \cH,Z_{{}^L \cH}(Z(M)^\circ))_{\phi_v,q \epsilon} , \\
& W_{q \ft_\cL} \cong W({}^L \cL,Z_{{}^L \cH}(Z(M)^\circ))_{\phi_v,q \epsilon_2} . 
\end{split}
\end{equation}
The argument following \eqref{eq:7.29} shows that we may replace $q \epsilon_2$
by $q \epsilon$ here.
Let $L_\emptyset$ be the unique minimal standard Levi subgroup of $\cH^\vee$. Then
\begin{align*}
& W({}^L \cH,Z_{{}^L \cH}(Z(M)^\circ)) \cong N_{W({}^L \cH,{}^L L_\emptyset)} (Z(M)^\circ))
/ W(Z_{{}^L \cH}(Z(M)^\circ), {}^L L_\emptyset) ,\\
& W({}^L \cL,Z_{{}^L \cH}(Z(M)^\circ)) \cong N_{W({}^L \cL,{}^L L_\emptyset)} (Z(M)^\circ))
/ W(Z_{{}^L \cH}(Z(M)^\circ), {}^L L_\emptyset) .
\end{align*}
Recall from Definition \ref{def:8.1} that $\phi_v |_{\mb W_F} \phi |_{\mb W_F} = 
z \phi_t |_{\mb W_F}$, where $z \in X_\nr ({}^L \cL) = (Z (\cL^\vee)^{\mb I_F} )^\circ_\Fr$ 
is strictly positive with respect to the standard parabolic subgroup $\mc P$ of $\cH$ 
having $\cL$ as Levi factor. Hence the isotropy group of $z$ in the Weyl group 
$W({}^L \cH,{}^L L_\emptyset)$ is the group generated by the reflections that fix $z$, 
which is precisely $W({}^L \cL, {}^L L_\emptyset)$. 

Since $\phi_t$ is bounded and $z$ determines an unramified character of $\cL$ with 
values in $\R_{>0}$, every element of $W_{q \ft}$ must fix both $\phi_t |_{\mb W_F}$ 
and $z$. By the above $W_{q \ft} \subset W({}^L \cL, {}^L L_\emptyset)$, and then
\eqref{eq:10.1} shows that $W_{q \ft} = W_{q \ft_\cL}$.\\
(b) From Lemma \ref{lem:8.3} we know that ${}^L \Psi^\cH (\phi,\rho) = {}^L \Psi^\cL 
(\phi,\rho)$. By Proposition \ref{prop:6.4} 
\[
q \Sigma_{q \ft_\cL}(u_\phi,\rho) \quad \text{is a constituent of} \quad
\Res_{\End_{G_2} (\pi_*^{G_2} (\widetilde{q \cE_{G_2}}))}^{
\End_G (\pi_* (\widetilde{q \cE}))} q \Sigma_{q \ft}(u,\eta) .
\]
By Lemma \ref{lem:6.2} 
\[
\End_G (\pi_* (\widetilde{q \cE})) \cong \C [W_{q \ft}, \kappa_{q \ft}] 
\quad \text{and} \quad \End_{G_t} (\pi_*^{G_t} (\widetilde{q \cE_{G_t}})) 
\cong \C [W_{q \ft_\cL}, \kappa_{q \ft_\cL}] .
\]
By Proposition \ref{prop:6.4}.b $\kappa_{q \ft} |_{W_{q \ft_\cL}} = \kappa_{q \ft_\cL}$,
so in view of part (a) these two algebras are equal. Thus
$q \Sigma_{q \ft_\cL}(u_\phi,\rho) = q \Sigma_{q \ft}(u_\phi,\rho)$. Together with
Lemma \ref{lem:8.3} this shows that the image of $(\phi,\rho)$ under Theorem 
\ref{thm:10.3}.a is the same for $\cH$ and for $\cL$.

Since $z$ lifts to a central element of $\cL^\vee$, $q \ft_\cL$ is the same for
$(\phi_t,\rho_t)$ and $(z \phi_t,\rho_t)$. That goes also for 
$q \Sigma_{q \ft_\cL}(u_{\phi_t},\rho_t)$. In combination with Lemma \ref{lem:8.3}
we find that 
\[
({}^L \Psi^\cL (z \phi_t,\rho_t), q \Sigma_{q \ft_\cL}(u_{\phi_t},\rho_t)) =
z \cdot ({}^L \Psi^\cL (\phi_t,\rho_t), q \Sigma_{q \ft_\cL}(u_{\phi_t},\rho_t)) . \qedhere
\]
\end{proof}

\appendix

\section{Erratum (2025)}

In 2025 Simon Riche pointed out to us that Theorem \ref{thm:4.1}.a does not hold 
in the stated generality. 

Namely, $G^\circ$ could be a group $\mr{Spin}_N (\C) \times \mr{Spin}_N (\C)$ 
such that $N$ is both a square and a triangular number. On the first copy of $\mr{Spin}_N
(\C)$ one can take the cuspidal local system $\mc L_+$ supported on the unipotent class
$\mc C_+$ (as in the proof of Theorem \ref{thm:4.1}), the second copy one can take 
cuspidal local system $\mc L_-$ supported on the unipotent class $\mc C_-$. Then
$\mc L = \mc L_+ \boxtimes \mc L_-$ is a cuspidal local system supported on the 
unipotent class $\mc C_+ \times \mc C_-$ in $G^\circ$. 

Let $\tau$ be the automorphism of $G^\circ$ that swaps the two copies of $\mr{Spin}_N (\C)$. 
Then $\tau (\mc C_+ \times \mc C_-) = \mc C_- \times \mc C_+$ is another unipotent class 
in $G^\circ$ that supports a cuspidal local system, namely $\tau^* (\mc L) =
\mc L_- \boxtimes \mc L_+$. The central characters of $\mc L$ and $\tau^* (\mc L)$
differ, but they are in the same Aut$(G^\circ)$-orbit. 

Below we repair this, by settling for a somewhat weaker statement of Theorem \ref{thm:4.1}.a.
This means that in general the unipotent class $\mc C_u^{G^\circ}$ is not stable under
all automorphisms of $G^\circ$, and it need not be equal to $\mc C_u^G$. Although 
an equality of the form $\mc C_u^G = \mc C_u^{G^\circ}$ is used several times in the
paper, that is mainly for convenience, with the below corrections it can be avoided.
Most of the time it suffices to replace $G$ by $G_{\mc C_u^{G^\circ}}$, 
the stabilizer of $\mc C_u^{G^\circ}$.\\[1mm]

\textbf{\large Corrections.}
\begin{itemize}
\item p.15, second line of Theorem \ref{thm:4.1}: add "Let $\Aut^s (G^\circ)$ be the 
subgroup of $\Aut (G^\circ)$ that stabilizes every simple factor of $G^\circ$."
\item p.15, statement of Theorem \ref{thm:4.1}.a: replace "$\Aut (G^\circ)$" by
"$\Aut^s (G^\circ)$"
\item p.15, sixth line of proof of Theorem \ref{thm:4.1}: replace "Aut" by "$\Aut^s$" 
\item p.15, seventh line of proof of Theorem \ref{thm:4.1}: replace "$\bar X$ 
decomposes" by "$\bar X$ and $\Aut^s (\tilde G)$ decompose" 
\item p.16, remove lines -6 and -5.
\item p.16--18, replace every occurrence of "$G / G^\circ$" by 
"$G_{\mc C_u^{G^\circ}} / G^\circ$" (In Example \ref{ex:A} this does not make 
a difference.)
\item p.19, line 19, replace "$Y$" by "$Y^\circ = \mc C_u^{G^\circ} Z(G^\circ)^\circ$"
\item p.19, \eqref{eq:4.13}, replace "$Y \times G$" by "$\widehat{Y^\circ} =
\{ (y,g) \in Y \times G : g^{-1} y g \in Y^\circ \}$"
\item p.19, lines 24--27, replace "$Y \times G$" by "$\widehat{Y^\circ}$"
(three times)
\item p.19, line 28, after "fibration", add "with fibre $G_{\mc C_u^{G^\circ}} / G^\circ$"
\item p.19, line 31, replace "$G / G^\circ$" by 
"$G_{\mc C_u^{G^\circ}} / G^\circ$"
\item p.23, lines 3--7, replace "$N_G (L)/L$" by "$N_G (L)_{\mc C_v^L} / L$" (three times)
\item p.24, line 6, replace "By Theorem \ref{thm:4.1}.a $N_G (L)$ stabilizes 
$\mc C_v^L$" by "Recall from \eqref{eq:4.1} that"
\item p.29, line 1, remove "$= \mc C_v^{M^\circ}$"
\item p.29, line 6, replace "$S = \mc C_v^M Z(M)^\circ$" by 
"$S = \mc C_v^{M^\circ} Z(M)^\circ$" 
\item p.31, line 14, remove "by Theorem \ref{thm:4.1}.a" 
\item p.31, line 15 and 21, replace "$N_G (M^\circ) / M$" by 
"$N_G (M^\circ)_{\mc C_v^{M^\circ}} M / M$"
\end{itemize}

\end{document}